\documentclass[11pt]{amsart}
\usepackage{parskip}
\usepackage{amsfonts,amsmath,amssymb,bbm}
\usepackage{dsfont}
\usepackage{verbatim,mathrsfs}
\usepackage{stmaryrd}
 \usepackage{setspace}
\usepackage{color}
\usepackage[dvipsnames]{xcolor}
\usepackage{hyperref}
\usepackage{enumerate}
\usepackage{pdfsync}
\usepackage{color}
\usepackage{etoolbox}
\makeatletter
\renewcommand\subsection{\@startsection{subsection}{2}%
	\z@{.5\linespacing\@plus.7\linespacing}{-.5em}%
	{\normalfont\bfseries}}
\makeatother
\makeatletter
\renewcommand\subsubsection{\@startsection{subsubsection}{3}%
	\z@{.5\linespacing\@plus.7\linespacing}{-.5em}%
	{\sffamily\bfseries}}
\makeatother
\numberwithin{equation}{section}
\newtheorem{theorem}{Theorem}[section]
\newtheorem{proposition}[theorem]{Proposition}
\newtheorem{lemma}[theorem]{Lemma}
\newtheorem{corollary}[theorem]{Corollary}
\theoremstyle{definition}
\newtheorem{definition}[theorem]{Definition}

\theoremstyle{remark}
\newtheorem{remark}[theorem]{Remark}
\usepackage{color}

\newcommand{\spnu}[2]{\langlerangle{#1}_{#2}}

\newcommand{\esD}{\mathbf{D}}

\newcommand{\esC}{\mathbf{C}}

\newcommand{\ZZ}{\mathbb{Z}}

\renewcommand{\limsup}{\mathop{\overline{\lim}}\limits}

\renewcommand{\r}{{\sigma^2_1}}%{\mathfrak{r}}

\newcommand{\lin}{\llbracket}
\newcommand{\rin}{\rrbracket}

\newcommand{\lb}{\left (}
\newcommand{\rb}{\right )}

\newcommand{\langlerangle}[1]{\left<#1\right>}

 \renewcommand{\L}{\mathbf{L}}

    \def\P{{\mathbb P}}

    \def\E{{\mathbb E}}
    \def\N{{\mathbb N}}
    \def\re{\mathtt{Re}}
    \def\im{\mathtt{Im}}
    \def\i{\textnormal {i}}
    
    \def\R{{\mathbb R}}

\newcommand{\eab}{{\nu_\beta}}

\newcommand{\meKp}{{\mathbf{n}_{\phi}}}
\newcommand{\meK}{{\mathbf{n}}}

\newcommand{\cont}[1]{{#1}}

\newcommand{\Mphi}{\mathds{I}_{\phi}}
\newcommand{\Mphiad}{\widehat{\mathds{I}}_{\phi}}

\newcommand{\eigLc}{\mathcal{P}^{{\phi}}}
\newcommand{\eigdLc}{\mathrm{V}^{{\phi}}}

\newcommand{\eigLd}{\mathsf{P}^{\phi}_{\!k}}
\newcommand{\eigLdd}{\mathsf{V}^{\phi}}
\newcommand{\eigL}{\mathsf{P}_{\!k}}

\newcommand{\SLp}{\mathds{K}^{\phi}}
\newcommand{\SL}{\mathds{K}}

\newcommand{\SLpd}{\widehat{\mathds{K}}^{\phi}}

\newcommand{\SLcp}{\cont{K}^{\phi}}
\newcommand{\SLcpd}{\widehat{\cont{K}}^{\phi}}

\newcommand{\Ladp}{\widehat{\Lambda}_{\phi}}

\newcommand{\Lnu}{{\L^{2}(\nu_{\phi})}}
\newcommand{\lnu}{{\ell}^{2}(\meKp)}

\newcommand{\poly}{\mathcal{P}}

\newcommand{\Be}{\mathbf{B}}

\newcommand{\cP}{{\mathcal P}}

\newcommand{\PP}{\mathbb{P}}

\newcommand{\RR}{\mathbb{R}}

\renewcommand{\ZZ}{\mathbb{Z}}

\newcommand{\iy}{\infty}

\newcommand{\wi}{\widehat}

\newcommand{\Ent}{\mathrm{Ent}}

\newcommand{\bq}{\begin{eqnarray*}}
\newcommand{\bqn}[1]{\begin{eqnarray}\hlabel{#1}}
\newcommand{\eq}{\end{eqnarray*}}
\newcommand{\eqn}{\end{eqnarray}}

\renewcommand{\lin}{\llbracket}
\renewcommand{\rin}{\rrbracket}

\newcommand{\ttsim}{\raise.17ex\hbox{$\scriptstyle\mathtt{\sim}$}}
\newcommand{\esP}{\mathbf{P}_{\!\mathfrak{e}}}

\newcommand{\dBg}{{\mathds G}}

\newcommand{\bb}{\mathbb}
\newcommand{\cc}{\mathcal}

\newcommand{\tred}{\textcolor{red}}

\newcommand{\nab}{\nabla}

\newcommand{\ovl}{\overline}

\newcommand{\ep}{\epsilon}
\newcommand{\var}{\operatorname{Var}}
\newcommand{\ent}{\operatorname{Ent}}

\newcommand{\mf}{\mathfrak}
\newcommand{\bbm}{\mathbbm}
\newcommand{\ttt}{\mathtt}
\newcommand{\bmrm}[2][2]{\mathbf{#2}^{#1}}

\newcommand{\Span}{\operatorname{Span}}
\newcommand{\Range}{\operatorname{Range}}

\newcommand{\pp}{\partial_+}
\newcommand{\pn}{\partial_-}
\newcommand{\eQ}{\mathds{Q}}

\newcommand{\dX}{\mathds{X}}
\newcommand{\lf}{\lfloor}
\newcommand{\rf}{\rfloor}
\newcommand{\f}{\mathsf{f}}
\newcommand{\g}{\mathsf{g}}
\newcommand{\eK}{\mathds{K}}
\newcommand{\eKad}{\wi{\mathds{K}}}
\newcommand{\eL}{\mathds{L}}

\newcommand{\dD}{\mathds{D}}
\newcommand{\BBB}{\mathscr{B}}

\newcommand{\lnp}{\ell^2(\meKp)}
\newcommand{\eLm}{\eL^{\phi_\mf{m}}}

\newcommand{\coeig}{\mathsf{V}}
\newcommand{\p}{\mathsf{p}}
\newcommand{\lnt}{\frac{1}{2}\log\left(1+\sigma^{-2}\right)}
\newcommand{\rpn}{(1+\sigma^{-2})^{-\frac{k}{2}}}
\newcommand{\rpnn}{(1+\sigma^{-1}_1)^{-\frac{k}{2}}}
\newcommand{\rpp}{(1+\sigma^{-2})^{\frac{k}{2}}}
\newcommand{\rppp}{(1+\sigma^{-1}_1)^{\frac{k}{2}}}
\newcommand{\rs}{{\sigma_1}}
\newcommand{\Spec}{\mathrm{Spec}}

\newcommand{\hlabel}{\phantomsection\label}
\renewcommand{\tred}{}

\addtocontents{toc}{\setcounter{tocdepth}{1}}

\title{Discrete self-similar and ergodic Markov chains}

\author{Laurent Miclo$^\dag$}\thanks{$^\dag$ Funding from the  grant ANR-17-EURE-0010 is acknowledged.}
\address{Toulouse School of Economics\\
Université de Toulouse and CNRS, France }
\email{miclo@math.cnrs.fr}

\author{Pierre Patie}\thanks{The authors are very grateful to the anonymous referees for the careful reading of the manuscript and their constructive comments.}
\address{School of Operations Research and Information Engineering, Cornell University, Ithaca, NY 14853.}
\email{	pp396@cornell.edu}

\author{Rohan Sarkar}
\address{School of Operations Research and Information Engineering, Cornell University, Ithaca, NY 14853.}
\email{rs2466@cornell.edu}

\usepackage{hyperref}
\begin{document}
\maketitle
\begin{abstract}
The first aim of this paper is to  introduce a class of  Markov chains on $\bb{Z}_+$ which are discrete self-similar in the sense that their semigroups satisfy an  invariance property expressed in terms of a discrete random dilation operator. After showing that this latter property requires the chains to be upward skip-free, we first establish a gateway relation, a concept introduced in \cite{miclo_patie},  between the semigroup of such chains and the one of  spectrally negative self-similar Markov processes on $\bb{R}_+$. As a by-product, we prove that each of these Markov chains, after an appropriate scaling, converge  in the Skorohod metric, to the associated self-similar Markov process. By a linear perturbation of the generator of these Markov chains, we obtain a class of ergodic Markov chains, which are non-reversible.  By means of intertwining and interweaving relations, where the latter was recently introduced in \cite{miclo_patie_2},  we derive several deep analytical properties of such ergodic chains including the description of the spectrum, the spectral expansion of their semigroups, the study of their convergence to equilibrium in the $\Phi$-entropy sense as well as their hypercontractivity property.
\end{abstract}
%\maketitle
\vspace{1cm}

\tableofcontents
\newpage

\section{Introduction}
Self-similar processes are ubiquitous in the theory of Markov processes and they have been studied intensively over the last three decades, both from theoretical and applied perspectives. A Markov process on $\bb{R}_+$ is called self-similar of index $1$ if for all $\alpha>0$, one has the following commutation type relation
\begin{align}\hlabel{eq:scale}
Q_td_\alpha=d_\alpha Q_{\alpha t}
\end{align}
 where $Q=(Q_t)_{t\geq0}$ is the semigroup associated with the process and $d_\alpha f(x)=f(\alpha x)$ is the dilation operator, that satisfies the semigroup property $d_\alpha d_\beta=d_{\alpha \beta}$ for all $\alpha,\beta >0$. Motivated by limit theorems, Lamperti \cite{Lamperti-72} obtained  a complete characterization of these processes.

  In this paper, we first aim at introducing  continuous-time Markov processes with state space the set of all nonnegative integers that also enjoy a scaling type property. Naturally, one cannot expect \eqref{eq:scale} to hold in this setting, because the set of integers is not stable by the dilation operators as defined above. However, in \cite{miclo_patie}, the authors introduced the following signed Binomial kernel defined by
 \[\dD_\alpha \f(n)=\sum_{k=0}^n\dbinom{n}{k}\alpha^k(1-\alpha)^{n-k}\f(k)\] which resembles the dilation operator through the multiplicative semigroup property $\dD_{\alpha \beta}=\dD_\alpha\dD_\beta$ for all $\alpha,\beta >0$, which will be proved in Proposition \ref{prop:D_alpha} below. Furthermore, they showed that the linear birth-death Markov chain, see Remark \ref{rem:lbd} below for definition, satisfies the following commutation type relation \[\eQ_t\dD_\alpha=\dD_\alpha\eQ_{\alpha t}\] where $\eQ$ is the associated semigroup. Motivated from this result, we introduce a class of continuous-time Markov chains on $\ZZ_+$ that satisfy the scaling property as above and are upward skip-free, that is, at any instant the Markov chains do not jump more than one step above and name them  \emph{discrete self-similar Markov chains}, see Definition~\ref{def:disc_selfsim}. This class of  Markov chains, to the best of our knowledge, have not been identified before. Moreover, we want to understand their connections  with self-similar Markov processes. To this end, we resort to  intertwining relationship between Markov processes. More specifically, for two Markov semigroups $P$ and $Q$, we say that they are intertwined if, for all $t\ge 0$, \[P_t\Lambda =\Lambda Q_t \] for some linear operator $\Lambda$. Note that when the underlying processes have different state spaces, one lattice and the other one continuous, we use the terminology gateway relation, coined in \cite{miclo_patie}, to emphasize the unexpected two-sided connection between the two worlds. The term duality is also used in a fast growing and fascinating literature on this topic related to differential operators arising in statistical mechanics, see e.g.~\cite{Assi, CF, Gr, Redig,JK} and references therein. More generally,  the concept of intertwining relation goes back to Dynkin \cite{dynkin:1965} who used it to construct new Markov semigroups from a reference one. These ideas were extended by Rogers and Pitman in \cite{Pitman-Rogers-81},
leading to the characterization of Markov functions; that is, measurable maps that preserve the Markov property. With the help of the intertwining relationship, we prove the Feller property of the discrete self-similar Markov chains, see Theorem~\ref{thm:gateway_main}, and obtain the spectrally negative self-similar Markov processes as the scaling limit of these Markov chains, see Theorem~\ref{thm:gateway_main}\eqref{it:scaling_lim}. The use of intertwining relations to prove limit theorems is not new and, in fact, a general framework was built up  by Borodin and Olshanski \cite{borodin}, where they apply it to construct a class of Markov chains on the Thoma cone. Unfortunately, their strategy is not applicable in our situation because their conditions are too stringent for us, namely the set of finitely supported functions are not invariant with respect to the discrete self-similar Markov semigroups. Nonetheless, still resorting to the intertwining relation, we are able to derive explicit formulas for the moments of these Markov chains and we identify their scaling limits by the method of moments. We emphasize that there are many instances of the appearance of positive self-similar Markov processes as the scaling limits of  models, such as coalescence-fragmentation processes, see Bertoin \cite{bertoin_fc}, random planar maps, see Le Gall and Miermont \cite{miermont}. We also mention the recent paper by Bertoin and Kortchemski \cite{bertoin_igor} where the authors introduce a  class of discrete-time Markov chains whose appropriate scaling limits are positive self-similar Markov processes.
It appears that our work offers another class of Markov chains in the domain of attraction of such self-similar  Markov processes, with the additional surprising feature that the connection between the two objects goes, thanks to the gateway relation, in both directions.

We proceed by introducing another  class of ergodic Markov chains which are obtained by a linear first order perturbation of  the generators of the discrete self-similar Markov chains. We name them  \textit{skip-free Laguerre} chains. The motivation behind this comes from the fact that their continuous analogue are the generalized Laguerre processes, studied in \cite{Patie-Savov-GeL}, which are also constructed by perturbation of the generator of self-similar processes by a linear convection term, that is a  first order differential operator with a linear coefficient.  We show that they generate a class of Feller semigroups of ergodic Markov chains  which intertwine  with the class of the generalized Laguerre semigroups.  Using this connection, we develop the spectral theory, including the spectrum and the eigenvalues expansions, in the Hilbert space $\ell^2$ of nonnegative integers weighted with the invariant distributions $\meKp$ of the semigroups of these non-reversible chains.
As by-product, and under some mild conditions, we prove compactness and also obtain a hypercoercivity estimate for the $\ell^2(\meKp)$ convergence to equilibrium,  which is given explicitly as a perturbed spectral gap inequality. This part involves a deep theory of non-self-adjoint operators as developed in \cite{Patie-Savov-GeL}, see Section~\ref{ss:hilbert} for more details.

We continue our analysis of these skip-free Laguerre semigroups by investigating the entropy decay to equilibrium as well as the hypercontractivity property. For self-adjoint Markov semigroups, these two phenomena are equivalent to the (modified) log-Sobolev inequalities. Unfortunately, in our context, this relation fails due to the non-self-adjointness of the semigroups. However, resorting to the idea of interweaving relation,  introduced recently in \cite{miclo_patie_2}, we relate the skip-free Laguerre semigroups with the self-adjoint diffusion Laguerre semigroups and deduce, up to some universal random time, both the entropy decay and the hypercontractivity. Finally, showing that this random time is infinitely divisible, we develop a thorough analysis  of the skip-free Laguerre semigroups subordinated with the associated subordinator, which generate a class of ergodic Markov chains with two-sided jumps, for which all the results described above are obtained explicitly.

The remaining part of the paper is organized as follows. Most of the frequently used notations are defined in Section~\ref{s:notations} while Section~\ref{s:main} contains all the main results of the paper. We provide some examples in Section~\ref{s:exm} and Section~\ref{s:proofs} is devoted to the proofs of the main results. Some aspects of spectral theory for non-self-adjoint operators have been reviewed in Subsection~\ref{ss:hilbert} and the results related to interweaving relations have been proved in Subsection~\ref{ss:interweaving}.
\begin{comment}
\textcolor{red}{structure}
\begin{enumerate}
  \item Intertwining between self-similar
  \begin{enumerate}
    \item 2 proofs : one direct where everything is computed, and the second one based on intertwining  mixture idea
    \item Discrete scaling property
     \item Moments (discrete version of Bertoin and Yor paper)
      \item First passage time above
  \end{enumerate}
  \item Intertwining ergodic Laguerre
  \item Spectral decomposition and convergence to equilibrium
\end{enumerate}
\end{comment}
\subsection{Notations and Preliminaries}\hlabel{s:notations}
For any locally compact topological space $E$ we write $\esC(E)$, $\esC_b(E)$, $\esC_c(E)$ and $\esC_0(E)$ to denote the class of continuous functions \tred{(the set of all functions when $E=\mathbb{Z}_+$)}, class of all bounded continuous functions, class of all compactly supported continuous functions and class of all continuous functions vanishing at infinity on $E$ respectively. %\textcolor{red}{Note that when $E=\bb{Z}_+$, $\esC(E),\esC_b(E),\esC_c(E)$ and $\esC_0(E)$ are simply the set of all functions, set of all bounded functions, set of all finitely supported functions and the set of all functions vanishing at infinity.
In addition, when $E=\bb{R} \text{ or } \bb{R}_+$, we write $\esC^\infty_b(E)$ to denote the class of all bounded smooth functions with bounded derivatives on $E$.

Next, for any nonnegative sigma-finite measure ${\mu}$ on $\bb{R}_+$ and $p\in [1,\infty]$, $\mathbf{L}^p({\mu})$ denotes the $L^p$ space with weight ${\mu}$. When $p=2$, the corresponding Hilbert space is endowed with the inner product denoted by $\langle f,g \rangle_{\mu}=\int_{\RR_+} f(x)\overline{g(x)}{\mu}(dx)$. When ${\mu}$ is the $\text{Lebesgue measure}$, we simply write $\mathbf L^2({\mu})=\mathbf L^2(\bb{R}_+)$ associated with the inner product $\langle\cdot,\cdot\rangle$. If the underlying space is the set of all integers $\ZZ_+$, then for any nonnegative discrete measure $\mathbf{m}$ on $\ZZ_+$, we write $\ell^p(\mathbf{m})$ to denote the weighted $\ell^p$ space on $\ZZ_+$ and for $p=2$, the inner product is written as $\langle \f,\g \rangle_{\mathbf{m}}=\sum_{n\in\ZZ_+}\f(n)\g(n)\mathbf{m}(n)$. When $\mathbf{m}$ is the counting measure, we use the notation $\ell^2(\ZZ_+)=\ell^2(\mathbf{m})$. For any measurable function $f\ge 0$ or $f\in\bmrm[1]{L}(E,{\mu})$, we write ${\mu} f=\int_E fd{\mu}$.

 For any two Banach spaces $\mathbf B_1, \mathbf B_2$, $\mathscr{B}(\mathbf B_1,\mathbf B_2)$ denotes the set of all bounded linear operators defined from $\mathbf B_1$ to $\mathbf B_2$. Finally, for any operator $A$ (possibly unbounded) defined on some Banach space, $\mathbf{D}(A)$ denotes the domain of the operator and we represent the operator as $(A,\mathbf{D}(A))$ and in case of Hilbert spaces, we denote the adjoint of $A$ by $\wi{A}$.

 We denote the complex plane by $\bb{C}$ and for any $z\in\bb{C}$, $\re(z),\im(z)$ denote the real and imaginary part of $z$ respectively. Next, for any $S\subset\bb{R}$, we write $\bb{C}_S=\{z\in\bb{C};\: \re(z)\in S\}$. In particular, when $S=\bb{R}_+$ (resp.~$\bb{R}_-$), we simply write $\bb{C}_+$ (resp.~$\bb{C}_-$).

For two functions $f,g$ defined on the real line, we use the following notation.
\begin{eqnarray*}
	f&\asymp &g  \text{ means that $\exists c>0$ such that, for all $x$,  $c^{-1}\le\frac{f(x)}{g(x)}\le c$} \\
	f&\overset{a}{\sim}& g \text{ means that $\lim_{x\to a}\frac{f(x)}{g(x)}=1$ for some $a\in[0,\infty]$} \\
	f(x)&\overset{a}{=}&\mathrm{O}(g(x)) \text{ means that $\limsup_{x\to a}\left|\frac{f(x)}{g(x)}\right|<\infty$}\\
	f(x)&\overset{a}{=}&\mathrm{o}(g(x)) \text{ means that $\lim_{x\to a}\frac{f(x)}{g(x)}=0$}.
\end{eqnarray*}
%\subsection{Mellin transform}
\section{Main Results}\hlabel{s:main}
\subsection{Discrete dilation and discrete self-similar Markov chains} We start by introducing a transformation on $\esC(\ZZ_+)$, which we name  the \emph{discrete dilation operator}.
For any $\alpha>0$ and $\f\in\esC(\ZZ_+)$, we define
\begin{align}\hlabel{eq:disc_dilation}
	\dD_\alpha\mathsf{f}(n)=\sum_{r=0}^n \dbinom{n}{r}\alpha^r(1-\alpha)^{n-r}\mathsf{f}(r).
\end{align}
It should be noted that $\dD_\alpha$ is well defined on $\esC(\ZZ_+)$ for all $\alpha\ge 0$ and it is a Markov kernel when $\alpha\in [0,1]$. When $\alpha>1$, $\dD_\alpha\f$ may not be bounded even if $\f$ is bounded. For instance, taking $\f(n)=(-1)^n$, for any $n\in\ZZ_+$, we have $|\dD_\alpha\f(n)|=(2\alpha-1)^n$, which grows exponentially with respect to $n$. The operator $\dD$ shares the multiplicative semigroup property with the dilation operator, that is, for all $\alpha,\beta>0$, we have $\dD_{\alpha\beta}=\dD_\alpha\dD_\beta$, see Proposition~\ref{prop:D_alpha} below. Next, we introduce the discrete self-similar Markov chains which are defined in terms of the operator $\dD_\alpha$.
\begin{definition}\hlabel{def:disc_selfsim}
	We say that the semigroup $\mathds{Q}=(\mathds{Q}_t)_{t\geq0}$ of a continuous-time Markov chain $\mathds{X}$  with state space $\ZZ_+$ is \emph{discrete self-similar} if for all $t\ge 0, \alpha\in [0,1]$,  the following identity
	\begin{align}\hlabel{eq:disc_selfsim}
		\mathds{Q}_t\dD_\alpha=\dD_\alpha\mathds{Q}_{\alpha t}
	\end{align}
holds on  $\esC_b(\ZZ_+)$.

	In terms of the law of the Markov chain $\mathds X=(\mathds X(t,n), n\in \ZZ_+)_{t\geq0}$, where $ \mathds X(t,n)$ means that it is issued from $n$, the discrete self-similarity can be interpreted by the following identity in distribution, for any $\alpha\in [0,1]$, $t\geq0$ and $n\in\ZZ_+$,
	\begin{equation}\hlabel{eq:disc_selfsim-M}
		\ttt B(\mathds X(t,n),\alpha) \stackrel{(d)}{=}\mathds X(\alpha t,\ttt B(n,\alpha))
	\end{equation}
	where $\ttt B(n,\alpha)$ is a Binomial random variable with parameter $n$ and $\alpha$, and $\mathds X(t,\ttt B(n,\alpha))$) is the chain at time $t$ with initial law the one of $\ttt B(n,\alpha)$.
\end{definition}
Next, we consider the class of triplets $(m,\sigma^2,\Pi)$ such that $m,\sigma^2\ge 0$ and $\Pi$ is a non-negative measure on $\bb{R}_+$ that satisfies
\begin{align}\hlabel{eq:levy}
	\int_0^\infty (y\wedge y^2)\Pi(dy)<\infty,
\end{align}
that is $\Pi$ is a L\'evy measure with a finite first moment away from $0$.
To each of these triplets, we associate  the so-called Bernstein function defined as
\begin{align}\hlabel{eq:bernstein_def}
	\phi(u)=m+\sigma^2u+\int_0^\infty(1-e^{-uy})\ovl{\Pi}(y)dy
\end{align}
where $\ovl{\Pi}(y)=\Pi(y,\infty)$ is the tail of the measure $\Pi$. Let $\Be$ denote the class of all functions of the form \eqref{eq:bernstein_def}.

We are now ready to introduce a class of discrete operators on $\ZZ_+$, which is the central object of this paper. For any $\phi\in\Be$ associated with the triplet $(m,\sigma^2,\Pi)$ and $\f\in\esC_c(\ZZ_+)$, we define
\begin{align}\hlabel{eq:skip_free_gen}
	\dBg_\phi\mathsf{f}(n)=\sigma^2 n(\partial_++\partial_-)\mathsf{f}(n)+(m+\sigma^2)\pp\mathsf{f}(n)+\dBg_\Pi\mathsf{f}(n)
\end{align}
where $\partial_\pm\f(n)=\f(n\pm 1)-\f(n)$ for all $n\in\ZZ_+$ and
\begin{align}
\dBg_\Pi\f(n)=\frac{1}{n+1}\int_0^\infty\left[\dD_{e^{-y}}\f(n+1)-\f(n+1)+y(n+1)\partial_+\f(n)\right]\Pi(dy).
\end{align}
We are now ready to state our first main result.
\begin{theorem}\hlabel{feller_skip}
 The operator $(\dBg_\phi,\mathbf{C}_c(\ZZ_+))$ generates a Feller Markov chain on $\ZZ_+$, denoted by $\dX_\phi=(\mathds X_{\phi}(t,n), n\in \ZZ_+)_{t\geq0}$ which is  self-similar, and $\mathbf{C}_c(\ZZ_+)$ serves as a core for $\dBg_\phi$.
\end{theorem}
This theorem is proved in Section~\ref{ss:gen_disc}.
\begin{remark}\hlabel{rem:lbd}
When $\Pi\equiv 0$, $\dX_\phi$ is the reversible linear birth-death chain with invariant measure \[\frac{\Gamma(n+m+1)}{\Gamma(n+1)},~n\in\ZZ_+.\] For a detailed account on such Markov chains, we refer to \cite{miclo_patie}.
\end{remark}

\begin{remark}
	In \eqref{eq:disc_selfsim}, we restrict $\alpha \in [0,1]$ as $\dD_\alpha$ is, in this case, a Markov kernel. However,
	since $\dD_\alpha \mathsf{p}_k(n)=\alpha^k \mathsf{p}_k(n)$, for all $\alpha>0$ and $k,n \in \ZZ_+$, where $\mathsf{p}_k$ is defined in \eqref{eq:p_z} below, Theorem \ref{thm:moments}, also below, yields  that
	for all $\phi\in\Be$ and $t\geq0$, $\eQ_{t}^\phi\dD_\alpha \mathsf{p}_k(n) =\dD_\alpha\eQ_{\alpha t}^\phi \mathsf{p}_k(n)$, where $\eQ^\phi$ is the discrete self-similar semigroup generated by $\dBg_\phi$. This  reveals that the discrete self-similarity property also holds in a more general framework than the one given in \eqref{eq:disc_selfsim}.
\end{remark}

A continuous-time Markov chain is called upward skip-free if it does not jump more than one step above at any instant, that is, for any $n\in\ZZ_+$ and $l\ge n+2$, $\dBg(n,l)=0$ where $\dBg$ is the generator of the Markov chain. It can be easily shown that the discrete self-similar Markov chain $\dX_\phi$ with generator $\dBg_\phi$ is upward skip-free, see \eqref{eq:jump} below. In the next theorem we show the converse claim that is any discrete self-similar Markov chains must be upward skip-free.
\begin{theorem}\hlabel{thm:skip-free}
	Let $\dX$ be any continuous-time discrete self-similar Markov chain on $\ZZ_+$. Then $\dX$ is upward skip-free.
\end{theorem}
This theorem is proved in Section~\ref{ss:skip-free-pf}.

\subsection{Connections with self-similar Markov processes: gateway relation and scaling limit}\hlabel{sss:disc-cont-ssmp}
Self-similar Markov processes on the positive real line are well studied as they appear as the weak limits of various Markov processes, see Lamperti \cite{Lamperti-62}. When these processes are spectrally negative, that is, they do not have any positive jumps, and with $0$ as an entrance-non-exit boundary, Lamperti \cite{Lamperti-72} showed that they are in bijection  with the subset of Bernstein functions $\Be$ defined in \eqref{eq:bernstein_def} and moreover, the generator of these processes are of the form
\begin{align}\hlabel{eq:ssmp_spec}
G_\phi f(x)=&\sigma^2 x f''(x)+(m+\sigma^2) f'(x)\\+&\frac{1}{x}\int_0^\infty[d_{e^{-y}}f(x)-f(x)+yxf'(x)]\Pi(dy)
\end{align}
where $\phi$ is defined in terms of the triplet $(m,\sigma^2,\Pi)$, see \eqref{eq:bernstein_def} and $f\in\esC^\infty_c(\bb{R}_+)$.
The careful reader will have noticed that the operator $\dBg_\phi$ in \eqref{eq:skip_free_gen} is the discrete analogue of the operator $G_\phi$, revealing that the former is a natural approximation of the latter. However, we provide below a deeper connection between these class of Markov processes (operators) by establishing a gateway relation between their semigroups, a concept introduced in \cite{miclo_patie_2},  meaning that the connection goes in both directions.  As a by-product, we show that discrete self-similar Markov chains, after scaling appropriately, converge to the self-similar Markov processes in the Skorohod's $J_1$-topology.
\begin{theorem}\hlabel{thm:gateway_main}
\begin{enumerate}
	\item \hlabel{it:gateway} \textbf{Gateway relation.} For any $\phi\in\Be$, let $Q^\phi$ and $\eQ^\phi$ denote the Feller semigroups generated by $G_\phi$ and $\dBg_\phi$ respectively. Then, for any $\f\in \mathbf{C}_0(\ZZ_+)$ and  $t\ge 0$,
	\begin{align}\hlabel{eq:P_intertwiningMT}
		Q^\phi_t \Lambda \f=\Lambda \eQ^\phi_t \f
	\end{align}
	where $\Lambda\mathsf f(x)=\bb{E}[\mathsf f(\mathrm{Pois}(x))]$, $\mathrm{Pois}(x)$ being a Poisson random variable with parameter $x>0$.

	\item \hlabel{it:scaling_lim} \textbf{Scaling limit.}
	For any $\phi\in\Be$, let $\dX_\phi$ (resp.~$X_\phi=(X_\phi(t,x))_{t\geq0}$) be the discrete self-similar Markov chain (resp.~the positive self-similar Markov process issued from $x$), then, for all $x>0$,
	\begin{align} \label{eq:SKc}
		\left(\frac{1}{n}\dX_\phi(nt,\lfloor nx\rfloor)\right)_{t\ge 0}\longrightarrow\left(X_\phi(t,x)\right)_{t\ge 0}
	\end{align}
	in Skorohod's $J_1$-topology.
\end{enumerate}
\end{theorem}
The intertwining relation  \eqref{eq:P_intertwiningMT} is proved in Proposition~\ref{thm:gateway}\eqref{it:1} and the scaling limit \eqref{eq:SKc} is proved in Section~\ref{ss:scaling_lim_pf}.
\begin{remark}
As mentioned to us by an anonymous referee, the gateway relationship \eqref{eq:P_intertwiningMT} has the following  neat probabilistic interpretation, using the notation of item \eqref{it:scaling_lim} above,
\begin{equation}\label{eq:co}
		 \mathds X_\phi(t,\mathrm{Pois}(x)) \stackrel{(d)}{=}\mathrm{Pois}(X_\phi(t,x))
	\end{equation}
which is valid for any $t,x>0$. Using the self-similarity property of $X_\phi$, this identity yields, for any fixed $t,x>0$  and large  integer $n$ (but not for $\lfloor nx \rfloor $),  $\frac{1}{n}\mathds X_\phi(nt,\mathrm{Pois}(nx)) \stackrel{(d)}{=}\frac{1}{n} \mathrm{Pois}(X_\phi(nt,nx))\stackrel{(d)}{=} \frac{1}{n}\mathrm{Pois}(nX_\phi(t,x))\rightarrow X_\phi(t,x)$ in distribution.  Moreover, the identity \eqref{eq:co} boils down when $x$ tends to $0$ to
\begin{equation*}%\hlabel{eq:disc_selfsim-M-En}
		\mathds X_\phi(t,0) \stackrel{(d)}{=}\mathrm{Pois}(X_\phi(t,0))
	\end{equation*}
where $X_\phi(t,0)$ stands for the entrance law of $X_\phi$ which is known to exist as $m\geq0$, see e.g.~\cite{Patie-Savov-GeL}.
\end{remark}

\subsection{Discrete Laguerre chains from discrete self-similar Markov chains}
Let us now consider a perturbation of the discrete self-similar Markov chains, that is, we introduce a new family of discrete operators on $\esC_c(\ZZ_+)$ defined by
\begin{align}
	\eL_{\phi}\mathsf{f}(n)=\dBg_\phi\mathsf{f}(n)+n\partial_-\f(n)
\end{align}
where $\phi\in\Be$ and $\dBg_\phi$ is defined in \eqref{eq:skip_free_gen}. Alternatively, the operator $\eL_{\phi}$  can be represented, for any $\f\in\esC_c(\ZZ_+)$, as follows
\begin{align}
	\eL_{\phi}\f(n)=\sum_{l=0}^{n+1}\eL_{\phi}(n,l)\f(l)
\end{align}
where \begin{align}\hlabel{eq:Lphi}
	\eL_{\phi}(n,l)=\begin{cases}
	\dBg_\phi(n,l) & \mbox{if $l\neq n, n-1$} \\
	\dBg_\phi(n,n-1)+n & \mbox{if $l=n-1$} \\
	\dBg_\phi(n,n)-n & \mbox{if $l=n$}
\end{cases}
\end{align}
\tred{with $\dBg_\phi(n,l)=\dBg_\phi\delta_l(n)$ and $\delta_l(n)=\mathbbm{1}_{\{l=n\}}$.}

\begin{theorem}\hlabel{thm:generation_lag}
\begin{enumerate}
\item\hlabel{it:lag_1} For any $\phi\in\Be$, the operator $(\eL_{\phi},\esC_c(\ZZ_+))$ generates a Feller Markov semigroup on $\esC_0(\ZZ_+)$, which we denote by $\eK^\phi$.
\item\hlabel{it:P-K} We have, for any $\f\in \mathbf{C}_0(\ZZ_+)$ and  $t\ge 0$,
	\begin{align}\hlabel{eq:P_K}
		\eK^\phi_t \f= \eQ^\phi_{e^{t}-1} \dD_{e^{-t}} \f.
	\end{align}
\item \hlabel{it:invariant} The semigroup $\eK^\phi$ has a unique invariant distribution denoted by $\meKp$ and $\meKp(n)>0$ for all $n\in\ZZ_+$. Moreover, $\meKp$ has moments of all orders and it is moment determinate.
\item\hlabel{it:self-adj} Finally, the semigroup $\eK^\phi$ is self-adjoint in $\ell^2(\meKp)$ if and only if $\phi(u)=m+\sigma^2 u$ for some $m,\sigma^2\ge 0$.
\end{enumerate}
\end{theorem}
We have omitted the proof of the item \eqref{it:lag_1} since it can be obtained by following a line of reasoning similar to  the proof of Theorem~\ref{feller_skip} from the claims given in Proposition \ref{prop:int_laguerre}. Item \eqref{it:P-K} is proved after this latter Proposition. The properties of the invariant distribution in item~\eqref{it:invariant} are proved in Proposition~\ref{prop:int_laguerre}\eqref{it:inv} and Proposition~\ref{prop:n_phi}. Item~\eqref{it:self-adj} is proved in Proposition~\ref{prop:int_laguerre}\eqref{it:self_adj}.

\begin{remark}
 In Proposition \ref{prop:int_laguerre} and \ref{prop:n_phi}, we provide additional properties, including several representations, of the invariant measure $\meKp$.
\end{remark}

We name the Markov semigroup $\eK^{\phi}$ (resp.~the Markov chain) the \emph{skip-free Laguerre semigroup} (resp.~\emph{skip-free Laguerre  chain}). This is motivated by the following observation. The operator $\eL_{\phi}$ can be viewed as the discrete analogue of the generalized Laguerre operator on $\RR_+$, studied in \cite{Patie-Savov-GeL}, and defined by
\begin{align}\hlabel{eq:cont_lag_gen}
	L_\phi f(x)=&G_\phi f(x)-xf'(x) \nonumber \\
	=&\sigma^2 xf''(x)+\lb m+\sigma^2-x\rb f'(x) \\+&  \frac{1}{x}\int^{\infty}_{0} \left(d_{e^{-y}}f(x)-f(x)+yxf'(x)\right) \Pi(dy)
\end{align}
where $G_\phi$ is defined in \eqref{eq:ssmp_spec} and $(\sigma,\beta,\Pi)$ is the characteristic triplet of  $\phi$.

We now aim to derive the spectral properties, convergence to the equilibrium and hypercontractivity phenomenon of $\eK^{\phi}$.

\subsection{Spectral expansion and the spectrum of the skip-free Laguerre semigroups} Since the semigroup $\eK^\phi$ has invariant distribution $\meKp$, we can extend it on the Hilbert space $\ell^2(\meKp)$. If $\phi$ is as in \eqref{eq:bernstein_def}, let $\sigma_1$ be defined as follows
	\begin{align}\hlabel{eq:tau}
	\sigma_1=\begin{cases}
		\sigma^{2} & \mbox{if $\sigma^2>0$} \\
		1 & \mbox{if $\sigma^2=0$}. \end{cases}
\end{align}
We now introduce a sequence of discrete (acting on $\ZZ_+)$ polynomials defined, for $k,n\in\ZZ_+$, by
\begin{align}\hlabel{eq:def_e}\rpnn\sum_{r=0}^{k}(-1)^r {k \choose r}  \frac{\mathsf{p}_r(n)}{W_{\phi}(r+1)}
\end{align}
where $W_\phi(k+1)=\prod_{r=1}^k\phi(r),\ W_\phi(1)=1$ and  $
	\mathsf{p}_r(n)=\frac{\Gamma(n+1)}{\Gamma(n+1-r)}$.
Since the invariant distribution $\meKp$ has finite moments of all order, see Theorem~\ref{thm:generation_lag}\eqref{it:invariant}, it is plain that,  for all $k\in\ZZ_+$, $\eigLd\in\ell^2(\meKp)$. Next, for  $k,n\in\ZZ_+$, we define
\begin{align}\hlabel{eq:coeig}
	\eigLdd_k(n)= \frac{\rppp}{\meKp(n)} \sum_{r=0}^{k\wedge n}(-1)^r \frac{(k+n-r)!}{(k-r)!(n-r)!r!}\meKp(k+n-r).
\end{align}
\begin{theorem}\hlabel{thm:eig_coeig}
	\begin{enumerate}
		\item\hlabel{it:spect}\textbf{Spectrum.} For any $\phi\in\Be, t\geq0$ and $k\in\ZZ_+$, $\eigLdd_k\in\ell^2(\meKp)$, and
		\begin{align*}
			\eK^\phi_t\eigLd=e^{-kt}\eigLd,~ \quad \wi{\eK}^\phi_t\eigLdd_k=e^{-kt}\eigLdd_k
		\end{align*}
where $\wi{\eK}^\phi_t$ is the $\ell^2(\meKp)$-adjoint of ${\eK}^\phi_t$.
		Hence, $\{e^{-kt}; k\in\bb{Z}_+\}\subseteq\Spec_p(\eK^\phi_t)\cap\Spec_p(\wi{\eK}^\phi_t)$, where for an operator $T$, $\Spec_p(T)$ denotes the point spectrum of $T$.
		\item \hlabel{it:biortho} \textbf{Biorthogonality.} $(\eigLd)_{k\ge 0}$ and $(\eigLdd_k)_{k\ge 0}$ are biorthogonal sequences in $\ell^2(\meKp)$, that is, for all $k,l\in\ZZ_+$,
		\[\left\langle\eigLd,\eigLdd_l\right\rangle_{\meKp}=\bbm{1}_{\{k=l\}}.\]
		
		\item\hlabel{it:spectral_exp} \textbf{Spectral expansion.} If $\sigma^2>0$, then, for all $f\in\ell^2(\meKp)$ and $t> \lnt$,
		\begin{equation}\hlabel{eq:spect_exp}
			\SLp_t \mathsf{f} = \sum_{k=0}^{\infty}e^{-kt} \left\langle \mathsf{f}, \eigLdd_k\right\rangle_\meKp\!\!\! \eigLd.
		\end{equation}
		\item\hlabel{it:spect_compact} \textbf{Compactness.} If $\sigma^2>0$, then, for all $t>\lnt$, $\eK^\phi_t$ is compact and, denoting by $\Spec(\eK^\phi_t)$ the spectrum of $\eK^\phi_t$, we have
		\begin{align*}
			\Spec(\eK^\phi_t)\setminus\{0\}=\Spec_p(\eK^\phi_t)=\{e^{-kt}; k\in\ZZ_+\}.
		\end{align*}
	\item \hlabel{it:lag_transition}  \textbf{Transition probabilities.} If $(\eK_t^\phi(\cdot,\cdot))_{t\geq0}$ denotes the transition probabilities of the skip-free Laguerre chain and $\sigma^2>0$, then, for all $t>\lnt$ and $n,l\in\ZZ_+$, we have
	\begin{align*}
		\eK^\phi_t(n,l)=\sum_{k=0}^\infty e^{-kt}\eigLd(n)\eigLdd_k(l)\meKp(l)
	\end{align*}
where the sum on the right-hand side of the above identity converges absolutely.
	\end{enumerate}
\end{theorem}
This theorem is proved in Section~\ref{ss:eig_coeig_pf}.
\begin{remark}
It should be noted that \eqref{it:spect} in the above theorem is different from the result in the case of generalized Laguerre semigroups on $\bb{R}_+$, their continuous analogue. Indeed, from \cite[Theorem 1.22(4)(d)]{Patie-Savov-GeL}, $e^{-kt}\in\Spec_p(\wi{K}^\phi_t)$ only if $k\in\ZZ_\phi$ (see \eqref{eq:Z_phi} for the definition of $\ZZ_\phi$) and $e^{-kt}\in\Spec_r(\wi{K}^\phi_t)$ if $k\notin\ZZ_\phi$, where  $\Spec_r(\eK^\phi_t)$ stands for the residual spectrum of $\eK^\phi_t$. However, for the discrete Laguerre semigroup $\eK^\phi$, $e^{-kt}\in\Spec_p(\wi{\eK}^\phi_t)$ for all $k\in \ZZ_+$.	
\end{remark}

\subsection{Convergence to equilibrium} In Theorem~\ref{thm:generation_lag}\eqref{it:invariant} we have seen that the non-self-adjoint skip-free Laguerre chains have an unique invariant distribution. In this section, we start by studying the rate of convergence to their invariant distributions via spectral gap inequality, which comes as a by-product of the spectral expansion obtained in the previous theorem. We proceed with explicit rate of convergence to equilibrium in the $\Phi$-entropy sense, which is a consequence of a more subtle relation with the self-adjoint birth-death Laguerre chain, namely an interweaving relation discussed in Section~\ref{ss:interweaving}. Before stating the result, let us introduce a few additional objects related to the Bernstein functions. For any $\phi\in\Be$ let us define
\begin{align}\hlabel{eq:dphi}
	\mathrm{d}_{\phi}=\min\{u\geq 0; \: \phi(-u)=-\infty, \phi(-u)=0 \} \in [0,\infty].
\end{align}
If $(m,\sigma^2,\Pi)$ is the triplet associated to $\phi$, let us write
\begin{align}\hlabel{eq:mphi}
	\mathrm{m}_\phi=\lim_{u\to\infty}\frac{\phi(u)-\sigma^2u}{\sigma^2}=\frac{m+\ovl{\ovl{\Pi}}(0)}{\sigma^2}
\end{align}
where $\ovl{\ovl{\Pi}}(0)=\int_0^\infty \Pi(y,\infty)dy$. The quantity $\mathrm{m}_\phi$ is finite whenever $\sigma^2>0$ and $\ovl{\ovl{\Pi}}(0)\in [0,\infty)$.

Next, for an open interval $I\subseteq\bb{R}$, we say that a function $\Phi:I\to\bb{R}$ is \emph{admissible} if
\begin{align}\hlabel{eq:admissible}
	\Phi\in\esC^4(I) \text{ with both $\Phi$ and $-\frac{1}{\Phi''}$ convex.}
\end{align}
Given an admissible function $\Phi$, and a probability measure $\mu$ on $\bb{R}$, we write for any $f:\bb{R}_+\to I$ with $f,\Phi(f)\in\bmrm[1]{L}(\mu)$
\begin{align}\hlabel{eq:phi_ent}
	\ent^\Phi_\mu(f)=\mu\Phi(f)-\Phi(\mu f)
\end{align}
for the so-called $\Phi$-entropy of $f$. When $\Phi(x)=x^2, I=\bb{R}$, \eqref{eq:phi_ent} is equal to $\var_\mu(f)$ and when $\Phi(x)=x\log x , I=\bb{R}_+$, \eqref{eq:phi_ent} yields the Boltzmann entropy of $f$ with respect to $\mu$. From Jensen's inequality it is plain that the $\Phi$-entropy is always nonnegative. We are now ready to state the following.

\begin{theorem}\hlabel{thm:spgap_ent} Let $\phi\in\Be$ be associated with the triplet $(m,\sigma^2,\Pi)$ such that $\sigma^2,\mathrm{d}_\phi>0$ and $\ovl{\ovl{\Pi}}(0)<\infty$. Then, the following holds.
\begin{enumerate}\item\textbf{Hypocoercive estimate.} \hlabel{thm:l0_norm} For all $\f\in\ell^2(\meKp)$ and $t\ge 0$, we have
	\begin{align}\hlabel{eq:spect_gap}
		\left\|\eK^\phi_t\f-\meKp\f\right\|_{\ell^2(\meKp)}\le\sqrt{\frac{(\mathrm{m}_\phi+1)(1+\sigma^2)}{\sigma^2(\mathrm{d}_\phi+1)}} e^{-t}\|\f-\meKp\f\|_{\ell^2(\meKp)}.
	\end{align}
\item \textbf{Entropy decay.} \hlabel{thm:entropy}
For all $\beta>\mathrm{m}_\phi$, $t\geq 0$ and $\f$ such that $\f,\Phi(\f)\in\ell^1(\meKp)$, we have
\begin{align}
	\ent^\Phi_{\meKp}\left(\eK^\phi_{t+\tau_\beta}\f\right)\le e^{-t}\ent^\Phi_{\meKp}(\f)
\end{align}
where, we recall that $\eK^\phi_{t+\tau_\beta}\f(n)=\E[\f(\dX_\phi(t+\tau_\beta,n)]$ and $\tau_\beta$ is an infinitely divisible positive random variable whose Laplace transform is given by
\begin{align}\hlabel{eq:tau_beta1}
	\int_0^\infty e^{-us}\bb{P}(\tau_\beta\in ds)=e^{-\phi_\beta(u)}, \ \ \ u>0,
\end{align}
with $\phi_\beta(u)=u\log\left(1+\sigma^{-2}\right)+\log\left(\frac{\Gamma(u+\beta+1)}{\Gamma(1+\beta)\Gamma(u+1)}\right)$.
\end{enumerate}
\end{theorem}
Item~\eqref{thm:l0_norm} of the above theorem is proved in Section~\ref{ss:l0_norm} and item~\eqref{thm:entropy} is proved in Section~\ref{ss:entropy_pf}.
\begin{remark}
The estimate in \eqref{thm:l0_norm} gives the hypocoercivity, in the sense of Villani \cite{Villani-09}, for the skip-free Laguerre semigroups. This notion continues to attract a lot of interests, especially in the area of kinetic Fokker--Planck equations; see e.g.~Baudoin \cite{Baudoin} and Dolbeault et al.~\cite{Mouhot} and the references therein. Unlike this literature, we are able to identify the hypocoercive constants, namely the exponential decay rate as the spectral gap and the constant in front of the exponential, which is greater than $1$ as with $\sigma^2,\mathrm{d}_\phi>0$ we have $\mathrm{m}_\phi>\mathrm{d}_\phi$, is a measure of the deviation of the spectral projections from forming an orthogonal basis.  Note that in general, the hypocoercive constants may be difficult to identify and may have little to do with the spectrum. Results in the spirit of \eqref{thm:l0_norm} have already been obtained by Achleitner et al.~\cite{BGK}, Patie and Savov \cite{Patie-Savov-GeL} as well as in  Patie and Vaidyanathan \cite{PV-Hypo} where a general framework based on intertwining relation is developed.
\end{remark}

\subsection{Hypercontractivity} A Markov semigroup defined on the state space $E$ with invariant distribution $\mu$ is said to be hypercontractive if there exists $\alpha>0$ such that
\begin{align*}
	|\!|\!|P_t|\!|\!|_{\bmrm{L}(E,\mu)\to\bmrm[p(\alpha t)]{L}(E,\mu)}\le 1
\end{align*}
where $p(t)=1+e^t$ and \[|\!|\!|P_t|\!|\!|_{\bmrm{L}(E,\mu)\to\bmrm[p(\alpha t)]{L}(E,\mu)}=\sup_{f:\|f\|_{\bmrm{L}(E,\mu)}=1}\|P_t f\|_{\bmrm[p(\alpha t)]{L}(E,\mu)}.\] It is readily seen that the hypercontractivity reflects the regularity of the semigroup. For self-adjoint Markov semigroups, hypercontractivity can be interpreted in terms of their (modified) log-Sobolev constants, see \cite[Theorem~5.2.3]{Bakry_Book} and references therein. Nonetheless, even for the self-adjoint birth-death Laguerre chain, it is difficult to obtain a precise  value of the (modified) log-Sobolev constant. Using the concept of interweaving, see Section~\ref{ss:interweaving}, we circumvent this issue, and in fact, we are able to obtain the hypercontractivity estimates for (non self-adjoint) skip-free Laguerre semigroups up to a random warm-up time.
\begin{theorem}\hlabel{thm:hyper}
	If $\sigma^2>0$ and $\ovl{\ovl{\Pi}}(0)<\infty$, then, for all $\beta>\mathrm{m}_\phi=\frac{m+\ovl{\ovl{\Pi}}(0)}{\sigma^2}$ and $t\geq0$,
	\begin{align*}
		\left|\!\left|\!\left|\eK^\phi_{t+\tau_\beta}\right|\!\right|\!\right|_{\ell^2(\meKp)\to\ell^{p(t)}(\meKp)}\le 1
	\end{align*}
	where $\tau_\beta$ is defined  in \eqref{eq:tau_beta1}.
\end{theorem}
This theorem is proved in Section~\ref{ss:hyper_pf}.

\subsection{Bochner subordination of skip-free Laguerre chains} In the previous two sections we have seen that Theorem~\ref{thm:entropy} and Theorem~\ref{thm:hyper} hold for skip-free Laguerre semigroups up to a random warm-up or delay time denoted by $\tau_\beta$. However, applying a time-change on the skip-free Laguerre chains, we can obtain a new class of skip-free Markov chains for which the above theorems hold with a deterministic warm-up or delay time. In other words, we obtain a new class of Markov chains (with two sided-jumps of arbitrary size) for which the quantity $\tau_\beta$ can be replaced by a deterministic number. Since $\tau_\beta$ is an infinitely divisible random variable with $\phi_\beta\in\Be$ as its L\'evy-Khintchine exponent, one can consider the subordinator $(\tau_\beta(t),t\ge 0)$ such that $\tau_\beta(1)\stackrel{(d)}{=}\tau_\beta$. With an abuse of notation, we still denote this subordinator by $\tau_\beta$. Now, let us consider the subordinated Laguerre semigroup defined by
\begin{align}
	\eK^{\phi,\tau_\beta}_t=\int_0^\infty\eK^\phi_s\ \bb{P}(\tau_\beta(t)\in ds).
\end{align}
Since $\lim_{t\to\infty}\tau_\beta(t)=\infty$ almost surely, the semigroup $\eK^{\phi,\tau_\beta}$ has the same invariant measure $\meKp$. Below, we provide the spectral expansion, the $\Phi$-entropy convergence and the hypercontractivity property of $\eK^{\phi,\tau_\beta}$.
\begin{theorem}\hlabel{thm:subordination} Let $\phi\in\Be$ be associated with the triplet $(m,\sigma^2,\Pi)$.
	\begin{enumerate}
		\item \textbf{Spectral Expansion.}\hlabel{it:sbo_spect} If $\sigma^2>0$ then for all $\beta>0$, $\f\in\ell^2(\meKp)$ and $t>\frac{1}{2}$ we have
		\begin{align*}
			\eK_t^{\phi,\tau_\beta}\f=\sum_{k=0}^\infty e^{-t\phi_\beta(k)}\langle\f,\eigLdd_k\rangle_{\meKp}\eigLd.
		\end{align*}
		\item \textbf{$\Phi$-entropy decay.}\hlabel{it:sbo_ent} If $\sigma^2>0,\ \ovl{\ovl{\Pi}}(0)<\infty$ and $\beta>\mathrm{m}_\phi$, then for all admissible (see \eqref{eq:admissible} for definition) function $\Phi$ and $\f$ such that $\f,\Phi(\f)\in\ell^1(\meKp)$, we have, for all $t\ge 0$,
		\[\ent^\Phi_{\meKp}\left(\eK_{t}^{\phi,\tau_\beta}\f\right)\le e^{-\phi_\beta(1)(t-1)_+}\ent^\Phi_{\meKp}(\f)\]
where $t_+=\max(t,0)$.
		\item \textbf{Hypercontractivity.}\hlabel{it:sbo_hyp} If $\sigma^2>0$, $\ovl{\ovl{\Pi}}(0)<\infty$ and $\beta>\mathrm{m}_\phi$, then, for all $t\ge 0$,
		\[\left|\!\left|\!\left|\eK^{\phi,\tau_\beta}_{t+1}\right|\!\right|\!\right|_{\ell^2(\meKp)\to\ell^{q(t)}(\meKp)}\le 1\]
		where $q(t)=1+(1+\sigma^{-2})^t$.
	\end{enumerate}
\end{theorem}
This theorem is proved in Section~\ref{sec:thlast}.

\section{Examples}\hlabel{s:exm}
\subsection{The discrete Laguerre chains and the Meixner polynomials}\hlabel{ex:1} Let us consider the Bernstein function $\phi(u)=\sigma^2 u+m$ where $\sigma^2>0, m\ge 0$. If $\eK^\phi$ is the skip-free Laguerre semigroup associated with $\phi$, then the generator is given by
\[ \eL_\phi(n,l)=\begin{cases} \sigma^2n+m+\sigma^2+1 & \mbox{if $l=n+1$} \\
	(1+\sigma^2)n & \mbox{if $l=n-1$} \\
	-(1+2\sigma^2)n-m-\sigma^2-1 & \mbox{if $l=n$} \\
	0 & \mbox{otherwise}.
 \end{cases}\]
From Proposition~\ref{prop:n_phi}, the unique invariant distribution of $\eL_{\phi}$ is given by
\[\meK_\phi(n)=\frac{\Gamma\left(n+\frac{m}{\sigma^2}+1\right)}{\Gamma\left(\frac{m}{\sigma^2}+1\right)n!}2^{-n-\frac{m}{\sigma^2}-1}, \ n\in\ZZ_+.\]
The semigroup $\eK^\phi$ is self-adjoint in $\ell^2(\meKp)$ and it follows from Theorem~\ref{thm:eig_coeig} that the eigenfunctions $\eigL^\phi$ of $\eK^\phi_t$ corresponding to its eigenvalue $e^{-kt}$ form an orthogonal sequence in $\ell^2(\meKp)$. More specifically, writing \tred{$\beta=\frac{m}{\sigma^2}$}, for all $k\in\ZZ_+$,
	%\begin{eqnarray}
	%&K^{(\beta)}_t\Lambda=\Lambda\eK^{(\beta)}_t & \textrm{ on }  \ell^{2}(\mathbf{n}_{\beta}),  \\
	% &\eK^{(\beta)}_t\wi{\Lambda}_\beta=\wi{\Lambda}_\beta K^{(\beta)}_t & \textrm{ on }  \bmrm{L}(\nu_\beta),
	%\end{eqnarray}
	%with $\Lambda\wi{\Lambda}_\beta=K^{(\beta)}_{\log 2}$
	\begin{eqnarray*}\hlabel{eq:def_pbeta}
		\eigL^{\phi}(n)&=&\left(1+\sigma^{-2}\right)^{-\frac{k}{2}}\Gamma(\beta+1)\sum_{r=0}^{k} (-\sigma)^{-2r}  { k \choose r} \frac{\mathsf{p}_r(n)}{\Gamma(r+\beta+1)}\\&=&\left(1+\sigma^{-2}\right)^{-\frac{k}{2}}
		{}_2F_1\left(-n,-k,\beta+1;-\sigma^{-2}\right)
	\end{eqnarray*}
	where
	\begin{align}\hlabel{eq:hyper-geom}
	{}_{2}F_{1}(a,b;c;x)=\frac{\Gamma(c)}{\Gamma(a)\Gamma(b)}\sum_{r=0}^{\infty}\frac{\Gamma(r+a)\Gamma(r+b)}{\Gamma(r+c)}\frac{x^r}{r!}.
	\end{align}
From \cite[Equation (7)]{Karlin-McG}, it follows that
	\begin{equation*}\hlabel{eq:eigen_norm}
		\left\|\eigL^{\phi}\right\|^2_{\ell^2(\meKp)} =\mf{c}_k\left(\beta\right)^{-1}
	\end{equation*}
where for any $a>0$, $\mf{c}_k(a)=\frac{\Gamma(a+k+1)}{\Gamma(a+1)\Gamma(k+1)}$.
Finally, for all $\mathsf{f}\in\ell^2(\meKp)$ and $t>0$, we have, in $\ell^2(\meKp)$,
	\begin{equation*}\hlabel{eq:spec-exp-rev}
		\eK^{\phi}_t \mathsf{f} = \sum_{k=0}^{\infty}\mf{c}_{k}\left(\frac{m}{\sigma^2}\right)e^{-kt} \left\langle \mathsf{f}, \eigL^{\phi} \right\rangle_{\mathbf{n}_\phi}\!\! \eigL^{\phi}.
	\end{equation*}

\begin{comment}
Note that, by means of the integral representation of the gamma function, one obtains
\begin{eqnarray}
\eab \Lambda f &=& \int_{0}^{\infty}e^{-x}\sum_{n=0}^{\infty}f(n)\frac{x^n}{n!} x^{\beta-1}e^{-x}\frac{dx}{\Gamma(\beta)}\\
&=& \sum_{n=0}^{\infty}f(n)\frac{\Gamma(n+\beta)}{\Gamma(\beta) n!}2^{-n-\beta}
\end{eqnarray}
which yields the identity in distribution
\begin{equation}\hlabel{}
\mathrm{Pois}(G(\beta))= B(\beta)
\end{equation}
where $\mathrm{Pois}(\lambda)$ is a Poisson variable of parameter $\lambda$, $G(\beta)$ (resp.~$B(\beta)$)
an independent gamma (resp.~negative binomial)  random variable of parameter $\beta$.
\end{comment}

\subsection{The perturbed Laguerre skip-free chain}
Consider the Bernstein function defined for $\mf{m}>1$ by
\[\phi_{\mf{m}}(u)=\frac{(u+\mf{m}+1)(u+\mf{m}-1)}{u+\mf{m}}=\frac{\mf{m}^2-1}{\mf{m}}+u+\int_0^\infty(1-e^{-uy})e^{-\mf{m}y}dy.\]
Let $\dBg_{\phi_{\mf{m}}}$ be the generator of the discrete self-similar Markov semigroup associated with $\phi_{\mf{m}}$. Then, according to \eqref{eq:skip_free_gen}, $\sigma^2=1, m=\frac{\mf{m}^2-1}{\mf{m}}$ and $\Pi(dy)=\mf{m}e^{-\mf{m}y}dy$. So, the infinitesimal generator $\dBg_{\mf{m}}$ is given by
\begin{align*}
	\dBg_{\phi_{\mf{m}}}(n,l)=\begin{cases}
		\frac{\mf{m}\Gamma(l+\mf{m})\Gamma(n-l+2)}{(n+1)\Gamma(n+\mf{m}+2)} & \mbox{if $l\in\lin 0,n-2\rin$}\vspace{.2cm} \\
		\frac{2\mf{m}}{(n+1)(n+\mf{m})(n+\mf{m}+1)}+n & \mbox{if $l=n-1$} \vspace{.2cm} \\
		\mf{m}-\frac{1}{\mf{m}+n+1} & \mbox{if $l=n+1$} \vspace{.2cm} \\
		\frac{\mf{m}}{(n+\mf{m})(n+\mf{m}+1)}-\frac{1}{\mf{m}} & \mbox{if $l=n$} \vspace{0.2cm}\\
		0 & \mbox{if $l>n+1$}.
	\end{cases}
\end{align*}
Now, the corresponding skip-free Laguerre chain has the generator $\eL_{\phi_{\mf{m}}}$ given by
\[\eL_{\phi_{\mf{m}}}(n,l)=\begin{cases}
	\dBg_{\mf{m}}(n,l) & \mbox{if $l\neq n,n-1$} \\
	\dBg_{\mf{m}}(n,n-1)+n & \mbox{if $l=n-1$} \\
	\dBg_{\mf{m}}(n,n)-n & \mbox{if $l=n$}.
\end{cases}\]
From Proposition~\ref{prop:n_phi}, the unique invariant distribution of $\eLm$ is given by
\begin{align*}
	\mathbf{n}_{\phi_\mf{m}}(n)=\frac{(n+\mf{m}+1)\Gamma(n+\mf{m})}{(\mf{m}+1)\Gamma(\mf{m})n!}2^{-(n+\mf{m}+1)}, \ n\in\ZZ_+.
\end{align*}
Let us compute the eigenfunctions and co-eigenfunctions of the semigroup $\eK^{\phi_\mf{m}}$ generated by $\eLm$. Denoting the eigenfunction (resp.~the co-eigenfunction) of $\eK^{\phi_\mf{m}}_t$ corresponding to the eigenvalue (resp.~co-eigenvalue) $e^{-kt}$ by $\eigL^{\phi_\mf{m}}$ (resp.~ $\coeig^{\phi_\mf{m}}_k$), we have
\begin{eqnarray*}
	\eigL^{\phi_\mf{m}}(n)&=&2^{-\frac{k}{2}}\sum_{r=0}^k(-1)^r\frac{\binom{k}{r}}{W_{\phi_{\mf{m}}}(r+1)}\mathsf{p}_r(n)
	\\&=&2^{-\frac{k}{2}}\left[(\mf{m}+1){}_2F_1\left(-k,-n,\mf{m}+1;-1\right)-{}_2F_1\left(-k,-n,\mf{m}+2;-1\right)\right], \\
	\coeig^{\phi_\mf{m}}_k(n)&=&\frac{2^{-\frac{k}{2}}}{\meKp(n)}\frac{\Gamma(n+k+\mf{m})}{k!n!}\left((n+\mf{m}+k){}_2F_1(-k,-n,-n-k-\mf{m};2) \right. \\&&\left.+{}_2F_1(-k,-n,-n-k-\mf{m}+1;2)\right)
\end{eqnarray*}
where ${}_2F_1$ is the hypergeometric function defined in \eqref{eq:hyper-geom}.
\subsection{The Beta skip-free chain} We consider the Bernstein function $\phi_{\ttt{m}}$ corresponding to a compound Poisson process with exponential jumps which is defined, for $\ttt{m}>1$ and $u>0$, by
\[\phi_{\ttt{m}}(u)=\frac{u}{\ttt{m}(u+\ttt{m})}=\int_0^\infty(1-e^{-uy})e^{-\ttt{m}y}dy.\]
Therefore, according to \eqref{eq:bernstein_def}, $\sigma^2=0, m=0$, $\Pi(dy)=\ttt{m}e^{-\ttt{m}y}dy$ and $\phi_{\ttt{m}}(\infty)=\frac{1}{\ttt{m}}$. If $L_{\phi_\ttt{m}}$ denotes the generator of the Laguerre semigroup corresponding to $\phi_\ttt{m}$ in continuous state space, we have for all $f\in\esC^\infty_c(\RR_+)$,
\begin{equation*}
	{{L}}_{\phi_\ttt{m}}f(x) = -x f'(x) + \frac{\ttt{m}}{x}\int^{\infty}_{0} \left(f(e^{-y}x)-f(x)+yxf'(x)\right)  e^{-\ttt{m}y} dy.
\end{equation*}
The Bernstein-gamma function associated with $\phi_\ttt{m}$ is
\[W_{\phi_\ttt{m}}(k+1)=\frac{\Gamma(\ttt m+1)\Gamma(k+1)}{\ttt{m}^k\Gamma(k+1+\ttt{m})}, \ k\in\ZZ_+.\]
From \cite[Proposition~2.6(1)]{Patie-Savov-GeL}, the semigroup generated by $L_{\phi_\ttt{m}}$ admits an unique invariant measure $\nu_{\phi_\ttt{m}}$ which is absolutely continuous with moment sequence $\left(W_{\phi_\ttt{m}}(k+1)\right)_{k\ge 0}$ and given by
\begin{equation*}
	\nu_{\phi_\ttt{m}}(dx)=
	\ttt{m}^2 (1-\ttt{m}x)^{\ttt{m}-1}dx,\  0<x<\frac{1}{\ttt m}.
\end{equation*}
Now coming back to the corresponding skip-free Laguerre chain in the discrete state space, \eqref{eq:inv_disc} implies that the unique invariant distribution of its semigroup $\eK^{\phi_\ttt m}$ is
\begin{align*}
	\meK_{\phi_\ttt m}(n)=&\frac{1}{n!}\sum_{r=0}^\infty W_{\phi_\ttt m}(n+r+1)\frac{(-1)^r}{r!} \\
	=&\frac{1}{n!}\sum_{r=0}^\infty \frac{\Gamma(\ttt m+1)\Gamma(n+r+1)}{\ttt{m}^{n+r}\Gamma(n+r+\ttt m+1)}\frac{(-1)^r}{r!}\\
	=&\frac{1}{\ttt m^n}\frac{\Gamma(\ttt m+1)}{\Gamma(n+\ttt{m}+1)}{}_1F_1\left(n,n+\ttt m;\frac{1}{\ttt m}\right)
\end{align*}
where ${}_1F_1$ is an hypergeometric function.
Finally, from Proposition~\ref{prop:eigen}, the eigenfunction of $\eK^{\phi_\ttt m}_t$ corresponding to $e^{-kt}$ is given by
\begin{align*}
	\eigL^{\phi_\ttt m}(n)=2^{-\frac{k}{2}}\sum_{r=0}^k(-1)^r\dbinom{k}{r}\frac{\mathsf{p}_r(n)}{W_{\phi_\ttt m}(r+1)}=2^{\frac{k}{2}}{}_3F_1(-k,-n,\ttt{m}+1;1;-\ttt{m})
\end{align*}
where ${}_3F_1(a,b,c;d;x)=\frac{\Gamma(d)}{\Gamma(a)\Gamma(b)\Gamma(c)}\sum_{r=0}^\infty\frac{\Gamma(r+a)\Gamma(r+b)\Gamma(r+c)}{\Gamma(r+d)}\frac{x^r}{r!}$ and
$\coeig^{\phi_\ttt m}_k$ is given by \eqref{eq:coeig}.
\section{Proof of the Main Results}\hlabel{s:proofs}
 We begin this section with some useful facts about to the discrete dilation operator.
\begin{proposition}\hlabel{prop:D_alpha}
	\begin{enumerate}
		\item\hlabel{it:d1} For all $\alpha>0$ and $\f\in\mathbf{C}_b(\ZZ_+)$,
		\begin{equation}\hlabel{eq:intdD}
			d_\alpha\Lambda \f=\Lambda \dD_\alpha \f
		\end{equation}
		where $d_\alpha f(x)=f(\alpha x)$ is the usual dilation operator on $\RR_+$ and  $\Lambda$ is as in Theorem \ref{thm:gateway_main}.
		\item\hlabel{it:d2} For all $\f\in\mathbf{C}(\ZZ_+)$, and $\alpha,\beta>0$, $\dD_{\alpha\beta}\f=\dD_\alpha\dD_\beta\f$.
		\item\hlabel{it:d3} $\dD_1=\mathrm{Id}$ and for all $\alpha>0$, $\dD^{-1}_\alpha=\dD_{1/\alpha}$.
\item \hlabel{it:d4} $(\dD_{e^{-t}})_{t\geq0}$ (resp.~$(d_{e^{-t}})_{t\geq0}$ form a semigroup on $\ell^2(\ZZ_+)$ (resp. on $\bmrm{L}(\RR_+)$) with generator $\partial^n_-=n\partial_-$ (resp.~$\partial^x=-x\frac{d}{dx}$) with $\partial^x \Lambda =\Lambda \partial^n_-$ on $\esC_c(\ZZ_+) $.
	\end{enumerate}
\end{proposition}
\begin{remark}
There are several analogies between the two dilation operators which make our choice of the discrete one natural. Indeed, as its continuous analogue, the discrete dilation  operator is a multiplicative semigroup, and, the generator of its associated additive semigroup is the discrete analogue of the continuous one, see item \ref{it:d4}. However, unlike in the  continuous  case, $\dD_\alpha$ does not have bounded inverse when $\alpha\in (0,1)$.
\end{remark}
\begin{proof}
The first item follows from \cite[Proposition~1]{miclo_patie}. Next,
	we note that, for any $\f\in\esC_b(\ZZ_+)$ and $\alpha>0$,
\begin{align}\hlabel{eq:est_lambda}
|\dD_\alpha\f(n)|\le |(2\alpha-1)|^n\|\f\|_\infty.
\end{align}
Then, \eqref{eq:est_lambda} implies that both $\Lambda\dD_\alpha \f$ and $d_\alpha\Lambda\f$ are well defined. Now, \eqref{it:d1} yields
\begin{align*}
	\Lambda\dD_{\alpha\beta}=d_{\alpha\beta}\Lambda=d_\alpha d_\beta\Lambda=d_\alpha\Lambda \dD_\beta=\Lambda\dD_\alpha\dD_\beta.
\end{align*}
Since $\Lambda:\esC_b(\bb Z_+)\to\esC_b(\bb{R}_+)$ is injective, see \cite[Lemma~4(4))]{miclo_patie}, item~\eqref{it:d2} follows. Item~\eqref{it:d3} is a direct consequence of item~\eqref{it:d2}. For item~\eqref{it:d4}, it is immediate from item~\eqref{it:d2} that $(\dD_{e^{-t}})_{t\ge 0}$ is a translation semigroup on $\esC_0(\ZZ_+)$. Moreover, from \eqref{eq:intdD} we have that for all $t\ge 0$ and $\f\in\esC_0(\ZZ_+)$
\[d_{e^{-t}}\Lambda\f=\Lambda\dD_{e^{-t}}\f.\] Differentiating the above identity with respect to $t$ when $\f\in\esC_c(\ZZ_+)$ and noting that $d_{e^{-t}}=e^{t\partial^x}$ one obtains that
\[\partial^x\Lambda\f=\Lambda\frac{d}{dt}\dD_{e^{-t}}\f_{|_{t=0}}\]
where we used that $\Lambda$ is a bounded operator. However, from \cite[Lemma~4.5]{miclo_patie}, after observing that $\Lambda=\nabla^{-1}$, we have $\partial^x\Lambda\f=\Lambda\partial^n_-\f$ whenever $\f\in\esC_c(\ZZ_+)$. Since $\Lambda$ is injective on $\esC_c(\ZZ_+)$, we conclude that for all $\f\in\esC_c(\ZZ_+)$ one has
\[\frac{d}{dt}\dD_{e^{-t}}\f_{|_{t=0}}\!\!=\partial^n_-\f\]
which proves item~\eqref{it:d4}.
\end{proof}

\subsection{Proof of Theorem~\ref{feller_skip}}\hlabel{ss:gen_disc} It is not difficult to see that the operator $\dBg_\Pi$ can be simplified as follows. We can write
$\dBg_\Pi f(n)=\sum_{l=0}^{n+1}\dBg(n,l)\mathsf{f}(l)$ where
\begin{align}\hlabel{eq:jump}
	\dBg_\Pi(n,l)=
	\begin{cases}
		\int_0^\infty \frac{1}{n+1}\dbinom{n+1}{l}e^{-ly}(1-e^{-y})^{n-l+1}\Pi(dy) & \mbox{if $l\in\lin 0, n-1\rin$  } \vspace{0.3cm}\\
		\int_0^\infty \frac{1}{n+1}\left(e^{-(n+1)y}-1+(n+1)y\right)\Pi(dy) & \mbox{if $l=n+1$} \vspace{0.3 cm} \\
		0 & \mbox{if $l>n+1$}
	\end{cases}
\end{align}
and $\dBg_\Pi(n,n)=-\sum_{l\neq n}\dBg_\Pi(n,l)$.

 To show that $\dBg_\phi$ is a Markov generator, we need to show that $\dBg_\phi(n,l)\ge 0$ for all $n\neq l$. From the expression of $\dBg_\phi$ in \eqref{eq:skip_free_gen}, it is enough to show that $\dBg_\Pi(n,l)\ge 0$ for all $l\neq n$, a fact which follows readily from \eqref{eq:jump}. To get that $\dBg_\phi$ generates a Feller semigroup on $\mathbf{C}_0(\ZZ_+)$, we wish to combine  Theorem~3.2 with Corollary~3.2 from  \cite[Chapter 8]{ethier-kurtz}. To this end, the following four conditions need to be checked
\begin{enumerate}[(i)]
\item \hlabel{cond1} $\sup_{n\in\bb{Z}_+}\frac{|\dBg_\phi(n,n)|}{n+1}<\infty$
\item \hlabel{cond2} $\lim_{n\to\infty} \dBg_\phi(n,l)=0$ for all $l\in\bb{Z}_+$
\item \hlabel{cond3} $\sup_{n\in\bb{Z}_+}\sum_{l\in\bb{Z}_+}\frac{n+1}{l+1}\dBg_\phi(n,l)<\infty$
\item \hlabel{cond4} $\sup_{n\in\bb{Z}_+}\frac{1}{n+1}\sum_{l\in\bb{Z}_+}(l-n)\dBg_\phi(n,l)<\infty$.
\end{enumerate}
First, we note that \[\dBg_\phi(n,l)=\begin{cases}
\sigma^2(n+1)+m+\dBg_\Pi(n,n+1) & \mbox{ if $l=n+1$}\\
\sigma^2 n +\dBg_\Pi(n,n-1) & \mbox{ if $l=n-1$} \\
-2\sigma^2n-\sigma^2-m+\dBg_\Pi(n,n) & \mbox{ if $l=n$} \\
\dBg_\Pi(n,l) & \mbox{ otherwise}.
\end{cases}
\]
It is plainly  sufficient to check all four conditions above  for $\dBg_\Pi$ merely. From the definition of $\dBg_\Pi(n,n)$, we get that, for all $n\in\bb{Z}_+$,
\begin{align}
\dBg_\Pi(n,n)=&-\int_0^\infty\frac{1}{n+1}\left(1-\sum_{l=0}^{n-1}\dbinom{n+1}{l} e^{-ly}(1- e^{-y})^{n-l+1}\right)\Pi(dy) \nonumber \\
&-\int_0^\infty\frac{1}{n+1} (1- e^{-(n+1)y}+(n+1)y)\Pi(dy). \nonumber
\end{align}
\tred{Since for any $y>0$, $\sum_{l=0}^{n+1}\dbinom{n+1}{l} e^{-ly}(1-e^{-y})^{n+1-l}=1$, the above expression reduces to}
\begin{align}
\dBg_\Pi(n,n)=\int_0^\infty( e^{-ny}- e^{-(n+1)y}-y)\Pi(dy). \hlabel{eq:int_bound}
\end{align}
Next, noting that
\begin{align*}
| e^{-ny}- e^{-(n+1)y}-y|= & \left|\int_n^{n+1} y(1- e^{-ry})dr\right|\\ \le & (2n+1)y^2\bbm{1}_{\{y\le 1\}}+y\bbm{1}_{\{y>1\}},
\end{align*}
the integral in \eqref{eq:int_bound} is finite due to \eqref{eq:levy} and therefore,
\begin{align*}
\limsup_{n\to\infty}\frac{|\dBg_\phi(n,n)|}{n+1}\le 2\sigma^2+2\int_0^1 y^2\Pi(dy)<\infty.
\end{align*}
This verifies condition (\ref{cond1}). Then, for any $l\in\ZZ_+$ and sufficiently large $n$,
$$\dBg_\Pi(n,l)=\frac{1}{n+1}\int_0^\infty\dbinom{n+1}{l} e^{-ly}(1- e^{-y})^{n-l+1}\Pi(dy).$$
When $l=0$,
\begin{align}
\dBg_\Pi(n,0)=\frac{1}{n+1}\int_0^\infty (1- e^{-y})^{n+1}\Pi(dy)
\end{align}
and clearly $n\mapsto \int_0^\infty(1- e^{-y})^{n+1}\Pi(dy)$ is a decreasing sequence. Thus, $\lim_{n\to\infty}\dBg_\Pi(n,0)=0$. When $l\ge 1$, let us define, for all \tred{$n\in\bb{Z}_+$ with $n\ge l+1$,}  $$a_n=\int_0^1 e^{-ly}(1- e^{-y})^{n-l+1}\Pi(dy),$$ $$ b_n=\int_1^\infty e^{-ly}(1- e^{-y})^{n-l+1}\Pi(dy).$$ \tred{We note that both $a_n,b_n$ are well defined if $n\ge l+1$.
Since, for all $n\ge l+1$, $a_{n+1}\le (1-e^{-1})a_n$, we have that $a_n\le a_{l+1}(1-e^{-1})^{n-l-1}$, and thus $$\lim_{n\to\infty}\dbinom{n+1}{l}a_n= 0.$$ On the other hand, observing that, for any $y>0$,
\begin{align*}
\lim_{n\to\infty}\dbinom{n+1}{l} e^{-ly}(1- e^{-y})^{n-l+1}&=0 \\
\sup_{n\ge 1}\dbinom{n+1}{l} e^{-ly}(1- e^{-y})^{n+1-l}&\le 1,
\end{align*}}
%To see why, consider $n$ i.i.d.~Bernoulli random variables $B_1,B_2,\cdots, B_{n+1}$ with $\mathbb{P}[B_i=0]=e^{-y}$. Then, by applying the strong law of large numbers $$\lim_{n\to\infty}\mathbb{P}[B_1+B_2+\cdots+B_{n+1}=l]=0,$$ which proves our claim. Also, for any $n\ge 1$, $$\dbinom{n+1}{l} e^{-ly}(1- e^{-y})^{n+1-l}< 1.$$
a dominated convergence argument entails that
$$\lim_{n\to\infty}\dbinom{n+1}{l}b_n= 0$$ which verifies condition (\ref{cond2}). \tred{For condition (\ref{cond3}), we first observe that for any $y>0$ and $l,n\in\ZZ_+$ with $l\le n+1$, the following identity
\begin{align*}
	\frac{1}{n+2}\sum_{j=1}^{n+1}(1-e^{-y})^j-\frac{y}{n+2}&=
	\sum_{l=0}^{n-1}\frac{1}{l+1}\dbinom{n+1}{l}e^{-ly}\left(1-e^{-y}\right)^{n-l+1} \\
	&+\frac{\left(e^{-(n+1)y}-1+(n+1)y\right)}{n+2} +\left(e^{-ny}-e^{-(n+1)y}-y\right)
\end{align*}
holds.
As a result of the above identity and invoking \eqref{eq:jump} one gets
\begin{align}
\sum_{l=0}^{n+1}\frac{n+1}{l+1}\dBg_\Pi(n,l)=\frac{1}{n+2}\int_0^\infty(1-e^{-y}-y)+\frac{1}{n+2}\sum_{j=2}^{n+1}(1-e^{-y})^j\Pi(dy).
\end{align}
}
Since for $y>0$, $|1-e^{-y}-y|\le y\wedge\frac{y^2}{2}$, using \eqref{eq:levy}, we get
\begin{align*}
\int_0^\infty |1-e^{-y}-y|\Pi(dy)\le \int_0^1 \frac{y^2}{2}\Pi(dy)+\int_1^\infty y\Pi(dy)<\infty
\end{align*}
and, for all $2\le j\le n+1$,
\begin{align*}
\int_0^\infty(1-e^{-y})^j\Pi(dy)\le \int_0^1 y^2\Pi(dy)+\int_1^\infty\Pi(dy)<\infty.
\end{align*}
Therefore, condition (\ref{cond3}) is satisfied as well. Finally, the last condition follows since, plainly,
\begin{align*}
\lim_{n\to\infty}\frac{1}{n+1}\sum_{l=0}^{n+1}(l-n)\dBg_\Pi(n,l)=\lim_{n\to\infty}\frac{1}{n+1}\int_0^\infty (e^{-y}-1+y)\Pi(dy)=0.
\end{align*}
Therefore, $\dBg_\phi$ generates a Feller semigroup on $\mathbf{C}_0(\ZZ_+)$ with $\mathbf{C}_c(\ZZ_+)$ as its core.

Next, to prove the discrete self-similarity property of the generated semigroup above, we need the following.
\begin{proposition}\hlabel{thm:gateway}
	\begin{enumerate}
		\item \hlabel{it:1} Let $Q^\phi$ and $\eQ^\phi$ denote the Feller semigroups generated by $G_\phi$ and $\dBg_\phi$ respectively. Then, for any $\f\in \mathbf{C}_0(\ZZ_+)$ and for all $t\ge 0$,
		\begin{align}\hlabel{eq:P_intertwining}
			Q^\phi_t \Lambda \f=\Lambda \eQ^\phi_t \f
		\end{align}
		where we recall that $\Lambda\mathsf f(x)=\bb{E}[\mathsf f(\mathrm{Pois}(x))]$ with $\mathrm{Pois}(x)$ a Poisson random variable with parameter $x>0$.
		\item \hlabel{it:2} The counting measure on $\ZZ_+$, denoted by $\bbm{m}$, is an excessive measure for the semigroup $\eQ^\phi$. Hence $\eQ^\phi$ can be extended uniquely to a strongly continuous contraction semigroup on $\ell^2(\ZZ_+)$, which we again denote by $\eQ^\phi$.
		\item \hlabel{it:3} The operator $\Lambda$ can be extended uniquely to an operator (also denoted by $\Lambda$) in $\mathscr{B}(\ell^2(\ZZ_+),\bmrm{L}(\bb{R}_+))$. Keeping the same notation for the extension of $Q^\phi$ on $\bmrm{L}(\bb{R}_+)$, we have, for all $\f\in\ell^2(\bb{Z}_+)$ and $t\ge 0$,
		\begin{align}\hlabel{eq:int_self_sim}
			Q^\phi_t \Lambda \f=\Lambda\eQ^\phi_t \f.
		\end{align}
		 Moreover, $\Lambda$ is a quasi-affinity, that is, it is bounded, injective and has dense range.
	\end{enumerate}
\end{proposition}
We split its proof into several parts.
\subsubsection{Proof of Proposition~\ref{thm:gateway}\eqref{it:1}}
First, let us write
\begin{equation}
\begin{aligned}
G_\phi=&G_{m,\sigma^2}+G_\Pi=\sigma^2x\frac{d^2}{dx^2}+(m+\sigma^2)\frac{d}{dx}+G_\Pi \\
\dBg_\phi=&\dBg_{m,\sigma^2}+\dBg_\Pi=(\sigma^2n+m+\sigma^2)\pp+n\pn+\dBg_\Pi
\end{aligned}
\end{equation}
where, for all $f\in \esC^2_b(\bb{R}_+),$ $$G_\Pi f(x)=\frac{1}{x}\int_0^\infty (f(x e^{-y})-f(x)+yxf'(x))\Pi(dy).$$
Let $\esP$ be the vector space of functions defined on $\bb{R}_+$ which are of the form $e^{-x} P(x)$, $P$ being a polynomial. We define the linear operator $\nab:\esP\to \mathbf{C}_c(\ZZ_+)$ as follows
\begin{align}\hlabel{eq:lambda_inv}
\nab f(n)=\frac{d^n}{dx^n}(e^x f(x))(0).
\end{align}
\begin{lemma}\hlabel{lem:rev_intertwining}
For any $f\in \esP$,
$$\dBg_\phi\nab f=\nab G_\phi f.$$
\end{lemma}
\begin{proof}
From \cite[Lemma~3]{miclo_patie}, it is known that, for all $f\in\esP$, $$\dBg_{m,\sigma^2}\nab f=\nab G_{m,\sigma^2} f.$$ Thus, it suffices to prove this lemma replacing $G_\phi, \dBg_\phi$ by $G_\Pi, \dBg_\Pi$ respectively. For $y>0$, let $\delta_y$ denote the Dirac measure at $y$. Taking $\Pi=\delta_y$ and writing $G_{\delta_y}$ and $\dBg_{\delta_y}$ simply as $G_y$, $\dBg_y$ respectively, we get
\begin{align}
G_y f(x)=\frac{1}{x}(f(x e^{-y})-f(x)+yx f'(x))
\end{align}
and
\begin{align}
\dBg_y (n,l)=\begin{cases}
\frac{1}{n+1}\dbinom{n+1}{l}(1- e^{-ly})(1- e^{-y})^{n-l+1} & \mbox{ if $l\in\lin 0,n-1\rin$} \\
\frac{1}{n+1}( e^{-(n+1)y}-1+(n+1)y) & \mbox{ if $l=n+1$} \\
0 & \mbox{ if $l>n+1$}
\end{cases}
\end{align}
with $\dBg_y(n,n)$ being such that $\sum_{l=\tred{0}}^{n+1}\dBg_y(n,l)=0$. Then, observing that $\dBg_\Pi(n,l)=\int_0^\infty\dBg_y(n,l)\Pi(dy)$ as well as $G_\Pi f(x)=\int_0^\infty G_yf(x) dy$ for all $f\in \esC^2_b(\bb{R}_+)$. We claim that it suffices to show that, for all $y>0$ and $f\in\esP$,
\begin{align}\hlabel{eq:dirac}
\dBg_y\nab f =\nab G_y f.
\end{align}
Indeed, when $f\in\esP$,
\begin{align*}
\dBg_\Pi\nab f(n)=\sum_{l=0}^{n+1}\dBg_\Pi(n,l)\nab f(l)=&\sum_{l=0}^{n+1}\nab f(l)\int_0^\infty\dBg_y(n,l)\Pi(dy)\\
=&\int_0^\infty\dBg_y\nab f(n) \Pi(dy).
\end{align*}
On the other hand,
\begin{align*}
\nab G_\Pi f(n)=\frac{d^n}{dx^n}(e^x G_\Pi f(x))(0)=\frac{d^n}{dx^n}\left(e^x\int_0^\infty G_y f(x)\Pi(dy)\right)(0).
\end{align*}
Since $\esP\subset \esC^\infty_b(\bb{R}_+)$, the above integration and differentiation can be interchanged, therefore yielding
$$\nab G_\Pi f(n)=\int_0^\infty\nab G_y f(n)\Pi(dy).$$
We now proceed to show \eqref{eq:dirac}. Since $\esP=\Span\{x\mapsto e^{-x}x^l; l\in\ZZ_+\}$, it suffices to prove \eqref{eq:dirac} only for $f(x)=\mathrm{h}_l(x):=e^{-x}x^l$. Now, for $l\ge 1$,
\begin{align}\hlabel{eq:diff}
\nab G_y \mathrm{h}_l(n)=\frac{d^n}{dx^n}\left[e^{x(1-xe^{-y})}x^{l-1}e^{-yl}-x^{l-1}+y(lx^{l-1}-x^l)\right](0).
\end{align}
When $l\in\lin 1,n-1\rin$, applying Leibniz rule we get
\begin{equation}\hlabel{eq:1_n-1}
\begin{aligned}
\nab G_y\mathrm{h}_l(n)=&e^{-ly}\sum_{m=0}^n\dbinom{n}{m}\frac{d^m}{dx^m}(x^{l-1})(0)\frac{d^{n-m}}{dx^{n-m}}\left(e^{x(1-e^{-y})}\right)(0) \\
=&\dbinom{n}{l-1}(l-1)!(1-e^{-y})^{n-l+1} \\
=&\frac{n!}{(n-l+1)!}e^{-ly}(1-e^{-y})^{n-l+1}=l!\dBg_y(n,l).
\end{aligned}
\end{equation}
Also, \eqref{eq:diff} entails
\begin{equation}\hlabel{eq:n_n+1}
\begin{aligned}
\nab G_y\mathrm{h}_n(n)=&n!(e^{-ny}(1-e^{-y})-y)=n!\dBg_y(n,n) \\
\nab G_y\mathrm{h}_{n+1}(n)=&n! (e^{-(n+1)y}-n!+(n+1)y)=(n+1)!\dBg_y(n,n+1).
\end{aligned}
\end{equation}
Finally,
\begin{align}\hlabel{eq:0}
\nab G_y\mathrm{h}_0(n)=\frac{d^n}{dx^n}\left(\frac{1}{x}(e^{x(1-e^{-y})}-1)\right)(0)=\frac{1}{n+1}(1-e^{-y})^{n+1}=\dBg_y(n,0).
\end{align}
On the other hand, for all $l\in\ZZ_+$, $\dBg_y\nab\mathrm{h}_l(n)=l!\dBg_y(n,l)$. Therefore, combining \eqref{eq:1_n-1}, \eqref{eq:n_n+1} and \eqref{eq:0}, we conclude that, for all $n,l\ge 0$,
\begin{align*}
\dBg_y\nab\mathrm{h}_l(n)=\nab G_y\mathrm{h}_l(n).
\end{align*}
This completes the proof of the lemma.
\end{proof}
The next lemma is a variant of \cite[Lemma~4]{miclo_patie}.
\begin{lemma}\hlabel{lem:Lambda}
$\nab:\esP\to \mathbf{C}_c(\ZZ_+)$ is bijective with inverse $\Lambda$ such that $\Lambda \f(x)=\bb{E}[\f(\mathrm{Pois}(x))]$ for all $\f\in \mathbf{C}_c(\ZZ_+)$. Moreover, $\Lambda$ extends to a bounded operator from $\mathbf{C}_0(\ZZ_+)$ to $\mathbf{C}_0(\bb{R}_+)$.
\end{lemma}
We also need the following useful result.
\begin{lemma}\hlabel{lem:domain_selfsim}
	For all $\phi\in \Be$, $\esP\subset\esD(G_\phi)$ where $\esD(G_\phi)$ is the domain of the generator of the Feller semigroup $Q^\phi$.
\end{lemma}
\begin{proof}
	Denoting the Dynkin characteristic operator of the semigroup $Q^\phi$ by $G^{D}_\phi$, it follows from \cite[Proposition~6.1]{Lamperti-72} that  $\esP\subset\esD(G^{D}_\phi)$. Also, for all $f\in\esP$, $G^{D}_\phi f=G_\phi f$. In light of \cite[Theorem~5.5, Chapter V.3]{dynkin:1965}, it suffices to show that for all $f\in\esP$, $G_\phi f\in\esC_0([0,\infty))$. Since any function $f\in\esP$ is of the form $f(x)=e^{-x}P(x)$ for some polynomial $P$, clearly $G_{m,\sigma^2} f\in\esC_0([0,\infty))$. Now, for any $f\in\esP$, we have
	\begin{align*}
		G_\Pi f(x)&=\frac{1}{x}\int_0^\infty[f(xe^{-y})-f(x)+yxf'(x)]\Pi(dy)\\			&=\frac{1}{x}\int_0^\infty[f(xe^{-y})-f(x)-(e^{-y}-1)xf'(x)]\Pi(dy) \\
			&+f'(x)\int_0^\infty (e^{-y}-1+y)\Pi(dy) \\
			&=~ A_1(x)+A_2(x).
	\end{align*}
Since $f\in\esP$, it implies that $f'\in\esC_0([0,\infty))$ and therefore $A_2(x)\to 0$ as $x\to\infty$. \tred{To deal with $A_1$, by means of Taylor expansion of $f$ up to order $2$ we obtain that}
\begin{align}\hlabel{eq:bound_0^1}
	\frac{1}{x}&\int_0^1 \left|f(xe^{-y})-f(x)-(e^{-y}-1)xf'(x)\right|\Pi(dy)\nonumber \\ &\le \tred{\frac{x}{2}}\int_0^1  \sup_{t\in [e^{-y}x,x]}\!\!\!\!|f''(t)|(1-e^{-y})^2\Pi(dy).
\end{align}
We note that any function $f\in\esP$, $f$ is either eventually increasing or decreasing, depending on the sign of the leading coefficient of the polynomial associated with the function. Without loss of generality, we assume that $f$ is eventually decreasing. Since for all $y\in [0,1]$, $e^{-y}\ge e^{-1}$, for all large values of $x$, we have $\sup_{t\in [e^{-y}x,x]}|f''(t)|\le|f''(e^{-1} x)|$, as $f\in\esP$ implies $f''\in\esP$. Thus, the right-hand side of \eqref{eq:bound_0^1} goes to $0$ as $x\to\infty$. On the other hand,
\begin{align*}
	\frac{1}{x}|f(xe^{-y})-f(x)-(e^{-y}-1)xf'(x)|\le (1-e^{-y})\|f'\|_\infty\le y\|f'\|_\infty.
\end{align*}
Using the dominated convergence theorem, we obtain that
\begin{align*}
	\lim_{x\to\infty}\frac{1}{x}\int_1^\infty[f(e^{-y}x)-f(x)+(e^{-y}-1)xf'(x)]\Pi(dy)=0.
\end{align*}
This shows that $A_2(x)\to 0$ as $x\to\infty$ which completes the proof of the lemma.
\end{proof}
Let $\esD(\dBg_\phi)$ denote the domain of the Feller generator $\dBg_\phi$. Now, coming back to the proofs of Proposition~\ref{thm:gateway}, Lemma~\ref{lem:rev_intertwining}, Lemma~\ref{lem:Lambda} and Lemma~\ref{lem:domain_selfsim} imply that, for all $\f\in \mathbf{C}_c(\ZZ_+)$,
\begin{align}\hlabel{eq:gen_intertwining}
\Lambda \f\in\esP\subset\esD(G_\phi) \text{ and } G_\phi\Lambda \f=\Lambda \dBg_\phi \f.
\end{align}
Since $\mathbf{C}_c(\ZZ_+)$ is a core for the generator $\dBg_\phi$, for any $\f\in\esD(\dBg_\phi)$ there exists a sequence $\{\f_n\}\subset \mathbf{C}_c(\ZZ_+)$ such that $\|\f_n-\f\|_\infty\to 0$ and $\|\dBg_\phi \f_n-\dBg_\phi \f\|_\infty\to 0$. Therefore, thanks to Lemma~\ref{lem:Lambda}, $$\|\Lambda \f_n-\Lambda \f\|_\infty\to 0, \  \|\Lambda \dBg_\phi \f_n-\Lambda\dBg_\phi \f\|_\infty\to 0.$$ Thus, \eqref{eq:gen_intertwining} entails that $G_\phi\Lambda \f_n$ converges in $\mathbf{C}_0(\bb{R}_+)$ which implies that $\Lambda \f\in\esD(G_\phi)$ as $(G_\phi,\esD(G_\phi))$ is a closed operator, and, for all $\f\in\esD(\dBg_\phi)$,
\begin{align}\hlabel{eq:int_full_domain}
G_\phi \Lambda \f=\lim_{n\to\infty}G_\phi\Lambda \f_n=\lim_{n\to\infty}\Lambda\dBg_\phi \f_n=\Lambda\dBg_\phi \f.
\end{align}
Using Kolmogorov's forward and backward equations, we get, for all $\f\in\esD(\dBg_\phi)$, $t>0$ and $s\in [0,t]$,
\begin{eqnarray*}
\frac{d}{ds}Q^\phi_s\Lambda\eQ^\phi_{t-s} \f & = & Q^\phi_sG_\phi\Lambda\eQ^\phi_{t-s}-Q^\phi_s\Lambda\dBg_\phi\eQ^\phi_{t-s} \f \\
& = & Q^\phi_s[G_\phi\Lambda-\Lambda\dBg_\phi]\eQ^\phi_{t-s} \f \\
& = & 0
\end{eqnarray*}
which is due to \eqref{eq:int_full_domain} together with the fact that $\eQ^\phi_{t-s}\f\in\esD(\dBg_\phi)$. Integrating the above identity, we obtain, for all $\f\in\esD(\dBg_\phi)$,
$$Q^\phi_t\Lambda \f=\Lambda\eQ^\phi_t \f.$$  Finally, using the density of $\esD(\dBg_\phi)$ and the boundedness of the operators $Q^\phi$, $\eQ^\phi$ and $\Lambda$, \eqref{eq:P_intertwining} follows.

\subsubsection{Proof of Proposition \ref{thm:gateway}\eqref{it:2}} It is plain that, for any $n\in\ZZ_+$, $$\int_0^\infty e^{-x}\frac{x^n}{n!}dx=1,$$ which implies that \tred{$\mu\Lambda=\bbm{m}$ where $\mu$ is the Lebesgue measure on $\bb{R}_+$.} Also, $\Lambda$ being a positive operator, for any $\f\in \mathbf{C}_0(\ZZ_+)$ with $\f\ge 0$, we have
$$\bbm{m} \f=\tred{\mu}\Lambda \f\ge \tred{\mu}Q^\phi_t\Lambda \f=\tred{\mu}\Lambda\eQ^\phi_t \f=\bbm{m}\eQ^\phi_t \f$$
\tred{where the second inequality in the above line holds as $\mu$ is an excessive measure for $Q^\phi$}.
This shows that $\bbm{m}$ is an excessive measure for $\eQ^\phi$.

\subsubsection{Proof of Proposition \ref{thm:gateway}\eqref{it:3}} For any $\f\in \mathbf{C}_c(\ZZ_+)$,
\begin{align*}
\|\Lambda \f\|^2_{\bmrm{L}(\bb{R}_+)}=\int_0^\infty (\Lambda \f(x))^2 dx=&\int_0^\infty\left(\sum_{n=0}^\infty e^{-x}\frac{x^n}{n!}\f(n)\right)^2dx \\
\le &\sum_{n=0}^\infty \f(n)^2\int_0^\infty e^{-x}\frac{x^n}{n!}dx \\
=&\sum_{n=0}^\infty \f(n)^2=\|\f\|^2_{\ell^2(\ZZ_+)}.
\end{align*}
Using the density of $\mathbf{C}_c(\ZZ_+)$ in $\ell^2(\ZZ_+)$, $\Lambda$ extends uniquely to a bounded operator from $\ell^2(\ZZ_+)$ to $\bmrm{L}(\bb{R}_+)$. Finally, for any $\f\in \mathbf{C}_0(\ZZ_+)\cap\ell^2(\ZZ_+)$, $\Lambda \f\in \mathbf{C}_0(\bb{R}_+)\cap\bmrm{L}(\bb{R}_+)$. Thus, for all $\f\in \mathbf{C}_0(\ZZ_+)\cap\ell^2(\ZZ_+)$, item (\ref{it:1}) ensures that
\begin{align*}
Q^\phi_t\Lambda \f=\Lambda\eQ^\phi_t \f.
\end{align*}
Again, using the density of $\mathbf{C}_c(\ZZ_+)$ in $\ell^2(\ZZ_+)$, \eqref{eq:int_self_sim} follows. Now, it remains to show that $\Lambda$ is a quasi-affinity. Boundedness of $\Lambda$ follows from item~\eqref{it:2}, and one  easily checks  that, for all $\f\in\ell^2(\ZZ_+)$,
\[\Lambda\f(x)=\sum_{n=0}^\infty e^{-x}\frac{x^n}{n!}\f(n) \ \ \text{ a.e.}\]
Therefore, $\ker(\Lambda)=\{0\}$, which proves the injectivity. The density of $\Range(\Lambda)$ follows by observing that $\wi{\Lambda}:\bmrm{L}(\bb{R}_+)\to\ell^2(\ZZ_+)$, the adjoint of $\Lambda$, takes the following form
\[\wi{\Lambda}f(n)=\frac{1}{n!}\int_0^\infty f(x)e^{-x}x^n dx=\bb{E}[f(\text{Gamma}(n+1))]\]
where $\text{Gamma}(n+1)$ is a gamma random variable with $n+1$ as the scale parameter and $1$ as the rate parameter. Approximating the $\bmrm{L}(\bb{R}_+)$ functions by compactly supported continuous functions, it can be shown that $\wi{\Lambda}$ is an injective operator, which proves that $\Range(\Lambda)$ is dense in $\bmrm{L}(\bb{R}_+)$. Hence, item~\eqref{it:3} is proven, which completes the proof of Proposition \ref{thm:gateway}.

\begin{corollary} \hlabel{cor:non-neg-gateway} For any $\f:\ZZ_+\to\bb{R}$ with $f\ge 0$, we have
$$Q^\phi_t\Lambda \f(x)=\Lambda\eQ^\phi_t \f(x)$$ for all $x\ge 0$.
\end{corollary}
\begin{proof}
For any nonnegative function $\f$, we can find $\{\f_n\}\subset\esC_c(\ZZ_+)$ such that $\f_n\uparrow \f$ pointwise. Then, Proposition~\ref{thm:gateway}(\ref{it:1}) yields, for all $x\ge 0$,
$$Q^\phi_t\Lambda \f_n(x)=\Lambda \eQ^\phi_t \f_n(x).$$ Since $\Lambda$ is a Markov kernel, $\Lambda \f_n\uparrow \Lambda \f$ as well. Writing $Q^\phi_t g(x)=\bb{E}_x[g(X_\phi(t))]$ and $\eQ^\phi_t \g(n)=\bb{E}_n[\g(\dX_\phi(t))]$ and invoking the monotone convergence theorem, the proof follows.
\end{proof}

\subsection*{End of the Proof of Theorem~\ref{feller_skip}} From the proof of Proposition~\ref{thm:gateway}\eqref{it:1}, we already have that
\[Q^\phi_{\alpha t}\Lambda\f=\Lambda \eQ^\phi_{\alpha t} \f\] for all $\f\in\esC_0(\ZZ_+)$. By a density argument, the above identity extends for all functions in $\esC_b(\ZZ_+)$. Now, for $\alpha\in [0,1]$, multiplying by $d_\alpha$ both sides of the above equation, we obtain that, for all $\f\in\esC_b(\ZZ_+)$,
\begin{align*}
	\Lambda\eQ^\phi_t\dD_\alpha\f=Q^\phi_t\Lambda\dD_\alpha\f=Q^\phi_t d_\alpha\Lambda \f=d_\alpha Q^\phi_{\alpha t}\Lambda\f=d_\alpha\Lambda\eQ^\phi_{\alpha t}\f=\Lambda\dD_\alpha\eQ^\phi_{\alpha t}\f
\end{align*}
where we  used the intertwining relationship between $d_\alpha$ and $\dD_\alpha$ given in Proposition~\ref{prop:D_alpha}. By means of the injectivity of $\Lambda$ on $\esC_b(\ZZ_+)$, we complete the proof. \qed

\subsection{Proof of Theorem~\ref{thm:skip-free}}\hlabel{ss:skip-free-pf} Let $\dBg$ denote the generator of the discrete self-similar Markov chain $\dX$. Then, from the definition of the discrete self-similarity, for any $\alpha\in [0,1]$, we have
\begin{align*}
	\dBg\dD_\alpha=\alpha\dD_\alpha\dBg \ \text{ on } \ \mathbf{D}(\dBg).
\end{align*}
Recalling that, for any $m,n\in\ZZ_+$, $\dD_\alpha (m,n)=\dD_\alpha\delta_n(m)=\dbinom{m}{n}\alpha^n (1-\alpha)^{m-n}\mathbbm{1}_{\{n\le m\}}$
we have, for all $l,n\in\ZZ_+$,
\begin{align}\hlabel{eq:G}
	\sum_{k\ge l}\dBg(n,k)\dbinom{k}{l}\alpha^l(1-\alpha)^{k-l}=\sum_{j\le n}\alpha^{j+1}(1-\alpha)^{n-j}\dBg(j,l).
\end{align}
Taking $n=0$ and $l=1$ in the above equation, we obtain, for all $k\in\ZZ_+$, that
\[\sum_{k\ge 1}\dBg(0,k)\alpha(1-\alpha)^{k-1}=\alpha\dBg(0,1).\]
Using the fact that $\dBg(0,k)\ge 0$ for all $k>0$, we conclude that $\dBg(0,k)=0$ for all $k\ge 2$. We now use an induction argument to prove that $\dBg(n,k)=0$ for all $k\ge n+2$. Let us assume that for all $n<N\in\ZZ_+$, $\dBg(n,k)=0$ for all $k\ge n+2$. Now plugging $n=N, l=N+1$ in \eqref{eq:G} and using the induction hypothesis, we have
\[\sum_{k\ge N+1}\dBg(N,k)\dbinom{k}{N+1}\alpha^{N+1}(1-\alpha)^{k-N-1}=\alpha^{N+1}\dBg(N,N+1).\] Again invoking the nonnegativity  of $\dBg(N,k)$ for $k\neq N$, we conclude that $\dBg(N,k)=0$ for all $k\ge N+2$. This completes the induction step and therefore the theorem is proved. \qed

\subsection{Factorial moments of discrete self-similar Markov chains} Let us recall that the skip-free Markov chain associated to $\phi$ is denoted by $\dX_\phi$. In the spirit of the work of  Bertoin and Yor \cite[Proposition~1(i)]{bertoin_yor} on the integer moments of the continuous analogues, we provide an explicit formula for the factorial moments of $\dX_\phi(t)$. For $z\in\bb{C}$, we recall that  $\mathsf{p}_z:\ZZ_+\to\bb{C}$ is the function defined by
\begin{align}\hlabel{eq:p_z}
	\mathsf{p}_z(n)=\frac{\Gamma(n+1)}{\Gamma(n+1-z)}.
\end{align}
It is well-known that, for any $n,k\in\ZZ_+$,
\begin{align}\hlabel{eq:stirling_1}
	n^k=\sum_{j=0}^k\begin{Bmatrix}k\\j\end{Bmatrix}\mathsf{p}_j(n)
\end{align}
where $\begin{Bmatrix}k\\j\end{Bmatrix}$ are the Stirling numbers of second kind, see \cite[p.~81]{Stanley}.
\begin{theorem}\hlabel{thm:moments}
	For any $n,k\in\ZZ_+$ and $t\ge 0$,
	\begin{align}
		\bb{E}\left[\mathsf{p}_k(\dX_\phi(t,n))\right]=\eQ^\phi_t\mathsf{p}_k(n)=\sum_{l=0}^k\dbinom{k}{l}\frac{W_{\phi}(k+1)}{W_{\phi}(l+1)}\mathsf{p}_l(n)t^{k-l}
	\end{align}
	where, for all $n\in\ZZ_+$, $W_\phi(n+1)=\prod_{k=1}^n\phi(k)$ and $W_\phi(1)=1$.
\end{theorem}
\begin{proof}
Defining $p_k(x)=x^k$, from \cite[Proposition~1(i)]{bertoin_yor}, we have, for all $k\in\ZZ_+$ and $t\ge 0$,
\begin{align}
Q^\phi_t p_k(x)=\bb{E}[(X_\phi(t,x))^k]=&x^k+\sum_{l=1}^k\dbinom{k}{l}\phi(k)\phi(k-1)\cdots\phi(k-l+1)x^{k-l}t^l \nonumber \\
=&\sum_{l=0}^k\dbinom{k}{l}\frac{W_{\phi}(k+1)}{W_{\phi}(l+1)}x^lt^{k-l}.\hlabel{eq:moment_ssmp}
\end{align}
On the other hand, it is easy to see that, for all $x>0$, $\Lambda\mathsf{p}_k(x)=p_k(x)$. Applying Corollary~\ref{cor:non-neg-gateway} with $f=\mathsf{p}_k$, yields, for all $t,x\ge 0$,
\begin{align*}
Q^\phi_t p_k(x)=Q^\phi_t\Lambda\mathsf{p}_k(x)=\Lambda\eQ^\phi_t\mathsf{p}_k(x).
\end{align*}
 Recalling that $\nabla=\Lambda^{-1}$, see \cite[Lemma~4]{miclo_patie}, we get
$$\bb{E}[\mathsf{p}_k(\dX_\phi(t,n))]=\eQ^\phi_t\mathsf{p}_k(n)=\nabla Q^\phi_t p_k(n)=\frac{d^n}{dx^n}\left(e^x Q^\phi_t p_k(x)\right)(0).$$ Finally using the expression in \eqref{eq:moment_ssmp} together with the Leibniz rule, the result follows.
\end{proof}
\begin{remark}
Using \eqref{eq:stirling_1} and the above theorem, $\bb{E}\!\left[\dX^k_\phi(t,n)\right]$ can be also computed explicitly for all $n,k\in\ZZ_+$.
\end{remark}

\subsection{Proof of Theorem~\ref{thm:gateway_main}\eqref{it:scaling_lim}}\hlabel{ss:scaling_lim_pf}
For showing the weak convergence, we need to check the following two facts. First, the tightness property  of the sequence  $(\lb Y_n(t)=\frac{1}{n}\dX_\phi(nt,\lfloor nx\rfloor\rb)_{t\ge 0})$ and the finite-dimensional convergence of $(Y_n)$ to $X_\phi$. For the tightness property, applying \cite[Theorem~16.1]{kallenberg-book} together with the strong Markov property of $(Y_n)$, it is enough to show that $Y_n(h_n)\overset{P}{\to} x$ (in probability) whenever $h_n\to 0$. We will in fact show that $\bb{E}[(Y_n(h_n)-x)^2]\to 0$ as $n\to\infty$. From Theorem~\ref{thm:moments}, we obtain for all $t\ge 0$,
\begin{align*}
\bb{E}[Y_n(t)]=\frac{1}{n}\bb{E}[\dX_\phi(nt,\lfloor nx\rfloor)]=\frac{1}{n}\bb{E}[\mathsf{p}_1(\dX_\phi(nt,\lfloor nx\rfloor))]=\frac{1}{n}(\lfloor nx\rfloor+\phi(1)nt)
\end{align*}
and
\begin{align*}
\bb{E}[Y^2_n(t)]=&\frac{1}{n^2}\bb{E}[\mathsf{p}_2(\dX_\phi(nt,{\lf nx\rf}))+\mathsf{p}_1(\dX_\phi(nt,{\lf nx\rf}))] \\
=&\frac{1}{n^2}(\mathsf{p}_2(\lf nx\rf)+2\phi(2)nt\lf nx\rf+\phi(1)\phi(2)n^2t^2+\lf nx\rf+\phi(1)nt).
\end{align*}
Since $h_n\to 0$, the last two equations imply $\bb{E}[Y_n(h_n)]\to x$ and $\bb{E}[Y^2_n(h_n)]\to x^2$ as $n\to\infty$. Therefore, $\bb{E}[(Y_n(h_n)-x)^2]\to 0$, which proves the tightness of $(Y_n)$. Next, to get the finite-dimensional convergence, it is enough to prove that, for all $0\le t_1<t_2<\cdots <t_k$ and $(\alpha_1,\alpha_2,\cdots, \alpha_k)\in\ZZ^k_+$,
\begin{align}\hlabel{eq:fd_moments}
\lim_{n\to\infty}\bb{E}\left[Y^{\alpha_1}_n(t_1)Y^{\alpha_2}_n(t_2)\cdots Y^{\alpha_k}_n(t_k)\right]=\bb{E}\left[X^{\alpha_1}_\phi(t_1,x)X^{\alpha_2}_\phi(t_2,x)\cdots X^{\alpha_k}_\phi(t_k,x)\right],
\end{align}
as the finite-dimensional distributions of $X_\phi$ are moment determinate. To prove the above assertion, we need the following lemma.
\begin{lemma}\hlabel{lem:moment_conv}
For any $t\ge 0$ and $k\in\ZZ_+$,
\begin{align}
\lim_{n\to\infty}\bb{E}\left[(Y_n(t)-X_\phi(t,x))^k\right]= 0.
\end{align}
\end{lemma}
\begin{proof}
From Theorem~\ref{thm:moments}, it is clear that for any $t\ge 0$ and $k\in\ZZ_+$, the sequence $(Y^k_n(t))_{n\ge 0}$ is uniformly integrable and, for each $k\in\ZZ_+$,
\begin{align*}
\bb{E}[Y^k_n(t)]=&\frac{1}{n^k}\bb{E}\!\left[\dX^k_\phi(nt,\lf nx\rf)\right] \\
=&\frac{1}{n^k}\lb\bb{E}[\mathsf{p}_k(\dX_\phi(nt,\lf nx\rf))]+o(n^k)\rb.
\end{align*}
Therefore,
\begin{align*}
\lim_{n\to\infty}\bb{E}[Y^k_n(t)]=&\lim_{n\to\infty}\frac{1}{n^k}\bb{E}[\mathsf{p}_k(\dX_\phi(nt,\lf nx\rf))] \\
=&\lim_{n\to\infty}\frac{1}{n^k}\lb \mathsf{p}_k(\lf nx \rf)+\sum_{l=1}^k\dbinom{k}{l}\frac{W_\phi(k+1)}{W_\phi(k-l+1)}\mathsf{p}_{k-l}(\lf nx \rf)n^l t^l \rb \\
=& x^k+\sum_{l=1}^k\dbinom{k}{l}\frac{W_\phi(k+1)}{W_\phi(k-l+1)} x^{k-l}t^l=\bb{E}[X^k_\phi(t,x)].
\end{align*}
Since the random variable $X_\phi(t,x)$ is moment determinate, the above identity indicates that for each $t\ge 0$, as $n\to \infty$,
\begin{align*}
 Y_n(t)\overset{d}{\longrightarrow} X_\phi(t,x).
 \end{align*}
 This together with the uniform integrability mentioned above proves the lemma.
\end{proof}
Now, coming back to the proof of the main theorem, we prove \eqref{eq:fd_moments} by induction. Indeed, \eqref{eq:fd_moments} is satisfied for $k=1$, thanks to Lemma~\ref{lem:moment_conv}. Moreover, if for some $k\in\ZZ_+$ and $(\alpha_1,\alpha_2,\cdots, \alpha_k)\in\ZZ^k_+$ \eqref{eq:fd_moments} holds, then, by an uniform integrability argument as in the proof of the lemma, one can show that
\begin{align}\hlabel{eq:moment_conv2}
\lim_{n\to\infty}\bb{E}\left[\lb Y^{\alpha_1}_n(t_1)\cdots Y^{\alpha_k}_n(t_k)-X^{\alpha_1}_\phi(t_1,x)\cdots X^{\alpha_k}_\phi(t_k,x)\rb^2\right]=0.
\end{align}
Now, writing
\begin{align*}
M_{n,k}=Y^{\alpha_1}_n(t_1)\cdots Y^{\alpha_k}_n(t_k), \ M_k=X^{\alpha_1}_\phi(t_1,x)\cdots X^{\alpha_k}_\phi(t_k,x),
\end{align*}
for any $(\alpha_1,\alpha_2,\cdots,\alpha_{k+1})\in\ZZ^{k+1}_+$, we have
\begin{align}
&\left|\bb{E}\!\left[\prod_{i=1}^{k+1}Y^{\alpha_i}_n(t_i)-\prod_{i=1}^{k+1}X^{\alpha_i}_\phi(t_i,x)\right]\right| \nonumber \\
=&\left|\bb{E}\left[M_{n,k}Y^{\alpha_{k+1}}(t_{k+1})-M_kX^{\alpha_{k+1}}_\phi(t_{k+1},x)\right]\right| \nonumber \\
=&\left|\bb{E}\left[M_{n,k}Y^{\alpha_{k+1}}-M_kY^{\alpha_{k+1}}(t_{k+1})+M_kY^{\alpha_{k+1}}(t_{k+1})-M_k X^{\alpha_{k+1}}_\phi(t_{k+1},x)\right]\right| \nonumber \\
\le& \sqrt{\bb{E}[(M_{n,k}-M_k)^2]\bb{E}[Y^{2\alpha_{k+1}}(t_{k+1})]} \hlabel{line4} \\ & +\sqrt{\bb{E}[M^2_k]\bb{E}[(Y^{\alpha_{k+1}}(t_{k+1})-X^{\alpha_{k+1}}_\phi(t_{k+1},x))^2]}. \nonumber
\end{align}
In view of Lemma~\ref{lem:moment_conv}, \eqref{eq:fd_moments} and \eqref{eq:moment_conv2}, the expression on the right-hand side of \eqref{line4} tends to $0$ as $n\to\infty$. This completes the induction step of our hypothesis, therefore proving \eqref{eq:fd_moments} for $k+1$. This completes the proof of the finite-dimensional  convergence of the process $(Y_n)_{n\geq0}$, which concludes the proof of the theorem. \qed

\subsection{Intertwining of the skip-free Laguerre and generalized Laguerre semigroups} In this section we establish the connection between the generalized Laguerre semigroups as defined in \cite{Patie-Savov-GeL} and the skip-free Laguerre semigroups introduced therein. From Theorem~1.6(2) in the aforementioned reference, it is known that, for any $\phi  \in \Be$, the generalized Laguerre semigroup $K^{\phi}$ on $\bb{R}_+$ has a unique invariant distribution $\nu_{\phi}$ that is absolutely continuous and moment determinate. In the next result we show that the intertwining relationship in \eqref{eq:P_intertwining}  is retained for the Laguerre semigroups as well. In the next proposition, we use the fact that the semigroup $\eK^\phi$ has a unique invariant distribution denoted by $\meKp$, which is proved in Proposition~\ref{prop:n_phi} below.
\begin{proposition}\hlabel{prop:int_laguerre}
	\begin{enumerate}
		\item\hlabel{prop:lag_int_c0} Let $\phi \in \Be$, then we have, for all $t\geq 0$ and $\mathsf f\in\mathbf C_0(\bb{Z}_+)$,
		\begin{equation}\hlabel{eq:intKK}
			K^{\phi}_t\Lambda\mathsf f=\Lambda\eK^{\phi}_t\mathsf f
		\end{equation}
		where we recall that $\Lambda\f(x)=\bb{E}[\f(\mathrm{Pois}(x))]$.
		\item \hlabel{it:inv} $\meKp=\nu_\phi\Lambda$ is an invariant distribution of $\eK^\phi$, and, for all $n\in\ZZ_+$,
		\begin{align*}
			\meK_{\phi}(n)=\frac{1}{n!}\int_0^{\phi(\infty)}e^{-x}x^n\nu_{\phi}(x)dx.
		\end{align*}
		
		\item \hlabel{it:l2_int} The Feller semigroup $\eK^{\phi}$ extends uniquely to a strongly continuous Markov semigroup on $\ell^2(\meK_{\phi})$, which is again denoted by $\eK^{\phi}$. Furthermore, the operator $$\Lambda:\mathbf{C}_0(\bb{Z}_+)\to\mathbf{C}_0(\bb{R}_+)$$ has a unique extension in $\BBB(\ell^2(\meK_{\phi}),\mathbf{L}^2(\nu_{\phi}))$, and, for all $t\ge 0$ and $\mathsf f\in\ell^2(\meK_{\phi})$,
		\begin{equation}\hlabel{eq:gateway_lag}
			K^{\phi}_t\Lambda \mathsf f=\Lambda\eK^{\phi}_t\mathsf f.
		\end{equation}
		Moreover, taking the adjoint in the above identity, one gets, for all $t\ge 0$ and $f\in\mathbf{L}^2(\nu_{\phi})$,
		\begin{align}\hlabel{eq:adjoint_int}
			\wi{\eK}^{\phi}_t\wi{\Lambda}_{\phi} f=\wi{\Lambda}_{\phi}\wi{K}^\phi_t f
		\end{align}
		where $\wi{\Lambda}_{\phi}:\mathbf{L}^2(\nu_{\phi})\to\ell^2(\meK_{\phi})$ is the adjoint of $\Lambda$, and, for all $f\in\bmrm{L}(\nu_\phi)$,
		\begin{align}\hlabel{eq:Lambda_tilde}
			\wi{\Lambda}_\phi f(n)=\frac{1}{n!\meKp(n)}\int_0^\infty e^{-x}x^n\nu_\phi(x) f(x)dx.
		\end{align}
		\item \hlabel{it:self_adj}$\eK^{\phi}$ is self-adjoint in $\ell^2(\meK_{\phi})$ if and only if $\phi(u)=\sigma^2 u+m$ for some $\sigma^2,m\ge 0$.
		\end{enumerate}
\end{proposition}

\begin{proof}Since $\eL^\phi=\dBg_\phi+n\partial_-$ and $L^\phi=G_\phi-x\frac{d}{dx}$, it suffices to show that, for all $\f\in \esC_c(\ZZ_+)$,
\begin{align}\hlabel{eq:inter_diff}
	-x\frac{d}{dx}\Lambda\f(x)=\Lambda (n\partial_-)f(x).
\end{align}
From \cite[Lemma~23]{miclo_patie}, \eqref{eq:inter_diff} readily follows by considering the reverse intertwining relationship (i.e., taking the inverse of $\Delta$ in (45) of the aforementioned reference). Thus, we conclude that, for all $\f\in\esC_c(\ZZ_+)$,
\[L^\phi\Lambda\f=\Lambda\eL^\phi\f.\]
The rest of the proof follows similarly as in the proof of Theorem~\ref{thm:gateway}\eqref{it:1}.

Next, from the intertwining relation in \eqref{prop:lag_int_c0}, we deduce that, for all $\f\in\esC_0(\ZZ_+)$,
\[ \nu_{\phi} \Lambda  \SLp_t \tred{\f}  =  \nu_{\phi} \SLcp_t \Lambda \tred{\f}  =  \nu_{\phi}\Lambda\f \]
implying that $\meKp = \nu_{\phi} \Lambda$ is an invariant finite measure for $\SLp$. Now, for any $n\in\ZZ_+$,
\begin{align}\hlabel{eq:nphi_pos}
	\meKp(n)=\nu_\phi\Lambda\delta_n=\frac{1}{n!}\int_0^\infty e^{-x}x^n\nu_\phi(x)dx>0
\end{align}
and \[\sum_{n=0}^\infty\meKp(n)=\int_0^\infty\nu_\phi(x)dx=1.\]
Hence, $\meKp$ is the invariant distribution of $\eK^\phi$. The uniqueness of the invariant distribution will be proved in Proposition~\ref{prop:n_phi}\eqref{prop:mom_mu}.
To prove  \eqref{it:l2_int}, we note that since $\esC_0(\ZZ_+)$ is dense in $\ell^2(\meKp)$ and $\Lambda:\esC_0(\ZZ_+)\to\esC_0(\ZZ_+)$ is a Markov kernel, $\Lambda$ can be uniquely extended to a bounded operator in $\mathscr{B}(\ell^2(\meKp),\bmrm{L}(\nu_\phi))$. Also, using the density of $\esC_0(\ZZ_+)$ in $\ell^2(\meKp)$ and item~\eqref{prop:lag_int_c0}, the identity \eqref{eq:gateway_lag} follows. Now, to compute the adjoint $\wi{\Lambda}_\phi$ of $\Lambda$, let us first show that the right-hand side of \eqref{eq:Lambda_tilde} as a function of $n$ belongs to $\ell^2(\meKp)$. Using Young's inequality and item~\eqref{it:inv} one has
\begin{align*}
	\sum_{n=0}^\infty\frac{1}{\meKp(n)^2}\left(\int_0^\infty e^{-x} \frac{x^n}{n!}\nu_\phi(x) f(x)dx\right)^2\meKp(n)&\le\sum_{n=0}^\infty\int_0^\infty e^{-x}\frac{x^n}{n!}f^2(x)\nu_\phi(x)dx\\& =\|f\|^2_{\bmrm{L}(\nu_{\phi})}.
\end{align*}
Now, writing $\f(n)=\sum_{n=0}^\infty\frac{1}{n!\meKp(n)}\int_0^\infty e^{-x}x^nf(x)\nu_\phi(x) dx$, we have, for all $\g\in\ell^2(\meKp)$,
\begin{align*}	\langle\g,\f\rangle_{\meKp}&=\sum_{n=0}^\infty\f(n)\g(n)\meKp(n)=\sum_{n=0}^\infty\frac{1}{n!}\g(n)\int_0^\infty e^{-x}x^nf(x)\nu_\phi(x)dx\\& =\int_0^\infty\Lambda\g(x)f(x)\nu_\phi(x)dx
\end{align*}
where the third equality is justified by Fubini theorem. This shows that $\wi{\Lambda}_\phi f=\f$ which proves \eqref{eq:Lambda_tilde}.
Finally, to justify \eqref{it:self_adj}, we note that, for a $\phi\in \Be$, $\eK^\phi$ is self-adjoint in $\ell^2(\meKp)$ if and only if, for all $l,n\in\ZZ_+$,
\begin{align}
	\eL^\phi(n,l)\meKp(n)=\eL^\phi(l,n)\meKp(l).
\end{align}
Since $\eL^\phi(n,l)=0$ whenever $l\ge n+2$ and $\meKp(n)>0$ for all $n\in\ZZ_+$ (see the proof of item~\eqref{it:inv}), the above identity holds only if $\eL^\phi(n,l)=0$ for all $l\neq n,n-1,n+1$. This happens only if $\phi(u)=\sigma^2 u+m$ for some $\sigma^2,m\ge 0$.
\end{proof}
\subsection{Proof of Theorem \ref{thm:generation_lag}\eqref{it:P-K}} First, we recall that, for all $t\geq0$ and $f\in \esC_0(\bb{R}_+)$, $K^{\phi}_tf= Q^{\phi}_{e^t-1}d_{e^{-t}}f$, see e.g.~\cite{Patie-Savov-GeL}. Then, from \eqref{eq:intKK}, we have for all $t\geq 0$ and $\mathsf f\in\mathbf C_0(\bb{Z}_+)$,
		\begin{eqnarray*}
			\Lambda\eK^{\phi}_t\mathsf f &=& K^{\phi}_t\Lambda\mathsf f=Q^{\phi}_{e^t-1}d_{e^{-t}}\Lambda\mathsf f=Q^{\phi}_{e^t-1}\Lambda  \dD_{e^{-t}} \mathsf f=\Lambda \eQ^{\phi}_{e^t-1}  \dD_{e^{-t}} \mathsf f
		\end{eqnarray*}
where we used, from the third identity onwards,  successively \eqref{eq:P_intertwining} and \eqref{eq:intdD}. We conclude the proof by invoking the Feller property of the semigroups as well as the injectivity of $\Lambda$ on $\mathbf C_0(\bb{Z}_+)$, see  \cite[Lemma~4(4))]{miclo_patie}.

\subsection{The invariant distribution of the skip-free Laguerre semigroup}\hlabel{sss:inv} We now show that the invariant distribution $\meKp$ in Proposition~\ref{prop:int_laguerre}\eqref{it:invariant} is unique and provide several useful  representations. We recall that, for any $\phi\in\Be$, $W_\phi$ is the so-called Bernstein-gamma function which is defined as a solution to the following functional equation
\begin{align*}
	W_\phi(z+1)=&\phi(z)W_\phi(z) \ \forall z\in\bb{C}_+, \:
	W_\phi(1)=1.
\end{align*}
The above functional equation has a unique solution in the class of Mellin transforms of probability measures on $\RR_+$. For a detailed account of these functions, we refer to \cite{patie2018}.
\begin{proposition}\hlabel{prop:n_phi}
	\begin{enumerate}
		\item\hlabel{prop:mom_mu} For all $\phi\in\Be$, the invariant distribution $\meKp$ of $\eK^\phi$ is unique and is determined by its factorial moments
		\begin{equation} \hlabel{eq:disc_moment_nu}
			\meK_{\phi}\p_k=W_{\phi}(k+1)
		\end{equation}
		where $\p_k$ is defined in \eqref{eq:p_z}.
		\item \hlabel{it:contour_inv} For any $n\in \ZZ_+$ and $0<c<n+1+{\mathrm{d}}_{\phi},$
		\begin{equation}\hlabel{eq:cont_mu_n}
			\meK_{\phi}(n)=\frac{1}{n!}\frac{1}{2\pi \i} \int_{c-\i\infty}^{c+\i\infty}\Gamma(z)W_{\phi}(n-z+1)dz\end{equation}
		where $\mathrm{d}_{\phi}=\min\{u\geq 0; \: \phi(-u)=-\infty, \phi(-u)=0 \} \in [0,\infty]$.
		\item \hlabel{it:sum_inv} If $0\leq \sigma^2<1$, then, for any $n\in \ZZ_+$,
		\begin{align}\hlabel{eq:inv_disc}
			\meK_{\phi}(n)=\frac{1}{n!} \sum_{r=0}^{\infty}(-1)^r \frac{W_{\phi}(n+r+1)}{r!}.
		\end{align}
%		\begin{comment}
%		for every $n\in \ZZ_+$,
%		\begin{eqnarray*}
%		% \nonumber % Remove numbering (before each equation)
%		\meKp(n) &=& \frac{1}{n!}\int_{0}^{\phi(\infty)} e^{-x}x^n   \nu_{\phi}(x)dx>0.
%		\end{eqnarray*}
%		
%		
%		
%		Moreover, for any bounded function $\mathsf g$, we have \[\meKp\mathsf g=\E[\mathsf g(\mathrm{Pois}(V)]\]
%		where $\mathrm{Pois}(V)$ is a mixed Poisson distributed random variable  with mixture variable $V$ with density $\nu_{\phi}$, which is independent of the Poisson variable $\mathrm{Pois}$. Thus,  for any $k\in \N$,
%		\begin{equation} \hlabel{eq:disc_moment_nu}
%		\meKp \mathsf{p}_k=W_{\phi}(k+1),
%		\end{equation}
%		and therefore $\meKp$ is moment determinate.
%		\end{comment}

	\end{enumerate}
\end{proposition}
Let us first derive the factorial moments of the skip-free Laguerre chains.
\begin{lemma}\hlabel{lem:timechange}
	For any $t\ge 0$ and $k\in\ZZ_+$,
	\begin{align}
		\mathds{K}^{\phi}_t\mathsf{p}_k(n)=\sum_{l=0}^k \dbinom{k}{l}\frac{W_{\phi}(k+1)}{W_{\phi}(l+1)}\mathsf{p}_l(n)e^{-tl}\left(1-e^{-t}\right)^{k-l}.
	\end{align}
\end{lemma}
\begin{proof}
	Let us recall that $Q^\phi$ denote the spectrally negative self-similar semigroup associated to $\phi$, and, for any $t\ge 0$, $x>0$ and $f\ge 0$,
	\begin{align*}
		K^{\phi}_t f(x)=Q^\phi_{1-e^{-t}}d_{e^{-t}}f(x),
	\end{align*}
and,	writing $p_k(x)=x^k$, $x>0$,  we have, from \cite{bertoin_yor}, that, for all $k\in\ZZ_+$,
	\begin{align*}
		Q^\phi_t p_k(x)=\sum_{l=0}^k\dbinom{k}{l}\frac{W_\phi(k+1)}{W_{\phi}(l+1)}x^l t^{k-l}.
	\end{align*}
	Therefore,
	\begin{align*}
		K^{\phi}_t p_k(x)=\sum_{l=0}^k\dbinom{k}{l}\frac{W_\phi(k+1)}{W_{\phi}(l+1)}e^{-tl}\left(1-e^{-t}\right)^{k-l}p_l(x).
	\end{align*}
	Recalling that $\nabla=\Lambda^{-1}$ and $\nabla p_l=\mathsf{p}_l$ for all $l\in\ZZ_+$, it follows that
	\begin{align*}
		\mathds{K}^{\phi}_t\mathsf{p}_k(n)=\mathds{K}^{\phi}_t\nabla p_k(n)=&\nabla K^{\phi}_t p_k(n) \\
		=&\sum_{l=0}^k\dbinom{k}{l}\frac{W_\phi(k+1)}{W_{\phi}(l+1)}e^{-tl}\left(1-e^{-t}\right)^{k-l}\nabla p_l(n)\\
		=&\sum_{l=0}^k\dbinom{k}{l}\frac{W_\phi(k+1)}{W_{\phi}(l+1)}e^{-tl}\left(1-e^{-t}\right)^{k-l}\mathsf{p}_l(n)
	\end{align*}
which completes the proof.
\end{proof}

\subsection{Proof of Proposition~\ref{prop:n_phi}}
From Lemma~\ref{lem:timechange}, we observe that, for all $k,n\in\ZZ_+$,
\[\lim_{t\to\infty}\eK^\phi_t\p_k(n)=W_\phi(k+1).\] On the other hand, recalling that $\Lambda\p_k=p_k$ where $p_k(x)=x^k$ and $\nu_\phi p_k=W_\phi(k+1)$, see \cite[Proposition~2.6(1)]{Patie-Savov-GeL}, we get
\[\meKp\p_k=\nu_\phi\Lambda\p_k=\nu_\phi p_k=W_\phi(k+1).\] Now it remains to show that $\meKp$ is determined by its moments. Let us write $\mf{e}_a(n)=e^{an}$. Then, applying Tonelli theorem we get
	\[\mf{e}_a\meKp=\sum_{n=0}^{\infty} e^{an}\meKp(n)=\int_{0}^{\infty}e^{-x}\sum_{n=0}^{\infty} \frac{(e^{a}x)^{n}}{n!} \nu_{\phi}(x)dx=\int_{0}^{\infty}e^{(e^{a}-1)x} \nu_{\phi}(x)dx.\]
	Next, we have
	\[\int_{0}^{\infty}e^{(e^{a}-1)x} \nu_{\phi}(x)dx=\sum_{r=0}^{\infty} W_{\phi}(r+1) \frac{(e^{a}-1)^r}{r!}\]
	where we used the fact that $\int_0^{\phi(\infty)} x^n\nu_\phi(x)dx=W_\phi(n+1)$ , see \cite[Proposition~2.6(1)]{Patie-Savov-GeL}, and thus $\mathfrak e_{a}\meKp <\infty$ as soon as $(e^{a}-1)<\sigma^{-2}$, that is for at least any $0<a<\log(1+\sigma^{-2})$. This provides the moment determinacy of $\meKp$. Next, to prove \eqref{it:contour_inv} and \eqref{it:sum_inv}, we observe,  from Proposition~\ref{prop:inter-bdsf}\eqref{it:inv}, that for all $n\in\ZZ_+$,
\[\meKp(n)=\int_0^{\phi(\infty)}e^{-x}x^n\nu_\phi(x)dx.\]
 Expanding the exponential function in the identity above and using a classical Fubini argument, see e.g.~\cite[Section 1.77]{Titchmarsh39}, combined with the expression \eqref{eq:disc_moment_nu} of the moment of $ \nu_{\phi}$, we get
\begin{equation} \hlabel{eq:seriesmun}
	\meKp(n) = \frac{1}{n!}\int_{0}^{\infty} e^{-x}x^n   \nu_{\phi}(x)dx  =\frac{1}{n!} \sum_{r=0}^{\infty} W_{\phi}(n+r+1) \frac{(-1)^r}{r!}
\end{equation}
where the series  is absolutely convergent as soon as $$\lim_{k\to \infty}\frac{\phi(k+n)}{k}=\lim_{k\to \infty}\frac{\phi(k)}{k}=\sigma^2<1.$$
%where we use the identity
%\begin{eqnarray}
    %  \nu_{\phi} \Lambda \mathsf g &=& \sum_{n=0}^{\infty} \mathsf g(n) \frac{1}{n!}\int_{0}^{\infty} e^{-x}x^n   \nu_{\phi}(x)dx
    %= \sum_{n=0}^{\infty} \mathsf g(n) \meKp(n) \hlabel{eq:nuLam-mu}
 %\end{eqnarray}
To justify the contour integral representation in \eqref{eq:cont_mu_n}, we consider two cases. Assume first that $\sigma^2>0$ and we recall that for  large $|\im(z)|$,
\begin{align}
  |\Gamma(z)| &\sim  C_{\re(z)}\: |\im(z)|^{\re(z)-\frac{1}{2}}e^{-\frac{\pi}{2}|\im(z)|}, \: \re(z)>0,  \hlabel{eq:stirling} \\
 \left|W_{\phi}(n-z+1)\right|&\leq C_{n-\re(z)} e^{-\frac{\pi}{2}|\im(z)|},\: \re(z)<n+1+{\mathrm{d}}_{\phi}, \nonumber
\end{align}
where here and below $ C_{\re(z)}>0$ is a constant depending only on ${\re(z)}>0$. Note that the first estimate is the classical Stirling formula, see e.g.~\cite[(2.1.8)]{Paris01}, whereas the second bound follows from \cite[Theorem~6.2(2b)]{Patie-Savov-GeL}.  Therefore, the mappings $z\mapsto \Gamma(z)$   and $z\mapsto W_{\phi}(n-z+1)$ are both in $\bmrm{L}(\R)$ and holomorphic in the strip $0<\re(z)<n+1+{\mathrm{d}}_{\phi}$, see \cite[Theorem~6.1(2)]{Patie-Savov-GeL}.  Moreover $z\mapsto W_{\phi}(n+z)$ and $z\mapsto \Gamma(z)$ are the Mellin transform of $ x\mapsto x^n \nu_{\phi}(x)$, see \cite{Patie-Savov-GeL},  and $x\mapsto e^{-x}$, respectively.  Consequently, both of these functions are in $\mathbf L^2(\R_+)$. An application of Parseval identity for the Mellin transform yields
\[\frac{1}{2\pi \i} \int_{c-\i\infty}^{c+\i\infty}\Gamma(z)W_{\phi}(n-z+1)dz=\int_{0}^{\infty} e^{-x}x^n   \nu_{\phi}(x)dx,\]
  from where we easily derive the expression \eqref{eq:cont_mu_n} for $\sigma^2>0$.
Next, we consider the other case, that is $\sigma^2=0$, which ensures that  the series representation \eqref{eq:seriesmun} of $\meKp(n)$ is valid for all $n\in\ZZ_+$.
Then, using the facts that the mappings $z\mapsto \Gamma(z)$   and $z\mapsto W_{\phi}(n-z+1)$ are both holomorphic in the strip $0<\re(z)<n+1+{\mathrm{d}}_{\phi}$ and within this strip, \eqref{eq:stirling} still  holds and
\[ \left|W_{\phi}(n-z+1)\right|\leq C_{(n-\re(z))}.\]
This implies that for all $n\in \N$, the  integral \eqref{eq:cont_mu_n} is absolutely convergent and an application of Cauchy theorem, see \cite{Paris01} for the detailed computation,  yields that the contour integral can be expanded  as  follows
\[\frac{1}{2\pi \i} \int_{c-\i\infty}^{c+\i\infty}\Gamma(z)W_{\phi}(n-z+1)dz=\sum_{r=0}^{\infty} W_{\phi}(n+r+1) \frac{(-1)^r}{r!},\]
which completes the proof of \eqref{eq:inv_disc}. \qed

\subsection{Intertwining between skip-free Laguerre semigroups} It has been shown in \cite[Theorem~2.1]{Patie-Savov-GeL} that for any $\phi\in\Be$, the generalized (non self-adjoint) Laguerre semigroup $K^{\phi}$ is intertwined with the diffusive (self-adjoint) Laguerre semigroup $K$ (when $\phi(u)=u$) and the intertwining operator is a multiplicative Markov kernel corresponding to the exponential functional of the subordinator associated with $\phi$. An analogous result holds for the skip-free Laguerre semigroups as well (see Theorem~\ref{prop:inter-bdsf} below), although, we prove it under the assumption that $\sigma^2$ in \eqref{eq:bernstein_def} is  positive. The following proposition describes the intertwining operator $\Mphi$ that links the semigroups corresponding to skip-free (non-reversible) and the reversible Laguerre chains respectively. Let $\poly=\Span\left\{\p_k,~k\in\ZZ_+\right\}$ and for any $\phi \in \mathbf B$ associated with the triplet $(m,\sigma^2,\Pi)$, let $\Mphi : \mathcal P \mapsto \cP$ be defined by
\begin{align}\hlabel{eq:binomial}
	\Mphi  \mathsf{f}(n)=\E\left[\mathsf{f}(\ttt{B}(I_{\sigma_1},n))\right]=\sum_{r=0}^{n}\mathsf{f}(r){n\choose r}\E\left[I_{\sigma_1}^r(1-I_{\sigma_1})^{n-r}\right]
\end{align}
with $\E\left[I_{\sigma_1}^k\right]=\frac{\sigma_1^{k}k!}{W_{\phi}(k+1)}$ for all $k\in\ZZ_+$, where $\sigma_1$ is defined in \eqref{eq:tau} and   we recall that $W_{\phi}(k+1)=\prod_{r=1}^{k}\phi(r),~W_{\phi}(1)=1$.

\begin{theorem}\hlabel{prop:inter-bdsf}
	\begin{enumerate}
		\item \hlabel{it:inter-bdsf} For any $\phi \in \mathbf B$, we have the  intertwining relation on $\poly$
		\begin{eqnarray} \hlabel{eq:inter_bd_sf}
			\eK^\phi_t\Mphi=\Mphi\eK^{\sigma_1}_{t} %\quad \textrm{ and } \quad \Inter{ \SLd}{\Mphiad}{ \SLpd }
		\end{eqnarray}
		where $\eK^{\sigma_1}=\eK^{\phi}$ with $\phi(u)=\sigma_1 u$.
		\item \hlabel{it:Iphi_bdd} Moreover, if $\sigma^2>0$, then $\Mphi : \ell^2(\meK_{\sigma^2}) \mapsto \ell^2(\meKp)$ is a linear operator that is bounded, injective  with a  dense range and for all $\mathsf{f} \in \ell^2(\meK_{\sigma^2})$,
		\begin{equation}\hlabel{eq:bounded}
			\|\Mphi \mathsf{f}\|_{\ell^2(\meKp)}\leq \|\mathsf{f}\|_{\ell^2(\meK_{\sigma^2})}
		\end{equation}
		where $\meK_{\sigma^2}$ is the unique invariant distribution of $\eK^{\sigma^2}$, and, for all $t\ge 0$, $\f\in\ell^2(\meK_{\sigma^2})$,
		\begin{align}\hlabel{eq:inter-bdsf1}
			\eK^\phi_t\Mphi\f=\Mphi\eK^{\sigma^2}_{t}\f.
		\end{align}
	\end{enumerate}
\end{theorem}
As a consequence of the above theorem, we obtain the intertwining relationship among the class of discrete self-similar Markov semigroups.
\begin{corollary}
	For $\phi\in\Be$ with $\sigma^2>0$, we have
	\begin{align*}
		\eQ^\phi_t\Mphi=\Mphi\eQ_{\sigma^2 t}
	\end{align*}
	both on $\esC_0(\ZZ_+)$ and $\ell^2(\ZZ_+)$, where we recall that $\eQ^\phi$ (resp.~$\eQ$) is the discrete self-similar semigroup corresponding to the Bernstein function $\phi$ (resp.~ $\phi(u)=u)$.
\end{corollary}
We need the following lemma to prove the above theorem.
\begin{lemma} \hlabel{lem:Mphi} Recall the definition of $\p_z$ in  \eqref{eq:p_z}. Then, for all $k,n\in\ZZ_+$, we have
 \begin{eqnarray*}
\Mphi \mathsf{p}_k(n)&=&\frac{{\sigma}^k_1k!}{W_{\phi}(k+1)}\mathsf{p}_k(n)
\end{eqnarray*}
where $\sigma_1$ is defined in \eqref{eq:tau}.
\end{lemma}
\begin{proof}
Recalling the definition of $\Mphi$ in $\eqref{eq:binomial}$, we have, for all $k,n\in\ZZ_+$,
\begin{align}
	\Mphi\p_k(r)=&\sum_{r=0}^n\p_k(r)\dbinom{n}{r}\bb{E}\left[I^r_{\sigma_1}(1-I_{\sigma_1})^{n-r}\right]=\bb{E}\left[\p_k(\mathtt{B}(n,I_{\sigma_1}))\right]
\end{align}
where $\mathtt{B}$ denotes the Binomial random variable with the parameters written in the parentheses and the moments of $I_{\sigma_1}$ are given in \eqref{eq:binomial}. Also, invoking the definition of the discrete dilation operator $\dD$, we can write the above quantity as
\[\bb{E}\left[\p_k(\mathtt{B}(n,I_\rs))\right]=\bb{E}[\dD_{I_\rs}\p_k(n)]=\p_k(n)\bb{E}[I^k_\rs]=\frac{\sigma^k_1k!}{W_\phi(k+1)}\p_k(n)\]
where the last equality follows from \eqref{eq:binomial}. This proves the lemma.
\end{proof}

%We now turn back to the proof of the propositions.
\subsection{Proof of Theorem~\ref{prop:inter-bdsf}} From Lemma~\ref{lem:timechange} and Lemma~\ref{lem:Mphi}, we have, for all $t\ge 0$ and $k\in\mathbb N$,
\begin{align*}
	\eK^\phi_t\Mphi\mathsf{p}_k=&\frac{\sigma^k_1k!}{W_{\phi}(k+1)}\eK^\phi_t\mathsf{p}_k \\
	=&\frac{\sigma^k_1k!}{W_\phi(k+1)}\sum_{l=0}^ke^{-tl}(1-e^{-t})^{k-l}\dbinom{k}{l}\frac{W_{\phi}(k+1)}{W_{\phi}(l+1)}\mathsf{p}_l \\
	=&\sigma^k_1k!\sum_{l=0}^ke^{-tl}(1-e^{-t})^{k-l}\dbinom{k}{l}\frac{1}{W_{\phi}(l+1)}\mathsf{p}_l(n).
\end{align*}
On the other hand, recalling that $\eK^{{\rs}}=\eK^{\phi}$ with $\phi(u)=\rs u$, we have
\begin{align*}
	\Mphi\eK^{{\rs}}_{t}\mathsf{p}_k=&\sum_{l=0}^ke^{-tl}(1-e^{-t})^{k-l}\sigma^{k-l}_1\dbinom{k}{l}\frac{k!}{l!}\Mphi\mathsf{p}_l \\
	=&k!\sum_{l=0}^ke^{-tl}(1-e^{-t})^{k-l}\sigma^{k-l}_1\dbinom{k}{l}\frac{\sigma^l_1}{W_\phi(l+1)},
\end{align*}
which shows that for all $k\in\bb N$,
\begin{equation*}
\eK^\phi_t\Mphi\mathsf{p}_k=\Mphi\eK^{{\rs}}_{t}\mathsf{p}_k
\end{equation*}
and therefore, on $\poly$,
\begin{align*}
	\eK^\phi_t\Mphi=\Mphi\eK^{{\rs}}_{t}.
\end{align*}
To prove now \eqref{it:Iphi_bdd}, it is plain, from Lemma~\ref{lem:Mphi}, that $\Mphi(\poly)=\poly$. Then, under the condition $\sigma^2>0$, we have that $\P(I_{\sigma^2}\in [0,1])=1$, see \cite[Proposition 6.7]{Patie-Savov-GeL} (note that $I_{\sigma^2}=\sigma^2 I_\phi$, where $I_\phi$ is the exponential functional defined in the aforementioned paper) and thus $\Mphi $ is a Markov operator. By means of H\"older's inequality,
one obtains, for any $\f\in \ell^2({\meK_{\sigma^2}})$,
\begin{equation}  \hlabel{eq:MfMf2}
	\|\Mphi \f\|_{\ell^2(\meKp)}\leq \|\Mphi \f^2\|_{\ell^1(\meKp)} = \meKp \Mphi \f^2.
\end{equation}
Now, for all $k\in\ZZ_+$, using Lemma~\ref{lem:Mphi} and Proposition~\ref{prop:n_phi}\eqref{prop:mom_mu}, we obtain \[\meKp\Mphi\p_k=\sum_{n\in\ZZ_+}\meKp(n)\Mphi\p_k(n)=\sum_{n\in\ZZ_+}\meKp(n)\p_k(n)\frac{\sigma^{2k} k!}{W_\phi(k+1)}=\sigma^{2k}k!=\p_k\meK_{\sigma^2}\] which shows that $\meKp\Mphi=\meK_{\sigma^2}$ as $\meKp, \meK_{\sigma^2}$ are moment determinate. Therefore, \eqref{eq:MfMf2} entails that $\Mphi$ is a bounded operator from $\ell^2(\meK_{\sigma^2})$ to $\ell^2(\meKp)$ when $\sigma^2>0$. Hence, by the density of $\poly$ in $\ell^2(\meK_{\sigma^2})$, the intertwining relation given by \eqref{eq:inter_bd_sf} extends to $\ell^2(\meK_{\sigma^2})$. This completes the proof of the proposition. \qed

\subsection{Hilbert sequences and spectral expansion}\hlabel{ss:hilbert} In this section, we introduce a few notions from non classical harmonic analysis which have been shown recently to be central in the understanding of  the spectral expansions of non self-adjoint operators in Hilbert spaces, see e.g.~\cite{Patie-Savov-GeL}. Two sequences  $( \eigL)_{k\geq0}$ and $(\mathsf{V}_k)_{k\geq0}$ are said to be biorthogonal  in the Hilbert space $\ell^2(\mathbf m)$ if for any $k,l \in \ZZ_+$,
\begin{equation} \hlabel{eq:def_bio}
	\spnu{ \eigL ,\mathsf{V}_l}{\mathbf m} = \mathbbm{1}_{\{k=l\}}.
\end{equation}
Moreover, a sequence that admits a biorthogonal sequence will
be called \textit{minimal} and a sequence that is both minimal and complete, in the sense that its linear span is dense in $\ell^2(\mathbf m)$, will be called \textit{exact}.  It is easy to show that a sequence $( \eigL )_{k\geq 0} $  is minimal if and only if none of its elements can be approximated by linear combinations of the others. If this is the
case, then a biorthogonal sequence will be uniquely determined if and only if $( \eigL )_{k\geq 0}$
is complete. We proceed with some basic notions related to the concept of frames in Hilbert spaces. A recent and thorough account on these Hilbert space sequences can be found in the book of Christensen \cite{Christensen-03}.  A  sequence $( \eigL )_{k\geq 0}$ in $\ell^2(\mathbf m)$ is a frame if there exist   $A,B>0$ such that  the frame inequalities
\begin{equation} \hlabel{eq:frame1}
	A   \|f\|^2_{\ell^2(\mathbf m)} \leq \sum_{k=0}^{\infty}  |\langle f,  \eigL \rangle_{\mathbf m} |^2 \leq B   \|f\|^2_{\ell^2(\mathbf m)}
\end{equation}
hold, for all $f \in\ell^2(\mathbf m)$.  If only the upper  bound exists,  $( \eigL )_{k\geq 0}$ is called a Bessel sequence. A frame sequence is always complete in the Hilbert space and when it is minimal,  it is called a Riesz sequence. The latter are very useful objects as they share substantial properties with orthonormal sequences. Indeed, a Riesz sequence always admits a unique biorthogonal sequence $(\mathsf{V}_k)_{k\geq0}$ which is also a Riesz sequence and both together form the so-called Riesz basis. Moreover, the expansion  in terms of the Riesz basis of any element of the Hilbert space is  unique and convergent in the topology of the norm.  When $( \eigL )_{k\geq 0}$ is merely a Bessel sequence, that is only the upper  frame condition in \eqref{eq:frame1} is  satisfied, then the so-called synthesis operator, that is the linear operator $\mathcal{S} : \ell^2(\ZZ_+) \rightarrow \ell^2(\mathbf m)$ defined by
\begin{equation} \hlabel{eq:def_syn}
	\mathcal{S} : \underline{c}=(c_k)_{k\geq 0} \mapsto \mathcal{S}(\underline{c})=\sum_{k=0}^{\infty}  c_k  \eigL
\end{equation}
is a bounded operator with (operator) norm $ \|\mathcal{S}\|_{\mathbf m} \leq \sqrt{B} $, that is, the series is norm convergent for any sequence in $\ell^2(\ZZ_+)$. However, $\mathcal{S}$ is not in principle onto as the $( \eigL )_{k\geq 0}$ does not form in general a basis of the Hilbert space.
\begin{comment}
 To this end, we recall that for a linear operator $L$ acting on a Hilbert space $\mathbf{H}$, with identity operator $\rm{I}$, we write $\sigma(L)=\{\lambda \in \bb{C}; \: L-\lambda \rm{I} \textrm{ is not invertible} \},$ $\Spec_p(L)=\{\lambda \in \sigma(L); \: L-\lambda \rm{I} \textrm{ is not injective} \},$ $\sigma_c(L)=\{\lambda \in \sigma(L)\setminus\Spec_p(L); \: L-\lambda \rm{I} \textrm{ has a dense range} \}$  and $\Spec_r(L)=\sigma(L)\setminus (\Spec_p(L) \cup \sigma_c(L))$ which corresponds to the spectrum, point spectrum, continuous spectrum and residual spectrum of $A$ respectively, see e.g.~the monograph of Dunford and  Schwartz \cite[XV.8]{Dunford_II}.
\end{comment}

\begin{proposition}\hlabel{prop:eigen}
Let $\phi \in \Be$.
	\begin{enumerate}
		\item\hlabel{it:eigen} For any   $k\in \ZZ_+$, $\eigLd \in \lnu$, and, for any $t>0$,
		\[\SLp_t \eigLd=e^{-kt}\eigLd.\]
		Moreover, $\textrm{Span}\{\eigLd, k>0\}=\cc{P}$ which is dense in $\lnu$.
		
		\item\hlabel{it:eignorm}Assume that $\sigma^2>0$. Then, $(\eigLd)_{k\geq 0}$ is an  exact Bessel sequence in $\lnu$ with bound $1$ and for any $k\in \ZZ_+$,
		\begin{equation}\hlabel{eq:normP}
			\left\|\eigLd\right\|_{\lnu}\leq 1.
		\end{equation}
		\item \hlabel{it:eignorm2} If $\sigma^2>0$ and $\mathrm{d}_{\phi}>0$. Then,  $ \left(\sqrt{\mf{c}_{k}(\mathrm{d}_{\phi})}\eigLd\right)_{k\geq 0}$ is a Bessel sequence with, for all $k\in \ZZ_+$,
		\begin{equation}\hlabel{eq:bound_dp}
			\left\|\eigLd\right\|_{\lnp}\leq \frac{1}{\sqrt{\mf{c}_{k}(\mathrm{d}_{\phi})}}
		\end{equation}
		where $\mf{c}_{k}(\mathrm{d}_{\phi})=\frac{\Gamma(k+\mathrm{d}_{\phi}+1)}{\Gamma(k+1)\Gamma(\mathrm{d}_{\phi}+1)}$.
	\end{enumerate}
\end{proposition}
%\tred{
%\begin{remark}
%	If $\mathscr{E}_k$ denotes the eigenspace of $\eK^\phi_t$ corresponding to the eigenvalue $e^{-kt}$, with the help of \cite{?} and the injectivity of $\Lambda$ in \eqref{} it follows that $\dim(\mathscr{E}_k)=1$.
%\end{remark}
%}

\begin{proof}
Let $k\in \ZZ_+$, then it is plain, from Proposition \ref{prop:mom_mu}, that,  as a polynomial, $\eigLd \in \lnu$. Then, we  recall, from \cite[Theorem 7.3]{Patie-Savov-GeL} (after multiplying both sides of the next identity by $\rpnn$) that, for any $t>0$ and $k \in \ZZ_+$,
 \begin{eqnarray*}
\SLcp_t \eigLc_k = e^{-k t} \eigLc_k
\end{eqnarray*}
where  we have set
\[\eigLc_k(x) =\rpnn\sum_{r=0}^{k}(-1)^r{k \choose r}\frac{x^r}{W_{\phi}(r+1)} \in \Lnu. \]
Thus, since $\Lambda$ is injective on $\poly\subset \lnu$, the algebra of polynomials, we have $\Lambda^{-1}p_k(n)=\mathsf{p}_{k}(n)$ and
thus, by linearity
\begin{equation}\hlabel{eq:LPp}
\Lambda^{-1} \eigLc_k(n)=\rpnn\sum_{r=0}^{k}(-1)^r  \frac{{k \choose r}}{W_{\phi}(r+1)} \Lambda^{-1}p_r(n)=\eigLd(n).\end{equation}
Finally, we observe from the gateway relationship \eqref{eq:gateway_lag} and the linearity of the operators, that
\[\Lambda \SLp_t \eigLd= K^{\phi}_t\Lambda\Lambda^{-1} \eigLc_k= K^{\phi}_t \eigLc_k=e^{-kt}\eigLc_k=\Lambda e^{-kt}  \eigLd.\]
The injectivity of  $\Lambda$ on $\poly$ yields the eigenfunction property. To complete the proof of \eqref{it:eigen}, we recall the moment determinacy of $\meKp$, stated in Proposition \ref{prop:n_phi}\eqref{prop:mom_mu}, which entails, from classical results on the moment problem, the density property of  the algebra of polynomials in the  weighted Hilbert space, see \cite{Akhiezer-65}. \tred{Next, when $\sigma^2>0$ and $\phi(u)=\sigma^2 u$, we recall, from Example~\ref{ex:1}, that $(\eigL^{\sigma^2})_{k\ge 0}$ is an orthonormal sequence of eigenfunctions of $\eK^{\sigma^2}_t$ associated to the eigenvalues $\{e^{-kt}\}_{k\ge 0}$. Now, from Lemma~\ref{lem:Mphi}, it is easily seen, from the definition of $\eigL^\phi$, that, for all $k\ge 0$,
	\begin{equation*}
	\Mphi\eigL^{\sigma^2}=\eigL^\phi.
	\end{equation*}
Since $\Mphi\in\mathscr{B}\left(\ell^2(\meK_{\sigma^2}),\ell^2(\meKp)\right)$ whenever $\sigma^2>0$, see \eqref{eq:bounded}, it follows that, for all $k\ge 0$, one has
\begin{equation}\hlabel{eq:normP1}
	\|\eigLd\|_{\lnu}\leq \| \eigL^{\sigma^2}\|_{\ell^2(\meK_{\sigma^2})}\le 1.
\end{equation}
After recalling that $(\eigL^{\sigma^2})_{k\ge 0}$ is a complete orthonormal sequence in $\ell^2(\meK_{\sigma^2})$,  we observe that, for any $\f\in\ell^2(\meKp)$,
\begin{align*}
	\sum_{k=0}^\infty\left\langle\f,\eigL^\phi\right\rangle_{\meKp}=\sum_{k=0}^\infty \left\langle\widehat{\mathds{I}}_\phi\f,\eigL^{\sigma^2}\right\rangle_{\meK_{\sigma^2}}=\|\widehat{\mathds{I}}_\phi\f\|^2_{\ell^2(\meK_{\sigma^2})}\le\|\f\|^2_{\lnu}.
\end{align*}
This shows that $(\eigL^\phi)_{k\ge 0}$ is a Bessel sequence in $\lnu$.
%	corresponding to the eigenvalue $e^{-kt}$ and $\eigL^\r(0)=1$ for all $k\ge 0$. As a result of the intertwining relationship \eqref{eq:inter_bd_sf} and boundedness of $\Mphi $ when $\sigma^2>0$, see \eqref{eq:bounded}, one gets that $\Mphi\eigL^\r$ is an eigenfunction of $\eK^\phi_t$ corresponding to the eigenvalue $e^{-kt}$. Now, Lemma~\ref{lem:Mphi} produces that $\Mphi(\cc{P})=\cc{P}$, which implies that $\{\Mphi\eigL^\r\}_{k\ge 0}$ is a complete sequence of eigenfunctions of $\eK^\r_t$. Lemma~\ref{lem:Mphi} also entails that $\overline{\Range(\Mphi)}=\ell^2(\meKp)$ \textcolor{blue}{complete here} $(\eigL^\r)_{k\ge 0}$ is an orthonormal sequence in $\ell^2(\meK_\r)$, see Example~\ref{ex:1},  we get that $(\eigLd)_{k\geq 0}$ is a Bessel sequence in $\lnu$ with
}Combining item \eqref{it:eigen} with the existence of a biorthogonal sequence, see \eqref{eq:bioPV} below, we get that $(\eigLd)_{k\geq 0}$ is exact, which proves \eqref{it:eignorm}.
Finally, to prove \eqref{it:eignorm2}, let $d_\epsilon=\mathrm{d}_\phi-\epsilon$ for some $0<\epsilon<\mathrm{d}_\phi$ and define $\phi_{d_\ep}(u)=\frac{u\phi(u)}{u+d_\ep}$. From \cite[Lemma 10.3]{Patie-Savov-GeL}, it follows that $\phi_{d_\ep}\in\Be$ and \[\lim_{u\to\infty}\frac{\phi_{d_\ep}(u)}{u}=\sigma^2.\] Now, we need the following whose proof  can be carried out by following a line of reasoning similar to the one  of Theorem~\ref{prop:inter-bdsf}.
\begin{lemma}\hlabel{lem:d_phi_int}
	For all $t\ge 0$,
	\begin{align}
		\eK^{\phi}_t\mathds{I}_{\phi_{d_\ep}}=\mathds{I}_{\phi_{d_\ep}}\eK^{(d_\ep,{\sigma^2})}_{t} \text{ on } \ \ell^2(\meK_{d_\ep,{\sigma^2}})
	\end{align}
\tred{where $\eK^{(d_\ep,{\sigma^2})}$ is the discrete Laguerre semigroup associated to $\phi(u)=\sigma^2(u+d_\ep)$ and $\meK_{d_\ep,\sigma^2}$ denotes its invariant distribution}.
\end{lemma}
Then, the proof of item \eqref{it:eigen} ensures that
\[\mathsf{P}^{(d_\ep,\sigma^2)}_k(n)=\rpn\Gamma(d_\ep+1)\sum_{r=0}^k (-1)^r\dbinom{k}{r}\frac{\p_r(n)}{\Gamma(r+d_\ep+1)}\] is an eigenfunction of $\eK^{(d_\ep,\sigma^2)}_t$ corresponding to the eigenvalue $e^{-kt}$. Therefore, using Lemma~\ref{lem:d_phi_int}, we have that
 $\mathds{I}_{\phi_{d_\ep}}\mathsf{P}^{(d_\ep,\sigma^2)}_k$ is an eigenfunction of $\eK^{\phi}_t$ corresponding to the eigenvalue $e^{-kt}$, and, in fact,
\begin{align*}
	\mathds{I}_{\phi_{d_\ep}}\eigL^{(d_\ep,\sigma^2)}(n)=&\rpn\Gamma(d_\ep+1)\sum_{r=0}^{k}   { k \choose r} \frac{(-1)^r\mathds{I}_{\phi_{d_\ep}}\mathsf{p}_r(n)}{\Gamma(r+d_\ep+1)}
	\\=&\rpn\sum_{r=0}^{k}  { k \choose r} \frac{(-1)^r\mathsf{p}_r(n)}{W_{\phi}(r+1)}
	=\eigLd (n).
\end{align*}
Since the sequence $ \left(\sqrt{\mf{c}_{k}(d_\ep)}\eigL^{(d_\ep,\sigma^2)}\right)_{k\geq 0}$ is an orthonormal sequence in $\ell^2(\meK_{d_\ep,\sigma^2})$, see \cite[equation (7)]{Karlin-McG} or Example~\ref{ex:1}, and $\mathds{I}_{\phi_{d_\ep}}$ is bounded, we deduce that $ \left(\sqrt{\mf{c}_{k}(d_\ep)}\eigLd\right)_{k\geq 0}$ is a Bessel sequence in $\lnu$ and $\|\eigLd\|_{\meKp}\leq\frac{1}{\sqrt{\mf{c}_{k}(d_\ep)}}$. Letting $\ep\downarrow 0$, the proof of \eqref{it:eignorm2} follows.
\end{proof}
\begin{proposition}\hlabel{prop:eigendual}
	Let $\phi\in\Be$, and,  for $k\in\ZZ_+$, $\eigLdd_k$ be defined as in \eqref{eq:coeig}. Then, the following holds.
\begin{enumerate}	
\item\hlabel{it:coeig1} For all $k\in\ZZ_+$, $\eigLdd_k\in\ell^2(\meKp)$ and, for all $t\ge 0$, \[\wi{\eK}^\phi_t\eigLdd_k=e^{-kt}\eigLdd_k.\]
\item \hlabel{it:biortho1} For all $k,l\in\ZZ_+$,
\[\left\langle\eigLd,\eigLdd_k\right\rangle_{\meKp}=\bbm{1}_{\{k=l\}}.\]
%\item \hlabel{it:denseco} If $\sigma^2>0$, the $\overline{\Span}\left(\eigLdd_k; k\ge 0\right)= \ell^2(\meKp)$ and hence it is an exact sequence.
\item\hlabel{it:coeig1_bessel} If $\sigma^2>0$ and $\ovl{\ovl{\Pi}}(0)<\infty$, then $\left(\rpn\frac{\eigLdd_k}{\sqrt{\mf{c}_k(\mathrm{m}_\phi)}}\right)_{k\ge 0}$ is a Bessel sequence in $\ell^2(\meKp)$ where we recall that $\mathrm{m}_\phi=\frac{m+\ovl{\ovl{\Pi}}(0)}{\sigma^2}$ and $\mf{c}_k(\mathrm{m}_\phi)=\frac{\Gamma(k+\mathrm{m}_\phi+1)}{\Gamma(\mathrm{m}_\phi+1)\Gamma(k+1)}$.
\end{enumerate}
\end{proposition}
%\begin{remark}
%By the same arguments used to prove \eqref{it:denseco},  we point out that one can get the density of the co-eigenfunctions in $\ell^2(\meKp)$ for a larger class of $\phi$, and we refer the interested reader to \cite[Theorem 1.5.1(3)]{Patie-Savov-GeL}.
%\end{remark}
\begin{proof}
Let us recall that for $\phi\in\Be$ defined by
	\[\phi(u)=m+\sigma^2u+\int_0^\infty (1-e^{-uy})\ovl{\Pi}(y)dy,\] one has
	\[\phi(\infty)=\lim_{u\to \infty}\phi(u)=\infty \mathbbm{1}_{\{\sigma^2>0\}}+\left(\overline{\overline{\Pi}}(0)+m\right)\mathbbm{1}_{\{\sigma^2=0\}} \]
	where $0 \leq \overline{\overline{\Pi}}(0)=\int_0^\infty\ovl{\Pi}(y)dy$.
	Finally, let
$ \mathbf{k}_{\phi}=\infty \mathbbm{1}_{\{\sigma^2>0\}}+\frac{\overline{\Pi}(0)}{2\phi(\infty)},$
	and define the set
	\begin{align}\hlabel{eq:Z_phi}
		\ZZ_{\phi}=\begin{cases}\ZZ_+ \qquad &\textrm{ if } \mathbf{k}_{\phi}=\infty\\
			\{k\in \ZZ_+; \: k< \mathbf{k}_{\phi}\} \qquad &\textrm{ otherwise. } \end{cases}
	\end{align}
	When both $\ovl{\Pi}(0)=\infty$ and $\phi(\infty)=\infty$, we have set $\frac{\ovl{\Pi}(0)}{\phi(\infty)}=\infty$. Also, the condition \eqref{eq:levy} on $\Pi$ implies that
	\[\int_0^\infty(1\wedge y)\ovl{\Pi}(y)dy<\infty\]
	and as a consequence, $\ovl{\ovl{\Pi}}(0)<\infty$ whenever $\ovl{\Pi}(0)<\infty$. Thus, $\mathbf{k}_\phi<\infty$ only when $\sigma^2=0$ and $\ovl{\Pi}(0)<\infty$. It is shown in \cite[Theorem 5.2]{Patie-Savov-GeL}, that $ \nu_{\phi} \in \esC_0^{\lfloor  2 \mathbf{k}_\phi \rfloor -1}(\R_+)$ and in \cite[Theorem 1.11]{Patie-Savov-GeL} that, for any $k\in  \ZZ_{\phi}$,  $  \mathrm{V}^\phi_k \in \Lnu$, where \[ \mathrm{V}^\phi_k(x)=\frac{\rppp}{k!}\frac{\frac{d^k}{dx^k}(x^{k} \nu_{\phi}(x))}{\nu_{\phi}(x)}.\]
	We now assume that $k\in  \mathtt \ZZ_{\phi}$ and recall  from \cite[Theorem 8.1]{Patie-Savov-GeL}, that, for all $t>0$, 	\begin{align*}
		\SLcpd_t   \mathrm{V}^\phi_k = e^{-k t}    \mathrm{V}^\phi_k.
	\end{align*}
	Now, the intertwining relationship \eqref{eq:adjoint_int} entails that, for any $t>0$,
	\[\SLpd_t \Ladp   \mathrm{V}^\phi_k =\Ladp \SLcpd_t   \mathrm{V}^\phi_k =e^{-kt}\Ladp    \mathrm{V}^\phi_k.\]
	Let us now characterize the quantity $\wi{\Lambda}_\phi \mathrm{V}^\phi_k$ when $k\in\ZZ_\phi$. From \eqref{eq:Lambda_tilde} it can be easily checked that $\Ladp  \mathrm{V}^\phi_0(n)=\Ladp \mathbf{1}(n)=\mathbf{1}$. Writing $\varrho_1=\frac{1}{2}\log(1+\sigma^{-1}_1)$, for any $n,k\in \N$ we have
	\begin{eqnarray}
		e^{-k\varrho_1}\Ladp\mathrm{V}^\phi_k(n)&=&\frac{1}{n!\meKp(n)}\int_{0}^{\infty} e^{-x}x^n\frac{\frac{d^k}{dx^k}(x^{k} \nu_{\phi}(x))}{k!}dx \nonumber  \\
		&=&\frac{(-1)^k}{n!\meKp(n)k!}\int_{0}^{\infty} \frac{d^k}{dx^k}(e^{-x}x^n)x^{k} \nu_{\phi}(x)dx\hlabel{eq:mellin}   \\
		&=&\frac{(-1)^k}{n!\meKp(n)k!}\sum_{j=0}^{k\wedge n}(-1)^{k-j}{ k \choose j}\frac{n!}{(n-j)!}\int_{0}^{\infty} e^{-x} x^{k+n-j} \nu_{\phi}(x)dx \nonumber  \\
		&=&\frac{1}{\meKp(n)}\sum_{j=0}^{k\wedge n}(-1)^{j}\frac{(k+n-j)!}{(k-j)!(n-j)!j!}\meKp(k+n-j) \nonumber \\
		&=& e^{-k\varrho_1}\eigLdd_k(n) \hlabel{eq:LVv} \hlabel{eq:V=eigldd}%\\\mathfrak{R}^{(k)}_{\meKp} \meKp(n),
	\end{eqnarray}
	where we used, for the second identity, the fact that, for all $j=1,\ldots, k$,
	\[\lim_{x\to 0, \phi(\infty)}\frac{d^{k-j}}{dx^{k-j}}(x^{k} \nu_{\phi}(x))\frac{d^{j}}{dx^{j}}(e^{-x}x^n)=0. \]
	Indeed, these asymptotic behaviors are deduced easily from \cite[Lemma 5.22]{Patie-Savov-GeL}, which states  that for any $x>0$, $0\le j\le k$ and $a<\mathrm{d}_{\phi}$, $\frac{d^{k-j}}{dx^{k-j}}(x^{k} \nu_{\phi}(x)) \leq C x^{j+a}$ for some constant $C>0$. Since $\Ladp: \Lnu  \mapsto \lnu$ is a bounded linear operator, see Proposition~\ref{prop:int_laguerre}\eqref{it:l2_int}, we have that $\eigLdd_k=\Ladp  \mathrm{V}^\phi_k \in \lnu$ and this concludes the proof of \eqref{it:coeig1} when $k\in\ZZ_\phi$. Now, let $\sigma^2>0$. Then, for any $k,l\in \ZZ_+$, we have,  from Propositions \ref{prop:eigen} and the previous computation that both $\eigLd, \eigLdd_l \in \lnu$ and using \eqref{eq:LPp} and \eqref{eq:LVv}, we obtain
	\begin{eqnarray}\hlabel{eq:bioPV}
		\left\langle\eigLd,\eigLdd_l\right\rangle_{\meKp}&=&\left\langle \eigLd,\Ladp\eigdLc_l\right\rangle_{\meKp}=\left\langle  \eigLc_k,\eigdLc_l\right\rangle_{\nu_{\phi}}=\mathbbm{1}_{\{k=l\}}\end{eqnarray}
	where we used that $(\eigLc_k,\eigdLc_k)_{k\geq 0}$ is a biorthogonal sequence in $\Lnu$, see \cite[Theorem 1.22(2)]{Patie-Savov-GeL}, recall that with the notation of this paper, $\eigLc_k=\rpn\mathcal{P}_k$ and $\eigdLc_k=\rpp\mathcal{V}_k$. This proves \eqref{it:biortho1} when $\sigma^2>0$.
	
	Next, assume that $k\notin\ZZ_{\phi}$ which implies that $\mathbf{k}_{\phi}<\infty$ and thus $\phi(\infty)<\infty$. This entails that the following two-sided bounds hold for any $n \in \ZZ_+$
	\begin{equation}
		e^{-\phi(\infty)} W_{\phi}(n+1) \leq \int_{0}^{\phi(\infty)} e^{-x}x^n   \nu_{\phi}(x)dx \leq W_{\phi}(n+1)\leq  \phi(\infty)^n
	\end{equation}
	where the last inequality follows since $\phi$ is non-decreasing.
	Thus, we have
	\begin{eqnarray}
		e^{-\phi(\infty)}\frac{ W_{\phi}(n+1)}{n!}\leq \meKp (n) \leq \frac{ W_{\phi}(n+1)}{n!}\leq \frac{\phi(\infty)^n}{n!}.
	\end{eqnarray}
	Hence, for any $k\in \ZZ_+$ fixed,  with \[S(k)=\sum_{n=0}^{k} \frac{1}{\meKp(n)}\left(\sum_{j=0}^{n}(-1)^{j}\frac{(k+n-j)!}{(k-j)!(n-j)!j!}\meKp(k+n-j)\right)^2<\infty,\] we have
	\begin{align}
		 e^{-2k{\varrho}_1}\|\eigLdd_k\|^2_{\lnu} &= S(k)+ \sum_{n=k}^{\infty} \frac{1}{\meKp(n)}\left(\sum_{j=0}^{k}(-1)^{j}\frac{(k+n-j)!}{(k-j)!(n-j)!j!}\meKp(k+n-j)\right)^2 \nonumber\\
		&\leq  S(k)+ \sum_{n=k}^{\infty} \frac{e^{\phi(\infty)}n!}{W_{\phi}(n+1)}\left(\sum_{j=0}^{k}\frac{(k+n-j)!}{(k-j)!(n-j)!j!}\frac{ \phi(\infty)^{k+n-j}}{(k+n-j)!}\right)^2 \nonumber \\
		&\leq S(k)+ e^{\phi(\infty)}\sum_{n=k}^{\infty} \frac{ n! \phi(\infty)^{2n}}{W_{\phi}(n+1)((n-k)!)^2}\left(\sum_{j=0}^{k}\frac{\phi(\infty)^{k-j}}{(k-j)!j!}\right)^2 \nonumber \\
		&\leq S(k)+ e^{\phi(\infty)}\left(\sum_{j=0}^{k}\frac{\phi(\infty)^{k-j}}{(k-j)!j!}\right)^2 \sum_{n=k}^{\infty} \frac{n! \phi(\infty)^{2n}}{W_{\phi}(n+1)((n-k)!)^2}<\infty \hlabel{eigLdd_bd}
	\end{align}
	where the last inequality follows after observing that,
	\[\lim_{n \to \infty}\frac{a_{n+1}}{a_n}=\lim_{n \to \infty}\frac{(n+1) \phi(\infty)^2}{\phi(n+1)(n+1-k)^2}=0\]
	where the $a_n$'s are the coefficient of the last series.
  When $\phi\in\Be$ is such that $
		\phi(u)=m+\int_0^\infty(1-e^{-uy})\overline{\Pi}(y) dy$,
	let us define $\phi_\ep\in\Be$ as $\phi_\epsilon(u)=\epsilon u+\phi(u)$ with $\ep> 0$. Then, from Proposition~\ref{prop:n_phi}, it follows that for small values of $\ep$ and for all $n\in\ZZ_+$,
	\begin{align*}
		\meK_{\phi_\ep}(n)=\frac{1}{n!}\sum_{r=0}^\infty (-1)^r\frac{W_{\phi_\ep}(n+r+1)}{r!}, \\
		\meKp(n)=\frac{1}{n!}\sum_{r=0}^\infty (-1)^r\frac{W_{\phi}(n+r+1)}{r!}.
	\end{align*}
	As $\phi_\ep(u)\ge \phi(u)$ for all $\ep$ and $u\ge 0$, it follows that $W_{\phi_\ep}(n)\ge W_{\phi}(n)$ for all $n\in\bb{N}$. Also, $W_{\phi_\ep}\downarrow W_{\phi}$ pointwise as $\ep\to 0$. Since, for small values of $\ep$ (e.g.~$0\le\ep<1$),
	\begin{align*}
		\sum_{r=0}^\infty\frac{W_{\phi_\ep}(r+n+1)}{r!}<\infty,
	\end{align*}
	the dominated convergence theorem yields the following pointwise convergence as $\ep\to 0$,
	\begin{align}
		\meK_{\phi_\ep}\to\meK_\phi.
	\end{align}
	 Hence, for any $j,k\in\ZZ_+$,
	\begin{align}
		\lim_{\ep\to 0}\langle\mathsf{V}^{\phi_\ep}_k,\mathsf{p}_j\rangle_{\mathbf{n}_{\phi_\ep}}=&\lim_{\ep\to 0}\sum_{n=0}^\infty\mathsf{p}_j(n)\mathsf{V}^{\phi_\ep}_k(n)\meK_{\phi_\ep}(n) \\ =&\lim_{\ep\to 0}\sum_{n=0}^\infty\mathsf{p}_j(n)\sum_{j=0}^{k\wedge n} (-1)^{k-j}\frac{(k+n-j)!}{(k-j)! (n-j)! j!}\meK_{\phi_\ep}(k+n-j) \nonumber \\
		=&\lim_{\ep\to 0}\sum_{n=0}^k\mathsf{p}_j(n)\sum_{j=0}^n (-1)^{k-j}\frac{(k+n-j)!}{(k-j)! (n-j)! j!}\meK_{\phi_\ep}(k+n-j) \hlabel{phi_ep} \\
		& \ + \lim_{\ep\to 0}\sum_{n=k}^\infty\mathsf{p}_j(n)\sum_{j=0}^k (-1)^{k-j}\frac{(k+n-j)!}{(k-j)! (n-j)! j!}\meK_{\phi_\ep}(k+n-j). \nonumber
	\end{align}
	In \eqref{phi_ep}, the first term is a finite sum and therefore
	\begin{align*}
		\lim_{\ep\to 0}&\sum_{n=0}^k\mathsf{p}_j(n)\sum_{j=0}^n (-1)^{k-j}\frac{(k+n-j)!}{(k-j)! (n-j)! j!}\meK_{\phi_\ep}(k+n-j) \\
		=&\sum_{n=0}^k\mathsf{p}_j(n)\sum_{j=0}^n (-1)^{k-j}\frac{(k+n-j)!}{(k-j)! (n-j)! j!}\meKp(k+n-j).
	\end{align*}
	For the second term in \eqref{phi_ep}, we have
	\begin{align*}
		&\sum_{n=k}^\infty\mathsf{p}_j(n)\sum_{j=0}^k (-1)^{k-j}\frac{(k+n-j)!}{(k-j)! (n-j)! j!}\meK_{\phi_\ep}(k+n-j) \\
		=&\sum_{j=0}^k (-1)^{k-j}\frac{1}{j! (k-j)!}\sum_{n=k}^\infty \mathsf{p}_j(n)\frac{(k+n-j)!}{(n-j)!}\meK_{\phi_\ep}(k+n-j).
	\end{align*}
	Since $\meK_{\phi_\ep}\to\meKp$ pointwise as $\ep\to 0$, the distribution $\meK_{\phi_\ep}$ converges to $\meKp$ weakly. Also, for any $k\in\ZZ_+$, as $\ep\to 0$,
	\begin{align*}
		\sum_{n=0}^\infty\mathsf{p}_k(n)\meK_{\phi_\ep}(n)=W_{\phi_\ep}(k+1)\to W_{\phi}(k+1)=\sum_{n=0}^\infty\mathsf{p}_k(n)\meKp(n).
	\end{align*}
	Applying \eqref{eq:stirling_1} on the previous identity we obtain that, for all $k\in\ZZ_+$, as $\ep\to 0$,
	\begin{align}\hlabel{eq:moment_nphi}
		\sum_{n=0}^\infty n^k\meK_{\phi_\ep}(n)\to\sum_{n=0}^\infty n^k\meKp(n).
	\end{align}
	Since, for any $k\in\ZZ_+$ and $j\le k$,
	\begin{align*}
		\frac{(k+n-j)!}{(n-j)!}=\mathrm O(n^k)
	\end{align*}
	uniformly with respect to $j$,
	using a dominated convergence theorem one can show that for each $j\le k$,
	\begin{align*}
		\lim_{\ep\to 0}\sum_{n=k}^\infty \mathsf{p}_j(n)\frac{(k+n-j)!}{(n-j)!}\meK_{\phi_\ep}(k+n-j)=\sum_{n=k}^\infty \mathsf{p}_j(n)\frac{(k+n-j)!}{(n-j)!}\meKp(k+n-j).
	\end{align*}
	Thus, \eqref{phi_ep} yields
	\begin{align}\hlabel{eq:moment_lim}
		\lim_{\ep\to 0}\left\langle\mathsf{V}^{\phi_\ep}_k,\mathsf{p}_j\right\rangle_{\meK_{\phi_\ep}}=\left\langle\mathsf{V}^{\phi}_k,\mathsf{p}_j\right\rangle_{\meKp}.
	\end{align}
 Now, if $\sigma^2=0$, since, for any $k\in\mathbb Z_+$, $\eigLd\in\cP=\Span\{\p_j, j\in\mathbb Z_+\}$ and the coefficient of $\p_j$ in $\mathsf{P}^{\phi_\ep}_k$ converges to the same in $\eigLd$ for all $j\in\ZZ_+$, as $\ep\to 0$, applying \eqref{eq:moment_lim} it follows that, for all $k,l\in\mathbb Z_+$,
\begin{align*}
	\bbm{1}_{\{k=l\}}=\lim_{\ep\to 0}\left\langle \mathsf{P}^{\phi_\ep}_k, \mathsf{V}^{\phi_\ep}_l\right\rangle_{\meK_{\phi_\ep}}=\left\langle\eigLd,\eigLdd_l \right\rangle_\meKp
\end{align*}
where $\phi_\ep(z)=\ep u+\phi(u)$. This proves \eqref{it:biortho1} for all $\sigma^2\ge 0$ hence for all $\phi\in\Be$.

	To show that $\eigLdd_k$ is a co-eigenfunction of $\eK^\phi$ when $\sigma^2=0$, we proceed as follows. Proposition~\ref{prop:eigen}\eqref{it:eigen} and
\eqref{eq:bioPV} yield that, for $l,k\in\ZZ_+$, and $t>0$,

	\begin{align*}
		\left\langle \wi{\mathds{K}}^\phi_t\eigLdd_k,\mathsf{P}^\phi_l\right\rangle_\meKp=&\left\langle \eigLdd_k, \eK^\phi_t\mathsf{P}^\phi_l\right\rangle_\meKp
		=e^{-tk}\left\langle\eigLdd_k,\mathsf{P}^\phi_l\right\rangle_\meKp=e^{-tk} \mathbbm{1}_{\{k=l\}}.
	\end{align*}
Therefore, for all $\phi\in\Be$, $t>0$ and $k,l\in\ZZ_+$, we get
	\begin{align*}
		\left\langle \wi{\mathds{K}}^\phi_t\eigLdd_k- e^{-tk}\eigLdd_k,\mathsf{P}^\phi_l\right\rangle_\meKp=& 0.
	\end{align*}
Since $(\eigLd)_{k\ge 0}$ is dense in $\ell^2(\meKp)$, we deduce that, for all $t\ge 0$ and $k\in\ZZ_+$,
	\begin{align}
		e^{tk}\eKad^\phi_t\eigLdd_k=\eigLdd_k,
	\end{align}
which proves \eqref{it:coeig1} for all $\phi\in\Be$.
%For the item \eqref{it:denseco}, since, for any $k\ge 0$, $\wi{\Lambda}_\phi \mathrm{V}^\phi_k=\eigLdd_k$, see \eqref{eq:V=eigldd}, $\overline{\Span}\{\mathrm{V}^\phi_k; k\geq0\}=\ell^2(\meKp)$, see \cite[Theorem 1.5.1(3)]{Patie-Savov-GeL}, and $\wi{\Lambda}_\phi$ is a bounded operator with a dense range, we obtain the first claim, the second one is immediate.

To prove item \eqref{it:coeig1_bessel}, it is known from \cite[Theorem 10.1(1)]{Patie-Savov-GeL} (after multiplying by the factor $\rpn$) that, when $\sigma^2>0$ and $\ovl{\ovl{\Pi}}(0)<\infty$,   $\left(\rpn\frac{\mathrm{V}^\phi_k}{\sqrt{\mf{c}_k(\mathbf{m}_\phi)}}\right)_{k\ge 0}$ is a Bessel sequence in $\bmrm{L}(\nu_\phi)$.
Recalling that, for any $k\ge 0$, $\wi{\Lambda}_\phi \mathrm{V}^\phi_k=\eigLdd_k$, see \eqref{eq:V=eigldd},  and $\wi{\Lambda}_\phi:\bmrm{L}(\nu_\phi)\to\ell^2(\meKp)$ is a contraction, we conclude that $\left(\rpn\frac{\eigLdd_k}{\sqrt{\mf{c}_k(\mathbf{m}_\phi)}}\right)_{k\ge 0}$ is a Bessel sequence in $\ell^2(\meKp)$ and for all $k\in\ZZ_+$, \[\|\eigLdd_k\|_{\ell^2(\meKp)}\le\rpp\sqrt{\mf{c}_k(\mathrm{m}_\phi)}.\]

\end{proof}
	
	\subsection{Proof of Theorem~\ref{thm:eig_coeig}}\hlabel{ss:eig_coeig_pf} The proof of the item \eqref{it:spect} and \eqref{it:biortho} follows directly from Proposition~\ref{prop:eigen}\eqref{it:eigen} and Proposition~\ref{prop:eigendual}\eqref{it:coeig1},\eqref{it:biortho1}.
	
		Finally for the proof of \eqref{it:spectral_exp}, we now assume that $\sigma^2>0$. We recall from \eqref{eq:tau} that $\sigma_1=\sigma^2$ in this case.
	Then, for all $\mathsf{f}\in\ell^2(\meK_{\sigma^2})$ and $t>0$, the intertwining relation \eqref{eq:inter-bdsf1} yields that
	\begin{eqnarray}\hlabel{eq:spec_exp}
		\SLp_t \Mphi \mathsf{f} &= & \Mphi \eK^{\sigma^2}_{t}  \mathsf{f} \nonumber \\
		&=& \Mphi \sum_{k=0}^{\infty}e^{-kt} \left\langle \mathsf{f}, \eigL^{\sigma^2} \right\rangle_{\meK_{\sigma^2}} \eigL^{\sigma^2} \nonumber \\
		&= &  \sum_{k=0}^{\infty}e^{-kt} \left\langle \mathsf{f}, \eigL^{\sigma^2} \right\rangle_{\meK_{\sigma^2}} \eigLd \hlabel{eq:bessel_exp}
	\end{eqnarray}
	where the second identity relies on the spectral decomposition of the reversible birth-death chain, see Example~\ref{ex:1} with $\phi(u)=\sigma^2 u$, whereas the last one is justified as follows. First, since $\sigma^2>0$, $\Mphi: \ell^2(\meK_{\sigma^2}) \mapsto \lnu $ is a bounded linear operator, and, with the help of Lemma~\ref{lem:Mphi} and the definition of $\eigLd$ in \eqref{eq:def_e}, it follows that $\Mphi \eigL^{\sigma^2} =\eigLd$. Moreover, from Proposition \ref{prop:eigen}, we have that the sequence $(\eigLd)_{k\geq 0}$ is a Bessel sequence and thus its associated synthesis operator $\mathcal{S}: \ell^2(\ZZ_+) \mapsto \lnu$, see \eqref{eq:def_syn} for definition, is bounded. Since $\left(\eigL^{\sigma^2}\right)_{k\geq0}$ is an orthonormal sequence in $\ell^2(\meK_{\sigma^2})$, it implies that for all $t\geq 0$,
	\begin{equation*}
	\left(e^{-kt}\left\langle \mathsf{f}, \eigL^{\sigma^2} \right\rangle_{\meK_{\sigma^2}}\right)_{k\geq 0} \in \ell^2(\ZZ_+)
	\end{equation*}
	and hence the series on the right-hand side of \eqref{eq:bessel_exp} is in $\lnu$. Next, as noted before, we have that $\Mphi\eigL^{\sigma^2}=\eigLd$ for all $k\in\ZZ_+$. Now, recalling that $(\eigLd,\eigLdd_k)_{k\ge 0}$ is biorthogonal in $\lnu$, see Proposition~\ref{prop:eigendual}\eqref{it:biortho1}, we have for any $l,k\in\ZZ_+$,
	\begin{align*}
		\left\langle\mathsf{P}^{\sigma^2}_l,\Mphiad\eigLdd_k\right\rangle_{\meK_{\sigma^2}}=\left\langle\Mphi\mathsf{P}^{\sigma^2}_l,\eigLdd_k\right\rangle_{\meKp}=\left\langle\mathsf{P}^\phi_l,\eigLdd_k\right\rangle_{\meKp}=\mathbbm{1}_{\{k=l\}}
	\end{align*}
	As $(\eigL^{\sigma^2})_{k\ge 0}$ is orthonormal in $\ell^2(\meK_{\sigma^2})$ (hence biorthogonal to iteself), by uniqueness of biorthogonal sequence we conclude that $\Mphiad\eigLdd_k=\eigL^{\sigma^2}$ for all $k\in\ZZ_+$.
	%Next,  since $\Mphi\in\mathscr{B}(\ell^2(\meK_\r),\lnu)$ with dense range, it is immediate that $\Mphiad$ is an injective operator in $\mathscr{B}(\lnu,\ell^2(\meK_\r))$. Now invoking the fact that $\eigL^\r$ is the unique (up to normalization) eigenfunction(\tred{ref}) of $\eK^\r_t$ for the eigenvalue $e^{-kt}$, the adjoint of the intertwining identity \eqref{eq:inter-bdsf1} combined with the fact $\wi{\eK}^\phi_t\eigLdd_k=e^{-kt}\eigLdd_k$ yields that for all $k\ge 0$,
	%\[\Mphiad \eigLdd_k=c\eigL^\r\] for some constant $c$. Noting that $\eigL^\r(0)=\Mphi..$ \tred{complete here..}since $\Mphi$ is injective with dense range it admits a densely defined inverse $\Mphi^{-1} : \lnu \mapsto \ell^2(\meK_\r)$, which also ensures the existence of an injective and densely defined adjoint operator $\Mphiad^{-1}:\ell^2(\meK_\r) \mapsto \lnu$.
	Therefore, writing $\mathsf{g} = \Mphi \mathsf{f}\in \lnu $, we have, for all $k\in \ZZ_+$,
	\begin{eqnarray*}
		\left\langle \mathsf{g}, \eigLdd_k \right\rangle_\meKp=\left\langle \mathsf{f}, \Mphiad\eigLdd_k\right\rangle_{\meK_{\sigma^2}}=\left\langle\f,\eigL^{\sigma^2}\right\rangle_{\meK_{\sigma^2}}
		%&=& \left\langle \Mphi^{-1}\mathsf{g}, \eigL^{\r} \right\rangle_{\meK_\r} =\left\langle \mathsf{g}, \Mphiad^{-1} \eigL^{(1,\r)} \right\rangle_\meKp =
	\end{eqnarray*}
	%where we used for the last identity the second intertwining relation in \eqref{eq:inter-bdsf1} which, by a standard argument, yields that $\Mphiad \eigLdd_k = \eigL^{(1,\r)}$.
	Thus, from \eqref{eq:spec_exp}, for $\mathsf{g}\in \textrm{Ran}(\Mphi)$, the range of $\Mphi$, one gets
	\begin{eqnarray*}\hlabel{eq:spec-exp}
		\SLp_t  \mathsf{g}
		&= &  \sum_{k=0}^{\infty}e^{-kt} \left\langle \mathsf{g}, \eigLdd_k \right\rangle_\meKp \!\!\!\eigLd= \mathds{S}_t\mathsf{g}
	\end{eqnarray*}
	where the last identity serves as defining the spectral operator. Note that since $\langle \SLp_t\mathsf{g}, \eigLdd_k \rangle_\meKp=\langle \mathsf{g}, \SLpd_t\eigLdd_k \rangle_\meKp=e^{-kt} \langle \mathsf{g}, \eigLdd_k \rangle_\meKp$, we deduce that
	\[\mathds{S}_t\mathsf{g}=\sum_{k=0}^{\infty}\left\langle \SLp_t\mathsf{g}, \eigLdd_k \right\rangle_\meKp\!\!\!\eigLd. \]
	Moreover,  as the closure (in $\lnu$) of $\textrm{Ran}(\Mphi)$ is $\lnu$, by the bounded linear transformation theorem, $\SLp_t$  is the unique continuous extension of the continuous operator $\mathds{S}_t: \textrm{Ran}(\Mphi)\mapsto \lnu$. We now extend the domain of $\mathds{S}_t$ to $\lnu$. First, by means of Cauchy-Schwartz inequality, we have, for any $\mathsf{g} \in \lnu$ and  $k\in \N$,
	\begin{eqnarray*}
		\left|\langle \mathsf{g}, \eigLdd_k \rangle_\meKp\right|\leq \|\mathsf{g}\|_{\ell^2(\meKp)} \: \left\|\eigLdd_k \right\|_{\ell^2(\meKp)}&=&\|\mathsf{g}\|_{\ell^2(\meKp)}\|\wi{\Lambda}_\phi \mathrm{V}^\phi_k\|_{\ell^2(\meKp)}\\ & \leq &\|\mathsf{g}\|_{\ell^2(\meKp)} \: \left\|  \mathrm{V}^\phi_k\right\|_{\Lnu}
	\end{eqnarray*}
	where we used Proposition \ref{prop:eigendual} and the fact that $\wi{\Lambda}_\phi$ is a bounded operator. Next, since from \cite[Theorem 10.1]{Patie-Savov-GeL}, we have for  $k$ large enough and all $\epsilon>0$, $\|  \mathrm{V}^\phi_k\|_{\Lnu}\leq C_\epsilon\rpp e^{\epsilon k}$, with $C_\epsilon>0$, this implies that for all $\mathsf{g} \in \lnu$ and $t>\lnt$,
	\[\left(e^{-kt} \langle \mathsf{g}, \mathrm{V}^\phi_k \rangle_\meKp\right)_{k\geq 0} \in \ell^2(\ZZ_+).\]
	Finally, the Bessel property of the sequence $(\eigLd)_{k\geq 0}$ entails that $\mathds{S}_t\mathsf{g} \in  \lnu$, which completes the proof.

For the item \eqref{it:spect_compact}, we recall from \cite[Theorem~10.1]{Patie-Savov-GeL} and the proof of Theorem~\ref{thm:eig_coeig}\eqref{it:spectral_exp} that, for all $\epsilon>0$ and $k\in\ZZ_+$, there exists $C_\ep>0$ such that
\begin{align}\hlabel{eq:exp_bound_coeig}
\left\|\eigLdd_k\right\|_{\lnp}\le\left\|\mathrm{V}^\phi_k\right\|_{\bmrm{L}(\nu_\phi)}\le C_\ep\rpp e^{\epsilon k}
\end{align}
whenever $\sigma^2>0$. Moreover, for all $k\in\ZZ_+$, $\|\eigLd\|_{\lnp}\le 1$. Therefore, \eqref{eq:spect_exp} entails that for all $t>\lnt$, the operator $\eK^\phi_t$ can be approximated by the sequence of finite dimensional operators \[\f\mapsto\sum_{k=0}^N e^{-kt}\left\langle\f,\eigLdd_k\right\rangle_{\meKp}\!\!\!\eigLd, \ N\geq 1,\] which proves the compactness of the semigroup.
Finally, for $l\in\ZZ_+$, let us choose $\f=\delta_l$ in \eqref{eq:spect_exp}. Then, for all $\sigma^2>0, t>0$ and $n\in\ZZ_+$, we have
\begin{align}
	\eK^\phi_t(n,l)=\eK^\phi_t\delta_l(n)=\sum_{k=0}^\infty e^{-kt}\eigLd(n)\eigLdd_k(l)\meKp(l)
\end{align}
where the last identity holds in $\ell^2(\meKp)$. Now, from Proposition~\ref{prop:eigen}\eqref{it:eignorm}, we have that, for all $k,n\in\ZZ_+$,
\begin{align*}
	\eigLd(n)^2\meKp(n)\le\|\eigLd\|^2_{\ell^2(\meKp)}\le 1
\end{align*}
while, from Jensen's inequality and \eqref{eq:exp_bound_coeig}, we get, for all $k,l\in\ZZ_+$, that there exists a uniform constant $C_\ep>0$ such that for all $\epsilon>0$,
\begin{align}
	|\eigLdd_k(l)\meKp(l)|\le\|\eigLdd_k\|_{\ell^1(\meKp)}\le\|\eigLdd_k\|_{\ell^2(\meKp)}\le C_\ep\rpp e^{\epsilon k}.
\end{align}
Since $\meKp(n)>0$ for all $n\in\ZZ_+$, see \eqref{eq:nphi_pos}, for all $\lnt+\ep<t$, we have
\begin{align}
\sum_{k=0}^\infty e^{-kt}|\eigLd(n)||\eigLdd_k(l)|\meKp(l)\le C_\ep\sum_{k=0}^\infty e^{-kt}\frac{e^{\ep k}}{\sqrt{\meKp(n)}}<\infty.
\end{align}
As $\ep>0$ is arbitrary, the proof of the item \eqref{it:lag_transition} is completed.
% This is the proof of the spectral gap
\subsection{Proof of Theorem~\ref{thm:spgap_ent}\eqref{thm:l0_norm}}\hlabel{ss:l0_norm} From \cite[Lemma 10.4]{Patie-Savov-GeL}, we get that \[\mathrm{m}_{\phi}=\lim_{n\to\infty}\frac{\phi(n)-\sigma^2 n}{\sigma^2}=\frac{m+\int_0^{\infty} \Pi(y,\infty)dy}{\sigma^2}>d_\ep=(\mathrm{d}_{\phi}-\epsilon)\mathbbm{1}_{\{d_{\phi}-\epsilon>0\}}.\] Let us write ${\varrho }=\frac{1}{2}\log(1+\sigma^{-2})$. Then, using \eqref{eq:spect_exp} along with the fact that $\mathsf{P}^\phi_{\!0}\equiv 1$, we obtain, for all $\f\in\ell^2_0(\meKp)=\left\{\g\in\ell^2(\meKp); \meKp\g=0\right\}$,
\begin{align}
	\eK^\phi_t\f=&\sum_{k=1}^\infty e^{-kt}\left\langle \f,\eigLdd_k\right\rangle_{\meKp}\eigLd \nonumber\\
	=&\sum_{k=1}^\infty e^{-kt}\sqrt{\frac{\mf{c}_k(\mathrm{m}_\phi)}{\mf{c}_k(\mathrm{d}_\phi)}}e^{k{\varrho}}\left\langle\f,e^{-k{\varrho}}\frac{\eigLdd_k}{\sqrt{\mf{c}_k(\mathrm{m}_\phi)}}\right\rangle_{\meKp}\sqrt{\mf{c}_k(\mathrm{d}_\phi)}\eigLd. \hlabel{eq:bessel}
\end{align}
Since $\left(\sqrt{\mf{c}_k(\mathrm{d}_\phi)}\eigLd\right)_{k\ge 1}$ is a Bessel sequence with bound $1$, we obtain from the boundedness of the synthesis operator, see \eqref{eq:def_syn}, and \eqref{eq:bessel} that, writing $\overline{\eigLdd_k}=\frac{\eigLdd_k}{\sqrt{\mf{c}_k(\mathrm{m}_\phi)}}$, for all $\f\in\ell^2_0(\meKp)$,
\begin{align}
	\|\eK^\phi_t\f\|^2_{\ell^2(\meKp)}\le&\sum_{k=1}^\infty e^{-2kt}\frac{\mf{c}_k(\mathrm{m}_\phi)}{\mf{c}_k(\mathrm{d}_\phi)}e^{2k{\varrho}}\left|\left\langle\f, e^{-k{\varrho}}\overline{\eigLdd_k}\right\rangle_{\meKp}\right|^2 \nonumber\\
	=&e^{-2t}\frac{\mf{c}_1(\mathrm{m}_\phi)}{\mf{c}_1(\mathrm{d}_\phi)}e^{2{\varrho}}\sum_{k=1}^\infty e^{-2(k-1)(t-{\varrho})} \frac{\mf{c}_k(\mathrm{m}_\phi)\mf{c}_1(\mathrm{d}_\phi)}{\mf{c}_k(\mathrm{d}_\phi)\mf{c}_{1}(\mathrm{m}_\phi)}\left|\left\langle\f,e^{-k{\varrho}}\overline{\eigLdd_k}\right\rangle_{\meKp}\right|^2.\hlabel{eq:bessel2}
\end{align} Now, from the proof of \cite[Theorem~1.18(3)]{Patie-Savov-GeL}, we know that
\[\sup_{k\ge 1}e^{-2(k-1)(t-{\varrho})}\frac{\mf{c}_k(\mathrm{m}_\phi)\mf{c}_1(\mathrm{d}_\phi)}{\mf{c}_k(\mathrm{d}_\phi)\mf{c}_1(\mathrm{m}_\phi)}\le 1\iff t> T=\frac{1}{2}\log\left(\frac{\mathrm{m}_\phi+2}{\mathrm{d}_\phi+2}\right)+{\varrho}.\] Thus, using this bound, the fact that $\left(e^{-k{\varrho}}\overline{\eigLdd_k}\right)_{k\ge 1}$ is also a Bessel sequence (with bound $1$) in $\ell^2(\meKp)$ and the second inequality in \eqref{eq:frame1},  we deduce from \eqref{eq:bessel2} that,  for all $\f\in\ell^2_0(\meKp)$ and $t>T$,
\begin{align}
\|\eK^\phi_t\f\|^2_{\ell^2(\meKp)}\le e^{2{\rm t}_\r}\frac{\mf{c}_1(\mathrm{m}_\phi)}{\mf{c}_1(\mathrm{d}_\phi)} e^{-2t}\|\f\|^2_{\ell^2(\meKp)}=\frac{1+\r}{\r}\frac{\mathrm{m}_\phi+1}{\mathrm{d}_\phi+1}e^{-2t}\|\f\|^2_{\ell^2(\meKp)}. \hlabel{eq:K_equib}
\end{align}
When $t\le T$, $\frac{1+\r}{\r}\frac{\mathrm{m}_\phi+1}{\mathrm{d}_\phi+1}e^{-2t}\ge \frac{\mathrm{m}_\phi+1}{\mathrm{d}_\phi+1}\frac{\mathrm{d}_\phi+2}{\mathrm{m}_\phi+2}\ge 1$ as $\mathrm{m}_\phi\ge \mathrm{d}_\phi$. Therefore, \eqref{eq:K_equib} holds for all $t>0$ as $\eK^\phi$ is a contraction semigroup. Finally, noting that, for any $\f\in\ell^2(\meKp)$, $\f-\meKp\f\in\ell^2_0(\meKp)$, the proof of the theorem follows.

\subsection{Interweaving between skip-free and continuous Laguerre semigroups}\hlabel{ss:interweaving}
Following \cite{miclo_patie_2}, for two Markov semigroups $P,P'$ defined on two Banach spaces $\mathbf B, \mathbf B'$ respectively, we say that $P$ has an  \textit{interweaving relation}  with  $P'$ if there exist two Markov kernels $\Lambda:\mathbf{B}'\to\mathbf{B}$ and $\Lambda':\mathbf{B}\to\mathbf{B}'$, and a non-negative  random variable $\tau$ such that
\begin{eqnarray*}
	&P\Lambda=\Lambda P' \ \text{  on  } \ \mathbf B' \\
	&P'\Lambda'=\Lambda' P \ \text{  on  } \ \mathbf{B} \textrm{ and } \\
	&\Lambda\Lambda'=P_\tau=\int_0^\infty P_t\PP(\tau\in dt).
\end{eqnarray*}
We call $\tau$ the \textit{warm-up time} or the \textit{delay} and we  write  $P\looparrowleft P'$  or $P \stackrel{\tau}{\looparrowleft} P'$  to emphasize the dependence on $\tau$. Note that when $\tau=\delta_{{t}_0}$ is the degenerate random variable at  $t_0>0$, we may simply write, when there is no confusion,  $P \stackrel{{t_0}}{\looparrowleft} P'$.

\noindent When $\tau$ is in addition infinitely divisible we say that $P$ admits an \textit{interweaving relation with an infinitely divisible} warm-up time  with  $P'$ and  we  write   $P \stackrel{\tau}{\looparrowleft} P'$.
\noindent Finally, when
we also have
\begin{equation} \hlabel{eq:sym}
	\Lambda\Lambda'=P'_\tau
\end{equation}
we say that  there is a \textit{symmetric  interweaving relation} between $P$ and $P'$ and we write $P \stackrel{\tau}{\leftrightsquigarrow} P'$. We refer to \cite{miclo_patie_2} for a thorough study and several applications of this concept that refines the one of intertwining relations.

Let now $\eK^\phi$ be the skip-free Laguerre semigroup corresponding to the Bernstein function $\phi$ associated with the triplet $(m,\sigma^2,\Pi)$, see \eqref{eq:bernstein_def}, and $K^{\sigma^2}$ be the diffusive Laguerre semigroup with generator
\begin{align}\hlabel{eq:diff_lag_gen}
	L^{\sigma^2}=\sigma^2 x\frac{d^2}{dx^2}+(\sigma^2-x)\frac{d}{dx}.
\end{align}
\begin{theorem}\hlabel{thm:discPQ} If $\sigma^2>0$ and $\ovl{\ovl{\Pi}}(0)=\int_0^\infty\Pi(y,\infty)dy<\infty$, then for all  $\beta>\mathrm{m}_\phi=\frac{m+\ovl{\ovl{\Pi}}(0)}{\sigma^2}$,
	\begin{align*}
		\eK^\phi\overset{\tau_\beta}{\leftrightsquigarrow} K^{\sigma^2}
	\end{align*}
	where $\tau_\beta$ is an infinite divisible random variable characterized, for any  $ u>0$, by
	\begin{eqnarray}\hlabel{eq:tau_beta}
		\int_0^\infty e^{-ut}\bb{P}(\tau_\beta\in dt)&=&\left(\frac{\sigma^2}{1+\sigma^2}\right)^u\frac{\Gamma(1+\beta)\Gamma(u+1)}{\Gamma(u+\beta+1)}\\ &=&\left(\frac{\sigma^2}{1+\sigma^2}\right)^ue^{-\phi_\beta(u)}. \nonumber
	\end{eqnarray}
%	In particular, $\bb{P}(\tau_\beta\in ds)=(1+\beta)(1+\log(s-\log 2))^\beta ds, \ s\in\left(\frac{1}{e}+\log 2,1+\log 2\right)$.
\end{theorem}
Before proving the theorem, let us show the following  lemma.
\begin{lemma}\hlabel{tau_int}
	Let $K^{\sigma^2}$ be the semigroup defined as above and $K$ be the semigroup with generator \[L=x\frac{d^2}{dx^2}+(1-x)\frac{d}{dx}.\] Then, for all $t\ge 0$,
	\[K^{\sigma^2}_td_{\frac{1}{\sigma^2}}=d_{\frac{1}{\sigma^2}}K_t \text{ on } \ \bmrm{L}(\nu)\] where for $\alpha>0$,  $d_\alpha f(x)=f(\alpha x)$ is the dilation operator on $\bb{R}_+$ and $\nu(x)dx= e^{-x}dx, x>0,$ is the unique invariant distribution of the semigroup $K$.
\end{lemma}
\begin{proof}
	It can be easily checked that if $f$ is a polynomial, then
	\begin{align}\hlabel{eq:tau_int}
	L^{\sigma^2} d_{\frac{1}{\sigma^2}}f=d_{\frac{1}{\sigma^2}} Lf.
	\end{align}
	Next, we recall from \cite[Theorem~1.6(3)]{Patie-Savov-GeL} that the set of all polynomials form a core for $L$ in $\bmrm{L}(\nu)$ and $d_{\frac{1}{\sigma^2}}\nu_{\sigma^2}=\nu$ where $\nu_{\sigma^2}$ is the invariant distribution of $K^{\sigma^2}$. Since \[d_{\frac{1}{\sigma^2}}:\bmrm{L}(\nu)\to\bmrm{L}(\nu_{\sigma^2})\] is an invertible operator, \eqref{eq:tau_int} extends at the level of the corresponding semigroups, which proves the lemma.
\end{proof}

\subsection{Proof of Theorem~\ref{thm:discPQ}} Let $K^{(\beta)}$ be the self-adjoint Laguerre semigroup with the generator \[L^{(\beta)}= x\frac{d^2}{dx^2}+(1+\beta-x)\frac{d}{dx}.\]
From \cite[Proposition~26]{miclo_patie_2}, it is known that, for all $\beta>\mathrm{m}_\phi$,
\begin{align*}
K^\phi\overset{\tau^{(\beta)}}{\leftrightsquigarrow} K^{(\beta)}
\end{align*}
where $\tau^{(\beta)}$ is an infinitely divisible random variable with Laplace transform given by
\[\int_0^\infty e^{-us}\bb{P}(\tau^{(\beta)}\in ds)=\frac{\Gamma(1+\beta)\Gamma(u+1)}{\Gamma(u+\beta+1)}, \ \ \ u>0.\]
More precisely, for all $t\ge 0$,
\begin{eqnarray*}
&K^\phi_t \mathrm{I}_\phi \wi{\mathrm{B}}_\beta=\mathrm{I}_\phi \wi{\mathrm{B}}_\beta K^{(\beta)}_t \ &\text{ on } \ \bmrm{L}(\nu_\beta) \nonumber \\
&K^{(\beta)}_t\mathrm{V}_\beta=\mathrm{V}_\beta K^\phi_t \  &\text{ on } \ \bmrm{L}(\nu_\phi) \hlabel{eq:Kbeta-Kphi}
\end{eqnarray*}
with
\begin{align}\hlabel{eq:I_phi}
\mathrm{I}_\phi f(x)=\bb{E}[f(xI_\phi)]
\end{align}
 and, for all $k\in\ZZ_+$, $\bb{E}[I^k_\phi]=\frac{k!}{W_\phi(k+1)}$. $V_\beta$ is another multiplicative Markov kernel associated with the random variable $Y_\beta$ whose law is determined by its moment sequence given, for all $k\in\ZZ_+$,  by
\begin{align*}
\bb{E}[Y^k_\beta]=\Gamma(1+\beta)\frac{W_\phi(k+1)}{\Gamma(k+1+\beta)}.
\end{align*}
Finally, we have $\wi{\mathrm B}_\beta f(x)=\frac{x^\beta}{\Gamma(\beta)}\int_0^\infty f((1+y)x)y^{\beta-1}e^{-yx}dy,  x>0,$ and $\mathrm{I}_\phi\wi{\mathrm{B}}_\beta\mathrm{V}_\beta=K^\phi_{\tau^{(\beta)}}$. Now, from Proposition~\ref{prop:inter-bdsf}, we know that
\begin{align}
\eK^\phi_t\Mphi=\Mphi\eK^{\sigma^2}_t \hlabel{eq:inter_Mphi} \ \text{ on } \ \ell^2(\mathbf{n}_{\sigma^2})
\end{align}
where $\eK^{\sigma^2}=\eK^\phi$ with $\phi(u)=\sigma^2 u$. On the other hand, from Proposition~\ref{prop:int_laguerre}, it is known that
 \begin{align}
 K^\phi_t\Lambda=\Lambda\eK^\phi_t \text{ on } \ \ell^2(\meKp)
 \end{align}
 and from \cite[Proposition 25]{miclo_patie} along with \cite[Proposition 30]{miclo_patie_2} and Lemma~\ref{tau_int}, we have
 \begin{eqnarray*}
&\eK^{\sigma^2}_t\wi{\Lambda}_{\sigma^2}=\wi{\Lambda}_{\sigma^2} K^{\sigma^2}_t \ &\text{ on } \ \bmrm{L}(\nu_{\sigma^2}) \\
&K^{\sigma^2}_td_{\frac{1}{\sigma^2}}\wi{\mathrm{B}}_\beta=d_{\frac{1}{\sigma^2}}\wi{\mathrm{B}}_\beta K^{(\beta)}_t \ &\text{ on } \ \bmrm{L}(\nu_\beta) \\
&K^{(\beta)}_t\mathrm{V}_\beta=\mathrm{V}_\beta K^\phi_t \ &\text{ on } \ \bmrm{L}(\nu_\phi)
 \end{eqnarray*}
where $\nu_{\sigma^2}$ (resp.~$\nu_\beta$) equals $\nu_\phi$ (the invariant distribution of the semigroup $K^\phi$, see Proposition~\ref{prop:int_laguerre}\eqref{it:inv}) with $\phi(u)=\sigma^2 u$ (resp.~$\phi(u)=u+\beta$), $d_\alpha f(x)=f(\alpha x)$ is the dilation operator and $\wi{\Lambda}_{\sigma^2}:\bmrm{L}(\nu_{\sigma^2})\to\ell^2(\mathbf{n}_{\sigma^2})$ is a Markov operator defined by
\begin{align*}
\wi{\Lambda}_{\sigma^2}(n,dx)=\frac{\sigma^{2(n-1)}}{(1+\sigma^2)^n}\frac{x^n}{n+1}e^{-x\left(1+\sigma^{-2}\right)}dx. \hlabel{eq:K-Kphi}
\end{align*}
By transitivity of the intertwining relation, it follows that
\begin{eqnarray}
&\eK^\phi_t\Mphi\wi{\Lambda}_{\sigma^2}=\Mphi\wi{\Lambda}_{\sigma^2} K^{\sigma^2}_t & \text{ on } \bmrm{L}(\nu_{\sigma^2}) \hlabel{eq:interweave}\\
&K^{\sigma^2}_t\Upsilon=\Upsilon\eK^\phi_t & \text{ on } \ell^2(\meKp) \hlabel{eq:interweave1}
\end{eqnarray}
where $\Upsilon=d_{\frac{1}{\sigma^2}}\wi{\mathrm{B}}_\beta\mathrm{V}_\beta\Lambda$. Now, from \eqref{eq:interweave} and \eqref{eq:interweave1}, it remains to show that $\Mphi\wi{\Lambda}_{\sigma^2}\Upsilon=\eK^\phi_{\tau_\beta}$, where $\tau_\beta$ is defined as in the proposition.
\begin{lemma}\hlabel{lem:Mphi_int}
	The operator $\mathrm{I}_\phi$ in \eqref{eq:I_phi} commutes with the dilation operator $d$. Moreover, if $\sigma^2>0$ then $d_{\sigma^2}\mathrm{I}_\phi\Lambda=\Mphi\Lambda $ on $\ell^2(\meKp)$.
\end{lemma}
\begin{proof}
 Since $\mathrm{I}_\phi$ is a multiplicative Markov kernel, commutation with the dilation operator follows readily. Now, for the intertwining relationship, by density of $\poly=\Span\{\p_k; k\in\ZZ_+\}$ in $\ell^2(\meKp)$, it suffices to show that, for all $k\in\ZZ_+$,
	\[d_{\sigma^2}\mathrm{I}_\phi\Lambda\p_k=\Lambda\Mphi\p_k.\]
	However, $\Mphi\p_k=\frac{\sigma^{2k} k!}{W_{\phi}(k+1)}$ and $d_{\sigma^2}\mathrm{I}_\phi\Lambda \p_k=d_{\sigma^2}\mathrm{I}_\phi p_k=\frac{\sigma^{2k} k!}{W_\phi(k+1)}$, which proves the lemma.
\end{proof}
Coming back to the main proof, by an application of Lemma~\ref{tau_int} and Lemma~\ref{lem:Mphi_int}, we obtain
\begin{equation}\hlabel{eq:lambda-mphi}
\Lambda\Mphi\wi{\Lambda}_{\sigma^2}\Upsilon=d_{\sigma^2}\mathrm{I}_\phi\Lambda\wi{\Lambda}_{\sigma^2}\Upsilon=d_{\sigma^2}\mathrm{I}_\phi\Lambda\wi{\Lambda}_{\sigma^2} d_{\frac{1}{\sigma^2}}\wi{\mathrm{B}}_\beta\mathrm{V}_\beta\Lambda.
\end{equation}
Again invoking \cite[Proposition 25]{miclo_patie_2}, we have $\Lambda\wi{\Lambda}_{\sigma^2}=K^{\sigma^2}_{\log\left(1+\sigma^{-2}\right)}$. Writing $\varrho=\frac{1}{2}\log(1+\sigma^{-2})$ as before, \eqref{eq:lambda-mphi} yields
\begin{align*}
\Lambda\Mphi\Upsilon=&d_{\sigma^2}\mathrm{I}_\phi K^{\sigma^2}_{\gamma}d_{\frac{1}{\sigma^2}}\wi{\mathrm{B}}_\beta\mathrm{V}_\beta\Lambda \\
=&d_{\sigma^2}\mathrm{I}_\phi d_{\frac{1}{\sigma^2}}\wi{\mathrm{B}}_\beta K^{(\beta)}_{\gamma}\mathrm{V}_\beta\Lambda \\
=&\mathrm{I}_\phi\wi{\mathrm{B}}_\beta\mathrm{V}_\beta K^\phi_{2\varrho}\Lambda \\
=& K^\phi_{\tau^{(\beta)}}K^\phi_{2\varrho}\Lambda\\
=&\Lambda\eK^\phi_{\tau_\beta}
\end{align*}
where in the last line of the above equation, we used the fact that $\tau_\beta=\tau^{(\beta)}+2\varrho$. By injectivity of $\Lambda$, it follows that $\Mphi\Upsilon=\eK^\phi_{\tau_\beta}$. This proves the proposition. \qed
\subsection{Proof of Theorem~\ref{thm:spgap_ent}\eqref{thm:entropy}}\hlabel{ss:entropy_pf} We recall that $L^{\sigma^2}$ is the generator of the self-adjoint Laguerre diffusion defined in \eqref{eq:diff_lag_gen} whose invariant distribution is $\nu_{\sigma^2}(x)dx=\frac{1}{\sigma^2}\nu(x/\sigma^2)=\frac{1}{\sigma^2} e^{-x/\sigma^2}dx, \ x>0$. Let us first prove the $\Phi$-entropy decay for $K^{\sigma^2}$, the semigroup generated by $L^{\sigma^2}$, that is, for all admissible function $\Phi$ and $f\in\bmrm[1]{L}(\nu_{\sigma^2})$ with $\Phi(f)\in\bmrm[1]{L}(\nu_{\sigma^2})$ one has
\begin{align}\hlabel{eq:entropy}
	\Ent^\Phi_{\nu_{\sigma^2}}(K^{\sigma^2}_t f)\le e^{-t}\Ent^\Phi_{\nu_{\sigma^2}}(f).
\end{align}
 In Lemma~\ref{tau_int}, we have shown that the semigroups $K^{\sigma^2}$ and $K$ are equivalent via the similarity transform induced by the dilation operator $d_{\sigma^2}$. We first claim that it is enough to prove the exponential entropy decay in \eqref{eq:entropy} replacing $K^{\sigma^2}$ by $K$. To see why, we note that for any $\sigma^2>0$, $f\in\bmrm[1]{L}(\nu_{\sigma^2})$ with $\Phi(f)\in\bmrm[1]{L}(\nu_{\sigma^2})$, one has by the change of variable along with Lemma~\ref{tau_int} that,
 \begin{align*}
 	\int_0^\infty\Phi(K^{\sigma^2}_t f(x))\nu_{\sigma^2}(x) dx &=\int_0^\infty \frac{1}{\sigma^2} \Phi\left(K_t d_{\sigma^2} f\left(\frac{x}{\sigma^2}\right)\right)\nu\left(\frac{x}{\sigma^2}\right)dx \\
 	&=\int_0^\infty \Phi(K_t d_{\sigma^2} f(x))\nu(x) dx.
 \end{align*}
 We also observe by the change of vairable that for any $f\in\bmrm[1]{L}(\nu_{\sigma^2}$), one has  $\Phi(\nu_{\sigma^2} f)=\Phi(\nu d_{\sigma^2} f)$. As a result, we have
 \begin{align*}
 	\Ent^\Phi_{\nu_{\sigma^2}}(K^{\sigma^2}_t f)=\Ent^\Phi_{\nu}(K_t d_{\sigma^2} f)
 \end{align*}
which proves our claim. Next, we state the following result regarding the exponential entropy decay of the semigroup $K$ generated by $L$.
\begin{lemma}\hlabel{lem:ent_lag}
	For any $\Phi$ as above and $f\in\bmrm[1]{L}(\nu)$ with $\Phi(f)\in\bmrm[1]{L}(\nu)$, one has
	\begin{align*}
		\Ent^\Phi_\nu(K_tf)\le e^{-t}\Ent^\Phi_\nu(f).
	\end{align*}
\end{lemma}
\begin{proof}
	Since $L$ is a diffusion operator, from \cite[Equations (6) and (7)]{gentil}, it suffices to show that for an admissible function $\Phi$ and $f\in\bmrm[1]{L}(\nu)$ with $\Phi(f)\in\bmrm[1]{L}(\nu)$, one has the following $\Phi$-entropy inequality
	\begin{align}\hlabel{eq:CD}
		\ent^\Phi_{\mu}(f)\le \mu(\Phi''(f)\Gamma(f))
	\end{align}
where $\Gamma$ is the carr\'e-du-champ operator, see \cite[Section~1.4.2]{Bakry_Book} associated to $L$, that is, for smooth funcitons
\[\Gamma(f)=L(f^2)-2f Lf.\]
From \cite[Theorem~2.1(2)]{chafai} it follows that \eqref{eq:CD} is equivalent to the fact that the operator $L$ satisfies the curvature dimension condtion $CD(\frac{1}{2},\infty)$, which is indeed true from \cite[Section 2.7.3]{Bakry_Book}. Hence the lemma is proved.
\end{proof}
Now coming back to the proof of Theorem~\ref{thm:spgap_ent}\eqref{thm:entropy}, due to the interweaving relation in Theorem~\ref{thm:discPQ} and the estimate in \eqref{eq:entropy}, the proof of this theorem follows directly from \cite[Theorem~8]{miclo_patie_2}. \qed

\subsection{Proof of Theorem~\ref{thm:hyper}}\hlabel{ss:hyper_pf} For ergodic self-adjoint diffusion semigroups, we know from \cite[Theorem~5.2.3]{Bakry_Book} that the hypercontractivity can be interpreted in terms of the log-Sobolev constants corresponding to the semigroups. Let us consider the self-adjoint Laguerre semigroup $K^{\sigma^2}$ defined in Proposition~\ref{thm:discPQ}. For this semigroup, the invariant distribution is $\nu_{\sigma^2}(x)dx=\frac{1}{\sigma^2}e^{-x/\sigma^2}dx , \ x>0,$ and the log-Sobolev constant is
\begin{align}\hlabel{eq:log_sob}
	c_{LS}=\inf_{f\in\esC^1_b(\bb{R}_+):\|f\|_{\bmrm{L}(\nu_{\sigma^2})}=1}\frac{4\int_{\bb{R}_+}xf'(x)^2\nu_{\sigma^2}(dx)}{\int_{\bb{R}_+}f(x)^2\log(f(x)^2)\nu_{\sigma^2}(dx)}.
\end{align}
The numerator in the above expression is four times the Dirichlet energy associated to $L^{\sigma^2}$ defined by
\[\cc{E}(f,f)=-\langle L^{\sigma^2} f,f\rangle_{\nu}=\int_{\bb{R}_+}xf'(x)^2\nu_\r(dx).\] It was shown by Bakry \cite{Bakry} that $c_{LS}=1$. Hence, by applying \cite[Theorem~5.2.3]{Bakry_Book}, we infer that for all $t\ge 0$,
\[|\!|\!|K^{\sigma^2}_t|\!|\!|_{\bmrm{L}(\nu_{\!\sigma^2})\to\bmrm[p(t)]{L}(\nu_{\sigma^2})}\le 1\]
where $p(t)=1+e^t$ and $\nu(dx)=e^{-x} dx, \ x>0$. Having the above hypercontractivity estimate, the rest of the proof follows from \cite[Theorem~9]{miclo_patie_2} and Theorem~\ref{thm:discPQ}. \qed

\subsection{Proof of Theorem~\ref{thm:subordination}} \label{sec:thlast} First, we note that the semigroup $\eK^{\phi,\tau_\beta}$ has the same invariant distribution $\meKp$ as $\eK^\phi$. Let us recall that for $\sigma^2>0$, $\varrho=\frac{1}{2}\log(1+\sigma^{-2})$. If $t>\frac{1}{2}$, Theorem~\ref{thm:eig_coeig}\eqref{it:spectral_exp} entails that, for all $s>0$ and $\f\in\ell^2(\meKp)$, we have
\begin{align}\hlabel{eq:last}
\eK^\phi_{s+2{\varrho} t} \f=\sum_{k=0}^\infty e^{-2k{\varrho} t}e^{-ks}\langle \f,\eigLdd_k\rangle_{\meKp}\eigLd \ \text{ in } \ \ell^2(\meKp).
\end{align}
For $t\ge 0$, let us define the random variable $\widetilde{\tau}_\beta(t)$ such that $\tau_\beta(t)=2{\varrho} t+\widetilde{\tau}_\beta(t)$. Indeed, from \eqref{eq:tau_beta1} it follows that for all $t\ge 0$,
\begin{align*}
	\log\bb{E}\left[e^{-u\widetilde{\tau}_\beta(t)}\right]=-\log\left(\frac{\Gamma(u+\beta+1)}{\Gamma(1+\beta)\Gamma(u+1)}\right).
\end{align*}
Then, integrating both sides of \eqref{eq:last} with respect to $\bb{P}(\widetilde{\tau}_\beta(t)\in ds)$ with $t>\frac{1}{2}$ we obtain
\begin{align*}
	\eK^{\phi,\tau_\beta}_t \f=\int_0^\infty \eK^\phi_s\f\:\bb{P}(\tau_\beta(t)\in ds)&=\int_0^\infty \eK^\phi_{s+2{\varrho} t}\f\:\bb{P}(\widetilde{\tau}_\beta(t)\in ds)\\
	=&\int_0^\infty\left(\sum_{k=0}^\infty e^{-2k{\varrho} t} e^{-ks}\langle \f,\eigLdd_k\rangle_{\meKp}\eigLd\right)\bb{P}(\widetilde{\tau}_\beta(t)\in ds)\\
	=&\sum_{k=0}^\infty e^{-2k{\varrho} t} \langle \f,\eigLdd_k\rangle_{\meKp}\eigLd\int_0^\infty e^{-ks}\bb{P}(\widetilde{\tau}_\beta(t)\in ds)
\end{align*}
where the last equality follows due to Fubini theorem with the help of the estimates $\|\eigLd\|_{\ell^2(\meKp)}\le 1, \ \|\eigLdd_k\|_{\ell^2(\meKp)}\le  C_\ep e^{k({\varrho}+\ep)}$ for arbitrary $\ep>0$ and $k\in\ZZ_+$, see Proposition~\ref{prop:eigen} and the proof of Theorem~\ref{thm:eig_coeig}\eqref{it:spectral_exp}. The proof of this item is concluded by recalling that, for all $k\in\ZZ_+$,
\[e^{-2{\varrho} t}\int_0^\infty e^{-ks}\bb{P}(\widetilde{\tau}_\beta(t)\in ds)=e^{-t\phi_\beta(k)}.\]
For the next item, applying Jensen's inequality we observe that, for all $\beta, t>0$,
\begin{align*}
	\Ent^\Phi_{\meKp}\left(\eK^{\phi,\tau_\beta}_{t+1} \f\right)\le \int_0^\infty \Ent^\Phi_{\meKp}\left(\eK^\phi_{s+\tau_\beta}\f\right)\bb{P}(\tau_\beta(t)\in ds).
\end{align*}
Using Theorem~\ref{thm:entropy}, when $\sigma^2>0$ and $\beta>\mathrm{m}_\phi$,  and with the right-hand side of the above inequality is bounded above by
\[\int_0^\infty e^{-s}\Ent^\Phi_{\meKp}(\f)\bb{P}(\tau_\beta(t)\in ds)=e^{-t\phi_\beta(1)}\Ent^\Phi_{\meKp}(\f).\] 
This proves \eqref{it:sbo_ent}. Finally, with $\tau_\beta=\tau_\beta(1)$ chosen independent of $(\tau_\beta(t))_{t\geq0}$ which we recall is a subordinator, we have,  by the triangle inequality, that, for any $\f\in\ell^{2}(\meKp)$,
 \begin{align*}	\left\|\eK^{\phi,\tau_\beta}_{t+1}\f\right\|_{\ell^{p(\alpha t)}(\meKp)}&=\left\|\int_0^\infty\eK^{\phi}_{s+\tau_\beta}\f \ \bb{P}(\tau_\beta(t)\in ds)\right\|_{\ell^{p(\alpha t)}(\meKp)}\\
 	&=\left\|\int_0^\infty\eK^{\phi}_{s+2\varrho t+\tau_\beta}\f \ \bb{P}(\widetilde{\tau}_\beta(t)\in ds)\right\|_{\ell^{p(\alpha t)}(\meKp)} \\
 	& \le\int_0^\infty \left\|\eK^\phi_{s+2\varrho t+\tau_\beta}\f\right\|_{\ell^{p(2\varrho t)}(\meKp)}\bb{P}(\widetilde{\tau}_\beta(t)\in ds).
\end{align*}
Invoking Theorem~\ref{thm:hyper}, the right-hand side of the above inequality is bounded above by \[\|\f\|_{\ell^2(\meKp)}\int_0^{\infty}\bb{P}(\widetilde{\tau}_\beta(t)\in ds)=\|\f\|_{\ell^2(\meKp)}\] which completes the proof.

\bibliographystyle{plain}

%\bibliography{./bib_pp}

\begin{thebibliography}{10}

\bibitem{BGK}
F.~Achleitner, A.~Arnold, and E.~A. Carlen.
\newblock On multi-dimensional hypocoercive {BGK} models.
\newblock {\em Kinet. Relat. Models}, 11(4):953--1009, 2018.

\bibitem{Akhiezer-65}
N.I. Akhiezer.
\newblock {\em The Classical Moment Problem}.
\newblock Oliver and Boyd, 1965.

\bibitem{Assi}
T. Assiotis.
\newblock On a gateway between the Laguerre process and dynamics on partitions.
\newblock {\em Latin American Journal of Probability and Mathematical Statistics}, 2019.


\bibitem{Bakry}
D.~Bakry.
\newblock{Remarques sur les semigroupes de Jacobi}
\newblock{Ast\'erique}, (236):23-39, 1996. Hommage \`a P.A.~Meyer et J.~Neveu.

\bibitem{Bakry_Book}
D.~Bakry, I.~Gentil, and M.~Ledoux.
\newblock {\em Analysis and geometry of {M}arkov diffusion operators}, volume
  348 of {\em Grundlehren der Mathematischen Wissenschaften [Fundamental
  Principles of Mathematical Sciences]}.
\newblock Springer, Cham, 2014.


\bibitem{Baudoin}
F.~Baudoin.
\newblock Bakry--\'{E}mery meet {V}illani.
\newblock {\em J. Funct. Anal.}, 273(7):2275--2291, 2017.


\bibitem{bertoin_fc}
J.~Bertoin
\newblock{The asymptotic behavior of fragmentation processes}
\newblock{\em J. Europ. Math. Soc.}, 5:395-416, 2003.


\bibitem{bertoin_yor}
J. Bertoin and M. Yor,
\newblock{On the entire moments of self-similar {M}arkov processes and exponential functionals of {L}\'{e}vy processes}
\newblock{\em Ann. Fac. Sci. Toulouse Math. (6)}, 11:33--45, 2002.

\bibitem{bertoin_igor}
J.~Bertoin, I.~Kortchemski,
\newblock{Self-similar scaling limits of {M}arkov chains on the positive integers}
\newblock{Ann. Appl. Probab.}, 26:2556-2595, 2016.

\bibitem{borodin}
A.~Borodin and G.~Olshanski.
\newblock Markov dynamics on the {T}homa cone: a model of time-dependent
  determinantal processes with infinitely many particles.
\newblock {\em Electron. J. Probab.}, 18:no. 75, 43, 2013.


\bibitem{gentil}
F. Bolley, I.~Gentil.
\newblock{Phi-entropy inequalities and Fokker-Planck equations},
\newblock{\em Progress in analysis and its applications} 463-469, 2010.

\bibitem{CF}
G. Carinci, C. Franceschini, C. Giardinà, F. Redig, W. Groenevelt,
\newblock{Orthogonal dualities of
Markov processes and unitary symmetries.}
\newblock{\em SIGMA Symmetry Integrability Geom. Methods Appl.} 15, no. 53, 2019.

\bibitem{chafai}
D.~Chafai.
\newblock{Entropies, convexity, and functional inequalities},
\newblock{\em J. Math. Kyoto. Univ.}, 44-2(2004), 325-363

%\bibitem{Carmona-Petit-Yor-98}
%Ph. Carmona, F.~Petit, and M.~Yor.
%\newblock Beta-{gamma} random variables and intertwining relations between
  %certain {Markov} processes.
%\newblock {\em Rev. Mat. Iberoamericana}, 14(2):311--368, 1998.

%\bibitem{CG}
%R. Carroll and  J.E.~Gilbert.
%\newblock{Some remarks on transmutation, scattering theory, and special functions.}
%\newblock {\em Math. Ann.}, 258(1), 39--54, 1981.

%\bibitem{Chihara-68}
%T.~S. Chihara.
%\newblock Orthogonal polynomials with {B}renke type generating functions.
%\newblock {\em Duke Math. J.}, 35:505--517, 1968.

\bibitem{Christensen-03}
O.~Christensen.
\newblock {\em An Introduction to Frames and Riesz Bases}.
\newblock Birkhäuser, 2003.

%\bibitem{Craven-Csordas-89_Jensen_Turan}
%T.~Craven and G.~Csordas.
%\newblock Jensen polynomials and the {T}ur\'an and {L}aguerre inequalities.
%\newblock {\em Pacific J. Math.}, 136(2):241--260, 1989.

%\bibitem{Davies}
%E.~B. Davies.
%\newblock Pseudo-spectra, the harmonic oscillator and complex resonances.
%\newblock {\em R. Soc. Lond. Proc. Ser. A Math. Phys. Eng. Sci.},
 % 455(1982):585--599, 1999.

%\bibitem{Davies_S}
%E.~B. Davies.
%\newblock Non-self-adjoint differential operators.
%\newblock {\em Bull. London Math. Soc.}, 34(5):513--532, 2002.

%\bibitem{Davies-Kuijlaars}
%E.~B. Davies and A.~B.~J. Kuijlaars.
%\newblock Spectral asymptotics of the non-self-adjoint harmonic oscillator.
%\newblock {\em J. London Math. Soc. (2)}, 70(2):420--426, 2004.

%\bibitem{Delsarte-Lions-57}
%J.~Delsarte and J.L. Lions.
%\newblock Transmutations d'op\'erateurs diff\'erentielles dans le domaine
 % complexe.
%\newblock {\em Comment. Math. Helv.}, 52:113--128, 1957.







\bibitem{Mouhot}
J.~Dolbeault, C.~Mouhot, and C.~Schmeiser.
\newblock Hypocoercivity for linear kinetic equations conserving mass.
\newblock {\em Trans. Amer. Math. Soc.}, 367(6):3807--3828, 2015.


%\bibitem{GL_Fragm}
%M.~Doumic~Jauffret and P.~Gabriel.
%\newblock Eigenelements of a general aggregation-fragmentation model.
%\newblock {\em Math. Models Methods Appl. Sci.}, 20(5):757--783, 2010.

%\bibitem{Duffin-Schaeffer-52}
%R.~Duffin and A.~Schaeffer.
%\newblock A class of non-harmonic fourier series.
%\newblock {\em Trans. Amer. Math. Soc.}, 72:341--366, 1952.

%\bibitem{Dunford_II}
%N.~Dunford and J.~T. Schwartz.
%\newblock {\em Linear operators. {P}art {II}: {S}pectral theory. {S}elf adjoint
%  operators in {H}ilbert space}.
%\newblock With the assistance of William G. Bade and Robert G. Bartle.
%  Interscience Publishers John Wiley \& Sons\ New York-London, 1963.



\bibitem{dynkin:1965}
E.~B. Dynkin.
\newblock {\em Markov processes. {V}ols. {I}, {II}}, volume 122 of {\em
  Translated with the authorization and assistance of the author by J. Fabius,
  V. Greenberg, A. Maitra, G. Majone. Die Grundlehren der Mathematischen Wi
  ssenschaften, B\"ande 121}.
\newblock Academic Press Inc., Publishers, New York, 1965.

%\bibitem{Erdelyi-55}
%A.~Erd{\'e}lyi, W.~Magnus, F.~Oberhettinger, and F.G. Tricomi.
%\newblock {\em Higher Transcendental Functions}, volume~3.
%\newblock McGraw-Hill, New York-Toronto-London, 1955.

\bibitem{ethier-kurtz}
Ethier, Stewart N. and Kurtz, Thomas G.
\newblock{\em Markov processes: Characterization and convergence.}

\bibitem{Gr}
W. Groenevelt, Orthogonal stochastic duality functions from Lie algebra representations, \newblock{\em J. Stat. Phys.} 174, 97–119, 2019.
%

\bibitem{JK}
S. Jansen and N. Kurt, On the notion(s) of duality for Markov processes,
\newblock{\em Probability Surveys}, 11, 59 -- 120, 2014.

\bibitem{kallenberg-book}
O.~Kallenberg,
\newblock{\em Foundations of modern probability},
\newblock{second edition},
\newblock{Springer-Verlag, New York}, 2002.

\bibitem{Karlin-McG}
S.~Karlin, J.~McGregor
\newblock{Linear growth birth and death processes.}
\newblock{\em J. Math. Mech.}, 7:643-662, 1958

\bibitem{koekoek}
R.~Koekoek, P.A.~Lesky, R.F.~Swarttouw
\newblock{\em Hypergeometric orthogonal polynomials and their q-analogues}
\newblock{Springer-Verlag, Berlin}, 2010

%\bibitem{Kost}
%B.~Kostant.
%\newblock On {L}aguerre polynomials, {B}essel functions, {H}ankel transform and
%  a series in the unitary dual of the simply-connected covering group of {${\rm
%  Sl}(2,{\bf R})$}.
%\newblock {\em Represent. Theory}, 4:181--224 (electronic), 2000.

\bibitem{Lamperti-62}
J.~Lamperti.
\newblock Semi-stable stochastic processes.
\newblock {\em Trans. Amer. Math. Soc.}, 104:62--78, 1962.

\bibitem{Lamperti-72}
J.~Lamperti.
\newblock Semi-stable {M}arkov processes. {I}.
\newblock {\em Z. Wahrsch. Verw. Geb.}, 22:205--225, 1972.

%\bibitem{McKean-Spectral}
%H.~P. McKean, Jr.
%\newblock Elementary solutions for certain parabolic partial differential
%  equations.
%\newblock {\em Trans. Amer. Math. Soc.}, 82:519--548, 1956.



\bibitem{miermont}
J.F.~Le-Gall, G.~Miermont
\newblock{Scaling limits of random planar maps with large faces}
\newblock{\em Ann. Probab.} 39, no. 1, 1--69, 2011.

\bibitem{miclo_patie}
L.~Miclo, P.~Patie,
\newblock{On a gateway between continuous and discrete {B}essel and {L}aguerre processes.}
\newblock{\em Annales Henri Lebesgue}, 2:59--98, 2019.

\bibitem{miclo_patie_2}
L. Miclo  and P. Patie, On interweaving relations,  \newblock{\em J. Funct. Anal.}, 280, no. 3, 53pp., 2021.

%\bibitem{Misra-Lavoine}
%O.~P. Misra and J.~L. Lavoine.
%\newblock {\em Transform analysis of generalized functions}, volume 119 of {\em
%  North-Holland Mathematics Studies}.
%\newblock North-Holland Publishing Co., Amsterdam, 1986.
%\newblock Notas de Matem{\'a}tica [Mathematical Notes], 106.

%\bibitem{Olver_W}
%F.~W.~J. Olver.
%\newblock Whittaker functions with both parameters large: uniform
%  approximations in terms of parabolic cylinder functions.
%\newblock {\em Proc. Roy. Soc. Edinburgh Sect. A}, 86(3-4):213--234, 1980.

\bibitem{Paris01}
R.~B. Paris and D.~Kaminski.
\newblock {\em Asymptotics and {M}ellin-{B}arnes integrals}, volume~85 of {\em
  Encyclopedia of Mathematics and its Applications}.
\newblock Cambridge University Press, Cambridge, 2001.

\bibitem{Patie-Savov-GeL}
P.\ Patie and M.\ Savov.
Spectral expansion of non-self-adjoint generalized Laguerre semigroups,
\newblock {\em Mem.\ Amer.\ Math.\ Soc.}, 272, no. 1336, vii+182 pp, 2021.

\bibitem{patie2018}
P.~Patie and M.~Savov.
\newblock{Bernstein-gamma functions and exponential functionals of L\'evy processes},
\newblock{\em Electron. J. Probab.}, vol. 23, p. 101 pp., 2018.

\bibitem{PV-Hypo}
P.~Patie and A.~Vaidyanathan,  A spectral theoretical approach for
  hypocoercivity applied to some degenerate hypoelliptic, and non-local
  operators,  {\em Kinetic and Related Models}, 13(3): 479-506, 2020.

\bibitem{Pitman-Rogers-81}
J.~Pitman and L.G. Rogers.
\newblock Markov functions.
\newblock {\em Ann. Probab.}, 9:573--582, 1981.


\bibitem{Redig}
F.~Redig and F.~Sau. \newblock Factorized duality, stationary product measures and generating functions.
\newblock {\em J. Stat. Phys.} 172(4): 980-1008, 2018.


%\bibitem{Rooney-Muck}
%P.~G. Rooney.
%\newblock Multipliers for the {M}ellin transformation.
%\newblock {\em Canad. Math. Bull.}, 25(3):257--262, 1982.

%\bibitem{Sato-99}
%K.~Sato.
%\newblock {\em L\'evy Processes and Infinitely Divisible Distributions}.
%\newblock Cambridge University Press, Cambridge, 1999.

%\bibitem{Schilling_pos}
%R.~L. Schilling.
%\newblock Conservativeness and extensions of {F}eller semigroups.
%\newblock {\em Positivity}, 2(3):239--256, 1998.

%\bibitem{Schilling_SF}
%R.~L. Schilling and J.~Wang.
%\newblock Strong {F}eller continuity of {F}eller processes and semigroups.
%\newblock {\em Infin. Dimens. Anal. Quantum Probab. Relat. Top.},
%  15(2):1250010, 28, 2012.

%\bibitem{Sjostrand-Survey}
%J.~Sj{\"o}strand.
%\newblock Some results on nonselfadjoint operators: a survey.
%\newblock In {\em Further progress in analysis}, pages 45--74. World Sci.
%  Publ., Hackensack, NJ, 2009.

\bibitem{Stanley}
R.P.~Stanley.
\newblock{\em Enumerative Combinatorics: Vol 1, Second edition}

%\bibitem{Stein}
%Elias~M. Stein.
%\newblock {\em Topics in harmonic analysis related to the {L}ittlewood-{P}aley
%  theory.}
%\newblock Annals of Mathematics Studies, No. 63. Princeton University Press,
%  Princeton, N.J.; University of Tokyo Press, Tokyo, 1970.

%\bibitem{Szego}
%G.~Szego.
%\newblock {\em Orthogonal polynomials}.
%\newblock American Mathematical Society, Providence, R.I., fourth edition,
%  1975.
%\newblock American Mathematical Society, Colloquium Publications, Vol. XXIII.

%\bibitem{Temme}
%N.~M. Temme.
%\newblock Asymptotic estimates for {L}aguerre polynomials.
%\newblock {\em Z. Angew. Math. Phys.}, 41(1):114--126, 1990.

%\bibitem{Sturm}
%A. Sturm, J.M. Swart, and F. V¨ollering. The algebraic approach to duality: an
%introduction. Pages 81–150 in: M. Birkner (ed.) et al., Genealogies of Interacting
%Particle Systems. World Scientific. Lect. Notes Ser., Inst. Math. Sci., Natl. Univ.
%Singap. 38, 2020

\bibitem{Titchmarsh39}
E.C. Titchmarsh.
\newblock {\em The theory of functions}.
\newblock Oxford University Press, Oxford, 1939.

\bibitem{Villani-09}
C.~Villani.
\newblock Hypocoercivity.
\newblock {\em Mem. Amer. Math. Soc.}, 202(950):iv+141, 2009.


\bibitem{Young}
Robert~M. Young.
\newblock {\em An introduction to nonharmonic {F}ourier series}.
\newblock Academic Press, Inc., San Diego, CA, first edition, 2001.

%\bibitem{GL_Biology}
%R.~Yvinec, C.~Zhuge, J.~Lei, and M.~C. Mackey.
%\newblock Adiabatic reduction of a model of stochastic gene expression with
%  jump {M}arkov process.
%\newblock {\em J. Math. Biol.}, 68(5):1051--1070, 2014.

%\bibitem{Zet}
%N.~Zettili.
%\newblock {\em Quantum mechanics: concepts and applications}.
%\newblock John Wiley \& Sons, Ltd., Chichester, second edition, 2009.

\end{thebibliography}
\end{document}